\documentclass[12pt,reqno]{amsart} 

\allowdisplaybreaks
\calclayout

\usepackage{amssymb}
\usepackage{amsmath}
\usepackage{url}
\usepackage{csquotes}
\usepackage{tikz}
\usepackage{etoolbox}
\usepackage{mathscinet}
\usepackage{amsthm}  
\usepackage{mathtools}              
\usepackage[english]{babel}
\usepackage{graphicx}               
\usepackage{enumitem}               
\usepackage{mathrsfs}               
\usepackage{xcolor} 

\newtheorem{theorem}{Theorem}[section]
\newtheorem*{theoremA}{Theorem\,A}
\newtheorem*{corollaryA}{Corollary\,A}
\newtheorem*{corollaryB}{Corollary\,B}
\newtheorem*{theoremB}{Theorem\,B}
\newtheorem*{theoremC}{Theorem\,C}
\newtheorem*{corollaryC}{Corollary\,C}

\newtheorem{lemma}{Lemma}[section]
\newtheorem{prop}{Proposition}[section]

\theoremstyle{definition}
\newtheorem{defn}{Definition}[section]
\newtheorem{rem}{Remark}[section]
\newtheorem{eg}{Example}[section]

\def\e{{\rm e}}
\def\disp{{\displaystyle}}

\def\lro{{\lambda^{\rho}}}
\def\a{{ a}}
\def\:{{\colon\!}}
\def\loc{{\rm loc}}
\newcommand{\be}{\begin{equation}}
\newcommand{\ee}{\end{equation}}
\newcommand{\ba}{\begin{align}}
\newcommand{\bs}{\begin{align*}}
\newcommand{\ea}{\end{align}}
\newcommand{\es}{\end{align*}}
\newcommand{\bi}{\begin{itemize}}
\newcommand{\ei}{\end{itemize}}

\newcommand{\CN}{{\mathcal{N}}}
\newcommand{\K}{{\mathcal{K}}}
\newcommand{\CY}{{\mathcal{Y}}}

\newcommand{\pF}{p^{*}}
\newcommand{\T}{\Theta}
\newcommand{\tmn}{t^{-\frac{n}{2}}}
\newcommand{\tn}{t^{\frac{n}{2}}}

\newcommand{\A}{{\alpha}}

\newcommand{\eps}{{\varepsilon}}

\newcommand{\LT}{L^{\T}}
\newcommand{\LO}{L_0^{\Phi}}
\newcommand{\Mphi }{M^{\Phi}}
\newcommand{\Lphi}{L^{\Phi}}
\newcommand{\Linf}{L^{\infty}}
\newcommand{\Lpsi}{L^{\Psi}}

\newcommand{\vphi}{{\varphi}}

\newcommand{\Psinv}{{\Psi^{-1}}}
\newcommand{\Phinv}{{\Phi^{-1}}}

\newcommand{\Tinv}{{\T^{-1}}}
\newcommand{\G}{{\gamma}}
\newcommand{\lam}{{\lambda}}

\newcommand{\N}{\mathbb{N }} 
\newcommand{\R}{\mathbb{R}}

\newcommand{\It}{\int_0^t}
\newcommand{\Is}{\int_{\R^n }}
\newcommand{\F}{{\mathscr F}}
\newcommand{\U}{{\mathscr U}}
\newcommand{\C}{{\mathscr C}}

\newcommand{\Lap }{\Delta}

\newcommand\dee{\, {\rm d}}
\numberwithin{equation}{section}

\title[Nonlinear  Heat Equations]{Well-posedness of   Heat Equations with 
Nonlinearities of Arbitrarily Rapid Growth}

\author{Yohei Fujishima}
\address{Department of Mathematical and Systems Engineering \\ Faculty of Engineering \\ Shizuoka University \\  3-5-1\\  Johoku\\ Hamamatsu\\ 432-8561\\Japan}
\email{fujishima@shizuoka.ac.jp}

\author{Kotaro Hisa}
\address{Department of Applied Mathematics\\  Fukuoka University\\ 8-19-1\\Nanakuma\\Jonan-ku\\ Fukuoka City\\ 
            814-0180\\Japan}
\email{hisak@fukuoka-u.ac.jp}

\author[Robert Laister]{Robert Laister$^{\dagger}$}
\address{School of Computing  and Creative Technologies\\  University of the West of England\\ 
            Frenchay Campus\\ Coldharbour Lane\\ Bristol  BS16 1QY\\UK}
\email{Robert.Laister@uwe.ac.uk}
\thanks{$^{\dagger}$Corresponding author}

\begin{document}



\begin{abstract}
We address   local- and global-in-time  well-posedness of the Cauchy problem for nonlinear  heat equations without imposing  growth rate restrictions on the nonlinearity {\em a priori}. Our results constitute a non-trivial expansion of the  classical $L^q$-theory  for nonlinearities  dominated by  polynomial growth and the exponential-Orlicz space theory  for nonlinearities of exponential growth, to one dealing with nonlinearities of arbitrarily large growth rate.  A key ingredient is a new  smoothing estimate for the action of the heat semigroup between two arbitrary Orlicz spaces, and in particular into $\Linf$. For nonlinearities growing at least exponentially we are able to identify  explicitly a critical space for  local well-posedness and, for small initial data, global well-posedness.  

\vspace{2mm}
\end{abstract}
\maketitle



\tableofcontents

\section{Introduction}
We consider the fundamental question of  local- and global-in-time well-posedness of the Cauchy problem for nonlinear  heat equations of the form
\begin{equation}\tag{NHE}\label{eq:she}
u_t=\Lap  u+f(u),\qquad u(0)=\vphi,
\end{equation}
where  $f\:\R\!\to\!\R$ is   locally Lipschitz  continuous   with $f(0)=0$ and  $\vphi$ is a (possibly unbounded) element of some Banach space $X$ of real-valued functions on $\R^n$. In particular we are especially  interested  in   nonlinearities $f$ which grow rapidly at infinity. 
By local well-posedness
we mean the existence, uniqueness and continuous dependence of  a solution $u\in C([0,T_{\vphi}),X)$  for every $\vphi\in X$. Under such circumstances one may then define a  local solution semiflow $\U$  on a suitable subdomain of $X\times [0,\infty)$.  Global-in-time well-posedness is defined similarly, but need only hold for $\vphi$ in a small open ball in $X$.  Such problems lie at the core of  modern theories of nonlinear evolution  equations and are intimately related to questions of  singularity formation and global continuation  of solutions, with many connections to  differential geometry,  functional analysis, probability theory and mathematical physics. 

The vast majority of  work  on the qualitative behaviour of nonlinear heat equations  has been carried out with nonlinearities dominated by polynomial growth, very often  with $f$ being a pure power law. Far  fewer papers have investigated  nonlinearities with stronger growth and 
those which do are usually for nonlinearities of exponential type. This is especially true of the literature on well-posedness. Indeed, to the best of our knowledge there exists no  explicit  local well-posedness theory for  nonlinearities   growing faster than $\exp (u^p)$, other than the trivial case of bounded initial data where $X=\Linf (\R^n)$.  
Our goal here is to go beyond previous considerations of polynomial and exponential growth  and establish   a well-posedness theory of  \eqref{eq:she} for a broad class of nonlinearities \emph{including those of arbitrarily large growth rate}.

The central problem is easy to state: given a nonlinearity $f$ as above, does there exist a Banach space $X$ (other than $\Linf (\R^n)$), such that \eqref{eq:she} is locally  well-posed in $X$? In such a case one can then ask  the deeper structural question,  
 one which pervades many contemporary theories of nonlinear partial differential equations:    is there a threshold (critical) choice, $X_c$, of  $X$ separating
 well-posedness from ill-posedness and if so,   in what way does $X_c$ depend on $f$?  One may then go on to consider whether \eqref{eq:she}  is also globally well-posed for  small initial data in $X$ or $X_c$. 

The central problem is less easy to solve, especially for nonlinearities of rapid growth.  Since the  work of \cite{W79}, 
satisfactory resolutions have only been obtained for two growth classes  of nonlinearity, namely those dominated by polynomial growth at infinity (\cite{BC,W80}) and those dominated by  $\exp (u^p)$  \cite{IRT1,MT2}.   One standard approach is to consider  the integral formulation  of  \eqref{eq:she}, namely
\be\label{eq:introVoC} u(t)=S(t)\vphi+\int_0^t S(t-s) f(u(s))\dee s\ee
where  $S(t)=\e^{t\Lap}$ is the heat semigroup,   and seek fixed points   $u\in C([0,T],X)$ of the nonlinear operator on the right hand side of \eqref{eq:introVoC} via the contraction mapping theorem.  
For singular, unbounded initial data it is known that the solvability of this problem depends critically, and often delicately, on the interaction between the  growth of $f(u)$ as $|u|\to\infty$ and the degree of smoothing of  $S(t)$ acting on $X$ or some related space. 

Previous  work (see e.g.,  \cite{BC,Gig86,IRT1,W79,W80}) has relied on stipulating either the space $X$ (such as Lebesgue space or exponential Orlicz space),   
a dominating  growth restriction  on  $f$ (such as polynomial or exponential) or the  decay rate of the heat semigroup $S(t)$.  
In \cite[Theorem 2]{W80} for example, the   assumption of a power law decay rate for $S(t)$ of the form $t^{-\A}$ led ultimately to the  restriction to polynomially dominated nonlinearities and spaces of power law integrability (Lebesgue spaces) in \cite[Theorem 1, Theorem 3 and  Theorem 4]{W80}. In \cite{IRT1,MT2} the growth of $f$ at infinity is assumed to be dominated by $\exp (u^p)$ ($p>1$), which led to a fixed point problem in a subspace of the exponential Orlicz space $\exp (L^p)$. In principle there should be no limitation to this approach:   for $f$ of a particular growth rate, guess a suitable  space $X$ and derive a smoothing estimate for $S(t)$ acting on  $X$ such that a contraction mapping results.  However, no such result exists in the literature for $f$ growing faster than exponentially (i.e., $\exp (u^p)$).  The problem of course lies in the reliability of choosing $X$ correctly and the difficulty in obtaining the desired  smoothing estimate. This  is   the main obstacle we overcome in this paper.

 We adopt a different approach and refrain from imposing any growth restrictions on either the space $X$, the nonlinearity $f$ or the smoothing decay rate of the heat semigroup $S(t)$.  Our choice of Banach space for this
purpose is very natural, namely the  family of Orlicz spaces $\Lphi(\R^n)$ parameterised by convex Young's functions $\Phi$. 
The central problem is then one of determining a suitable $\Phi$ for well-posedness of \eqref{eq:she} for a  given  $f$.
 In Theorem\,B we establish  a sufficient condition  for the local well-posedness of \eqref{eq:she}  in the form of  the convergence of an  integral involving $f$ and $\Phi$ (see \eqref{eq:introIC}). Likewise 
in Corollary\,B we provide an analogous condition for global well-posedness  (see \eqref{eq:IntroGlobalIC}).
The integral conditions  act like implicit  compatibility conditions between the nonlinearity $f$ and the Young's function $\Phi$. 
Thus, for a given $f$ one may check the integral condition to see whether a particular choice of $\Phi$  is suitable for the well-posedness of \eqref{eq:she};  likewise  for $f$, given $\Phi$.  With some additional  structural conditions imposed on $f$ (assumptions {\bf (C)} and  $\boldsymbol{(\Lambda)}$  and a minimal rate of growth (assumption {\bf (G)}) - see  Section\,\ref{sec:thmC} - we show  in Theorem\,C   how, given $f$,  one can choose $\Phi$ explicitly to obtain local well-posedness  of  \eqref{eq:she}. Moreover we identify a critical choice $\Phi_c$ of $\Phi$ separating local well-posedness from nonexistence of positive solutions and  show that $\Phi_c$   grows essentially  like $f$ at infinity. In Corollary\,C we provide an explicit  condition on $\Phi$ for global well-posedness of \eqref{eq:she} with small initial data. Thus we are able to address each aspect of the central problem.

 In the very  general set-up here with no specified growth rate  on either   $\Phi$ or $f$, the main technical obstacle we face and perhaps the one which has hindered progress in the field thus far, is in obtaining  an explicit and  sharp smoothing  estimate for the action of the heat semigroup between  arbitrary  Orlicz spaces.
 In Theorem\,A we obtain such an estimate and, in particular, in Corollary\,A we deduce the smoothing estimate  into  $L^\infty(\R^n)$. 
The $L^\infty$ smoothing estimate is crucial for obtaining  both Theorem\,B and Theorem\,C.  
In fact we will need to work in a particular subspace of $L^{\Phi}$, namely the Morse-Transue space $\Mphi(\R^n)$, in order that  $S(t)$ be a $C_0$-semigroup.
Such a space was introduced in \cite{RT1} and used by \cite{IRT1} to establish local well-posedness in the special case where $f$ and $\Phi$ grow like   $\exp (u^2)$. 

  In Theorem\,\ref{thm:J}  we provide a large canonical  family of  odd nonlinearities to which our results apply.  Of  particular importance is the fact that this family contains  nonlinearities  of  arbitrarily rapid  growth (Lemma\,\ref{lem:rapid}).

\subsection{Overview of the Main Results}\label{sec:IntroResults}
 Recall (see Definition~\ref{def:Orlicz}) that the  Orlicz space $\Lphi=\Lphi(\R^n)$ consists of Lebesgue measurable functions $\vphi$ for which $\Phi (k|\vphi|)$ is  integrable over $\R^n$ for {\em some} $k>0$, while the Morse-Transue subspace $\Mphi=\Mphi(\R^n)$ requires  integrability of $\Phi (k|\vphi|)$ for {\em all} $k>0$, where $\Phi\: [0,\infty]\!\to\! [0,\infty]$ is some  nondecreasing, convex Young's function with $\Phi(0)=0$ (see Definition~\ref{def:Young}). For $\vphi\in\Lphi$, the behaviour of $\Phi$ near zero imposes restrictions on the decay of  $\vphi$ at spatial infinity  while the growth of $\Phi$ at infinity  imposes restrictions on the strength of any local singularities of  $\vphi$.   Orlicz space is the natural choice of  Banach space  in which to generalise the $L^q$-theory of \cite{BC,W79,W80}, with
 $L^q$  corresponding to the special case $\Phi(u)\!=\!u^q/q$.

For expositional convenience  we  refer to any polynomially dominated nonlinearity  satisfying $|f(u)|\le C |u|^p$ for some $p>1$ and $u$ large enough, as  \emph{P-type} and any $f$ satisfying 
   $  |u|^p/ |f(u)|\to 0$ as $ |u|\to\infty$ for all $p>1$, as \emph{FP-type}   (faster than polynomial). 
  Currently the only general theory   of local well-posedness for \eqref{eq:she} in Banach spaces for a broad  class of nonlinearities is that of P-type \cite{BC,LRSV,LS20,QS19,W79,W80} and   FP-type  with growth dominated by  $\exp (u^p)$ \cite{IRT1,MT2}.  
  The corresponding  problem for global existence with small initial data has received more attention \cite{FKRT,MT1,MT2}.

In any  general framework seeking a  local solution semiflow $\U$, one  would like  $S(t)$ to be a $C_0$-semigroup on the underlying   space $X$ so that a  continuous solution trajectory   $u\in C([0,T_{\vphi}),X)$ might be obtained for every $\vphi\in X$. However,  for Orlicz spaces $\Lphi$ this is generally  not the case since the set of smooth compactly supported functions $C_0^{\infty}(\R^n)$ are  not always dense within them (see e.g., \cite[Section 2]{IRT1} where $\Phi(u)=\exp (u^2)-1$). Hence one cannot hope to obtain a  general well-posedness theory   requiring  $u\in C([0,T),X)$ and without {\em a priori} growth restrictions on $f$, by taking  $X=\Lphi$. Under more stringent  growth  restrictions on  $f$ such as those of P-type (e.g., $f(u)=|u|^{p-1}u$), it is well-known (\cite{BC,QS19,W79,W80}) that one {\em can}  obtain a well-posedness theory in Lebesgue space $\Lphi =L^q$ due to the fact that $C_0^{\infty}$ {\em is} dense in $L^q$ for $1\le q<\infty$; but without such restrictions a  general well-posedness theory  in  $\Lphi$ is not possible.

The obvious strategy therefore  is to work in $X= \LO$ (as in e.g.,  \cite{IRT2,MT2}), where $\LO$ denotes the closure of  $C_0^{\infty}$ in $\Lphi$ (Definition~\ref{def:Orlicz}(b)). Since we also want to allow singular (unbounded) initial data we do not want to restrict $\Lphi$ to being a subset of $\Linf$ and therefore (see  Remark~\ref{rem:subspace}(a)) we want $\Phi$ to be finite-valued. In such case (see Lemma~\ref{lem:L0M}(i)) $\LO\subseteq\Mphi$ and so working with initial data in $\LO$ necessarily constrains those data to lie in $\Mphi$.  In fact, under the  mild restriction that $\Phi$ be an $N$-function  (a finite, positive  Young's function satisfying $\Phi(u)/u\to 0$ as $u\to 0$ and  $\Phi(u)/u\to \infty$ as $u\to \infty$ -  see Definition\,\ref{def:Nfn} and \cite[Section~8]{AF}) there is  a convenient  characterisation  of $\LO$ available, namely $\LO=\Mphi$ (see Lemma~\ref{lem:L0M}(ii)).  Since this restriction on $\Phi$ is so mild and the characterisation  of   $\LO$ via $\Mphi$ so convenient,  we  work mainly with initial data in $M^{\Phi}$ throughout this paper.

We now describe our first key results concerning the heat semigroup. Their precise statements, together with subsidiary results, are given in Section\,\ref{sec:thmA}. By `compatible Young's functions' we mean  $\Phi$ and $\Psi$ satisfying \eqref{eq:cc}.

\vspace{2mm}

\bi [leftmargin=*]
\item[] 
{\bf Theorem\,A and Corollary\,A (outline).} \emph{For compatible Young's functions $\Phi$ and $\Psi$,  the heat semigroup  $S(t)\: L^{\Phi}\!\to\!L^{\Psi}$ satisfies
\[  \left\|S(t)\vphi \right\|_{\Psi}\le  \C\left\|\vphi \right\|_{\Phi}  \frac{{\Phinv(t^{-\frac{n}{2}})}}{\Psinv(t^{-\frac{n}{2}})}, \qquad  t>0.
\]
In particular, if $\Phi$ is an $N$-function  then  $S(t)$ is a $C_0$-semigroup on $M^{\Phi}$  and for all $t>0$
we have the $M^{\Phi}$-$L^{\infty}$ smoothing estimate
\[\left\|S(t)\vphi \right\|_{\infty}\le \C \left\|\vphi \right\|_{\Phi} \Phinv\left( t^{-\frac{n}{2}}\right),\qquad \vphi\in M^{\Phi}.\]
}
\ei
\vspace{2mm}

Knowing explicitly the dependence on $\Phi$ of the $\Linf$-smoothing estimate  enables us to determine a sufficient condition  on $\Phi$ and $f$
for  the  application of the contraction mapping theorem to the integral formulation \eqref{eq:introVoC}, 
yielding the existence of a classical solution.  This estimate  is also key to  obtaining unconditional uniqueness and continuous dependence of classical solutions, akin to those of the $L^q$ case for polynomially dominated nonlinearities in \cite{BC}, which we prove without further recourse to the contraction mapping theorem.
 Additionally, for monotonically increasing $f$ we establish  a comparison principle for the solution semiflow  which is of interest in its own right in the context of monotone dynamical systems theory.  With little extra effort,   we  then establish a sufficient condition for global-in-time well-posedness of  \eqref{eq:she} for small initial data in 
 $\Mphi$.  This latter topic has been of some interest in the literature concerned with exponential  nonlinearities, especially in its relation to nonlinear Schr\"odinger equations and the Trudinger-Moser embedding \cite{I1, MT1,MT2, NO,RT1}.  This appears to be  the source of  interest in nonlinearities and spaces with growth like $\exp(u^2)$.
 
We next outline our first main  well-posedness results,   their precise statements are given in Section\,\ref{sec:thmB}.

\vspace{2mm}
\bi [leftmargin=*]
\item[]
{\bf Theorem\,B and Corollary\,B (outline).}  \emph{Let $\Phi$ be an  $N$-function. If there exists    $\lam >0$    such that
\be\label{eq:introIC}
\int_1^{\infty} x^{-\pF}\ell\left(  \lam    \Phinv\left(x   \right)\right)\dee x<\infty, 
\ee
where $\pF\!:=\!1+2/n$ is the Fujita exponent  and $\ell$ is the Lipschitz modulus of  $f$, then for every  $\vphi\in M^{\Phi}$ \eqref{eq:she} possesses a unique   classical solution $u$ 
depending continuously on $\vphi$.  For $f$  increasing  the solution  is order-preserving in $\vphi$,  so that a comparison principle holds.
Subject to the stronger condition
\be\label{eq:IntroGlobalIC}
\int_0^{\infty} x^{-\pF}\ell\left(  \lam    \Phinv\left(x   \right)\right)\dee x<\frac{n}{4},
 \ee
for every   $\vphi$ in a sufficiently small ball in $M^{\Phi}$ there exists  a  unique global-in-time     classical solution $u$
of \eqref{eq:she} satisfying 
$u(t)\to 0$ in $L^\infty$ as $t\to\infty$.}
\ei

\vspace{2mm}

A common feature of semilinear heat equations is that  local existence of solutions is
governed only by the growth of the nonlinearity $f$ at infinity \cite{FI1,IRT1,LRSV,W80}, while global existence also 
imposes restrictions on $f$ near zero \cite{I1,LS21,MT2,MT1,W81}. We see this reflected in the integral conditions
\eqref{eq:introIC} and \eqref{eq:IntroGlobalIC} respectively, which arise naturally when estimating the integral term in \eqref{eq:introVoC} via  smoothing estimates for the heat semigroup (such as those in Theorem\,A and Corollary\,A) in order to obtain a contraction mapping.

The power of Theorem\,B  lies in its flexibility and the fact  that no  particular structural properties or limitation on growth rates  of  $f$ or $\Phi$ are imposed \emph{a priori}.  Instead, \eqref{eq:introIC} represents a kind of `compatibility condition for existence' between the nonlinearity  $f$ and the underlying space  $\Mphi$ specified by $\Phi$.  It generalises the one obtained for P-type nonlinearities in \cite[Theorem\,4.4]{LRSV}
 in the special case $L^{\Phi} =L^1$.
 It also bears a formal resemblance to \cite[Eqn.\,(2.2)]{W80} (indeed there is a clear lineage to that work)  but there are two crucial qualitative distinctions. Firstly, the definition of $\ell$ in \cite{W80} is for the Nemytskii operator associated with $f$, acting on some unspecified function space for which one has no explicit representation. Secondly,  the smoothing estimate in \cite{W80} is assumed to be of power law type and therefore,   as illuminated by  Theorem\,A(b), effectively limits $\Phi$ to one of power law type (i.e., Lebesgue spaces) and consequently $f$ to be of  P-type.

We  use \eqref{eq:introIC}  in  Theorem\,B  to establish, for a given $f$, how one should choose $\Phi$ in order to obtain local well-posedness of \eqref{eq:she}.
To achieve this we  impose some  structural conditions on $f$ such as convexity for $u\ge 0$ and  regular or rapid variation of $\log f$ at infinity. We also assume that $f$ satisfies
 \be\label{eq:FIlim}
  \lim_{u\to\infty}{ f'(u)F(u)}=1,\qquad F(u):= \int_{u}^\infty {1}/{f(s)}\dee s
  \ee
(see Section\,\ref{sec:background}    for additional context regarding \eqref{eq:FIlim}). For the sake of brevity at this point we refer the reader to Section\,\ref{sec:thmC} for the full details of these  conditions in  {\bf (C)},  $\boldsymbol{(\Lambda)}$   and  {\bf (G)}. Below we aim to summarise the essence of these assumptions and their consequences for \eqref{eq:she}.

Consider for the moment the  class of nonlinearities   
of the form  
\be\label{eq:expfintro}f(u)=\e^{u^\rho \mu (u)},\ee
where $\rho > 0$  is a constant and $\mu$ is a \emph{slowly varying} function (see \cite{BGT,Sen} and  the discussion around \eqref{eq:kara1}).  In the terminology of  {regularly varying functions},  if $f$ satisfies \eqref{eq:expfintro} then $\log f$ is said to be \emph{regularly varying of index} $\rho$.  
For $\rho\in (0,1)$  the  growth of $f$ is commonly referred to as `sub-exponential', $\rho =1$ as exponential \cite{PV} and $\rho >1$ as `super-exponential' (the nonlinearities in \cite{ I1, IRT1,MT1,MT2, NO,RT1} are all of this latter type with $\mu\equiv 1$).

In order to treat nonlinearities of arbitrarily large growth rate, we wish  to formally permit $\rho =\infty$.   
We do this rigorously by allowing $\log f$ to be a   \emph{rapidly varying} function (see \cite{Sen} and around \eqref{eq:rhoinf}), for which 
there is no  upper bound on the growth rate of $f$.  

With this in mind we outline our third main theorem.  We write $g\lesssim h$ whenever  $g(x)\le h(Kx)$ for   large $x$ and some constant $K>0$ and   $g\gtrsim h$ whenever $h\lesssim g$.   Likewise  `$g\lesssim h$ as $x\to 0$'  (resp. on $[0,\infty )$), etc. if the relations hold for small $x$ (resp. all $x\ge 0$).

\vspace{2mm}

\bi[leftmargin=*]
\item[]
{\bf Theorem\,C and Corollary\,C (outline).}  \emph{Let $\Phi$ be an  $N$-function and    $f$  satisfy   {\bf  (C)} and {\bf (G)}. 
\bi\item[(a)]   If $f$ satisfies $\boldsymbol{(\Lambda)}$
and
$\Phi\gtrsim f$ at infinity then \eqref{eq:she}  is locally well-posed in $M^{\Phi}$. If in addition   $\Phi\gtrsim  f^r$  near zero, for some $r\ge 1$ and $r>n/2$, then uniqueness of the classical solution holds in the  class  $C([0,T],M^{\Phi})$ of  mild solutions. Conversely, if $\Phi\lesssim f$  then there exist nonnegative initial data in $ L^{\Phi}$ for which no nonnegative solution of \eqref{eq:she}  exists. 
\item[(b)]  Suppose    $f$ satisfies $\boldsymbol{(\Lambda_0)}$ and that  $f'$ is regularly varying at zero of index $m-1$  for some $m>\pF$. 
If  $\Phi\in\mathcal{N}$ satisfies $\Phi\gtrsim f^q$ near zero  for some $0<q<n(m-1)/2$ and $\Phi\gtrsim f$ at infinity,     then \eqref{eq:she}  is locally well-posed in $M^{\Phi}$ and  globally  well-posed   for small initial data in $\Mphi$.
\ei
}
\ei

\vspace{2mm}

Theorem\,C(a) resolves two parts of the central problem. Firstly, for a given $f$ growing at least sub-exponentially but otherwise without dominating growth restrictions, we obtain $\Phi\gtrsim f$ at infinity 
as an explicit sufficient condition for local well-posedness. Secondly, the nonexistence result shows that this condition is sharp (critical) in the sense that even with 
$\Phi\approx f$ at infinity (where $g\approx h$ if and only if  $g\lesssim h$ and $h\lesssim g$), nonexistence of positive solutions can occur for $\vphi$ in the superspace $\Lphi$. 
  The latter is obtained by verifying that a specific initial condition $\varphi$ (given by \eqref{eq:nonexist})  constructed explicitly  in \cite{FHIL2} in terms of  the function $F$, belongs to $L^\Phi$.
In this sense any $\Phi$ satisfying $\Phi\approx f$
at infinity is a critical Young's function. In our terminology above this means that $X_c=M^f$.
This type of criticality  has been observed previously in  \cite{IRT1,IRT2} but only in the special case where $f$ has super-exponential growth with $\rho=2$. The significance of   Theorem\,C is that it holds without any upper growth restriction on $f$, i.e., $\rho=\infty$ is permitted.

Theorem\,C(b) addresses the global aspects  of the central problem by providing an additional  sufficient order relation near zero 
 for the global well-posedness of  \eqref{eq:she} with small initial data in the critical space $X_c=M^{f}$ (note that by Lemma\,\ref{lem:LHop}(iii) $f^k\approx f$ at infinity). 

 In Theorem\,\ref{thm:J}  we consider odd nonlinearities of the form $f(u)=u^mJ(u)$ (for $u\ge 0$) satisfying the conditions of Theorem\,C and Corollary\,C. Thus $u^m$ captures the behaviour of $f$ near zero while $J$ captures the behaviour of $f$ at infinity. Of  particular significance  this family contains  nonlinearities  of  arbitrarily rapid growth. Specifically, for any positive increasing convex function $r$ there is a corresponding  $J(u)\ge\exp(r(u))$ for which our well-posedness results apply (see Lemma\,\ref{lem:rapid}).  We also present additional examples of both FP-type and P-type and compare  with known results.

\subsection{Background}\label{sec:background}
For several decades   the most  influential work on the well-posedness question for  \eqref{eq:she}  has been the  $L^q$-theory of  \cite{W79,W80} 
who considered nonlinearities with  growth  dominated at large values by  a power law  function $|u|^p$ and \(X\!=\!L^q\).  There, for  $p>\pF$, the  condition  $q\ge q_c:=n(p-1)/2$ was shown to be sufficient for  
local well-posedness in $L^q$ (in a restricted sense regarding uniqueness, later removed in \cite{BC}). Moreover, for  $f$  {minorised}  by $u^p$, the condition  $q\ge q_c$ was also shown to be necessary for existence of positive solutions. In this sense the condition on $q$ was sharp, and  $L^{q_c}$ considered to be a   `critical space' (though this terminology  originated  in a different context related to norm-preservation in $L^{q}$ under a scaling invariance of \eqref{eq:she}  when $f(u)=|u|^{p-1}u$).  Viewing $L^q$ as an Orlicz space $L^{\Phi}$ with Young's function $\Phi(u)=u^q/q$, one may equivalently consider the function  $\Phi_c(u)=u^{q_c}/{q_c}$ to be a `critical Young's function'.

An extensive literature  has since developed on the well-posedness of  \eqref{eq:she} in Lebesgue spaces for nonlinearities  of P-type \cite{BC,LRSV,LS20,RS,W79,W80,W81}.  For an overview of its development and the state of the art,  \cite{QS19} is a comprehensive source. This problem has also been studied in Sobolev spaces  \cite{IJMS,Rib}, Besov spaces \cite{MZ} and a variety of other spaces \cite{Gig86, IIK,KY}. However, for nonlinearities of FP-type and framed in Banach spaces a similar corpus does not exist, with only special cases (i.e., exponential $f$) considered.

Coming some thirty-five years after  \cite{W80},   the first step-change in local well-posedness  theory
for \eqref{eq:she} dealing  successfully with nonlinearities of FP-type was the work of \cite{IRT1}.
There, important advances were made for  space dimension two $(n\!=\!2)$ with exponentially dominated nonlinearities of the form 
\[
|f(u)-f(v)|\le C|u-v| | \e^{\lam u^2}+\e^{\lam v^2}| 
\]
and $X=\Lphi=\exp L^2(\R^2)$, $\Phi (u)=\exp(u^2)-1$. With $\exp L_0^2(\R^2)$ denoting the closure of smooth, compactly supported functions in $\exp L^2(\R^2)$, they showed (\cite[Theorem~1.2]{IRT1}) that for all $\vphi\in \exp L_0^2(\R^2)$, there exists a unique solution $u\in C([0,T],\exp L_0^2(\R^2))$ of \eqref{eq:she}.  Continuous dependence on the initial data was not shown so the existence of a local semiflow
was not established. In  \cite[Theorem~1.1]{IRT1} it was shown that for 
$f$  minorised asymptotically by $\e^{\lam u^2}$ for some small $\lam >0$,  there exist nonnegative initial data $\vphi\in \exp L^2(\R^2)$ for which there is no nonnegative classical solution of \eqref{eq:she}. Thus $\Phi_c(u)\!=\!\e^{u^2}-1$ acted as a critical Young's function. We consider such cases in Section\,\ref{sec:exp} as an application of  Theorem\,C or Corollary\,C (with $\rho =2$). 

Following the template set down in \cite{IRT1}, in \cite{MT2} the authors obtained the analogous results in any dimension $n\ge 1$  for $f$ satisfying 
\[
|f(u)-f(v)|\le C|u-v| | \e^{\lam u^p}+\e^{\lam v^p}| ,\qquad p>1,
\]
obtaining a unique solution $u\in C([0,T],\exp L_0^p(\R^n))$.

The utilisation of the space $\exp L_0^p$  in \cite{IRT1,MT2}  (following its introduction in \cite{RT1}) is noteworthy, for then one has strong continuity of $S(t)$ in $\exp L_0^p$; i.e., $S(t)\vphi\!\to\!\vphi$ in $\exp L^p$ as $t\to 0$ for all $\vphi\in\exp L_0^p$. This enabled  both existence and uniqueness of   solutions  in the space $ C([0,T],\exp L_0^p)$  to be established. By contrast, in several other works (e.g., \cite{I1,MT1,RT1}) solutions  achieve the initial data only in a much  weaker sense, such as the weak$^*$-topology or $\|u(t)-S(t)\vphi\|_{\exp L^p}\to 0$ as $t\to 0$.

More subtly, and of particular relevance here, the proofs in \cite{IRT1,MT2} intrinsically made use of the fact that $\Phi(u)\!=\!\exp(u^p)-1$ is an $N$-function, so that  $\exp L_0^p$ coincided with the Morse-Transue space  $M^{\Phi}\!=\!\exp M^p$ of $\exp L^p$. This is a key observation when seeking to generalise the well-posedness theory to arbitrary Orlicz spaces, since by restricting initial data to 
  $M^{\Phi}$ one obtains both strong continuity of $S(t)$ (via $\exp L_0^p$) and invariance under $S(t)$ (via  $\exp M^p$) and thus generate  a $C_0$-semigroup in $\Mphi=\exp L_0^p=\exp M^p$.
  This is a key part of  Theorem\,A.

 Common to all the aforementioned works  in Orlicz spaces are the following two  ingredients: 
 \bi[leftmargin=1.5em]
 \item[1.] an upper bound is imposed {\em a priori} on the growth rate of the nonlinearity $f$;
 \item[2.] that  upper bound is precisely an {\em exponential} function.
 \ei
 These two features have been essential and ubiquitous in the development thus far in the Orlicz space theory for nonlinear heat equations. The reason  is that  all  smoothing estimates for  $S(t)$ between Orlicz or Lebesgue spaces relied in an essential way upon some  compatibility between the exponential   function $\Phi$ (itself induced by the exponential growth bound assumption on $f$ and some educated guesswork)   and the homogeneous power law functions $u^r$ associated with  Lebesgue space norms. This compatibility allows one to estimate the exponential-Orlicz norm of  $S(t)\vphi$ via a convergent  series of norms in Lebesgue spaces, each of which can then  be estimated by standard smoothing estimates between  Lebesgue spaces. A simplistic but illuminating  calculation with $\Phi(u)\!=\!\exp (u^p)-1$ proceeds  thus:
 \begin{align*}
\Is \Phi\left(\frac{|S(t)\vphi|}{k}\right)\dee y &= \Is \exp\left(\frac{|S(t)\vphi|^p}{k^p}\right)-1\dee y
= \sum_{r=1}^{\infty}\frac{\|S(t)\vphi \|_{pr}^{pr}}{r!k^{pr}}\\
 & \le \sum_{r=1}^{\infty}\frac{\|\vphi \|_{pr}^{pr}}{r!k^{pr}}
 =\Is \exp\left(\frac{|\vphi|^p}{k^p}\right)-1\dee y\\
 &=\Is \Phi\left(\frac{|\vphi|}{k}\right)\dee y,
\end{align*}
from which one deduces (see Definition\,\ref{def:Orlicz} of the Orlicz norm) that $S(t)\:\exp L^p\to \exp L^p$ is bounded, with
\[
\|S(t)\vphi\|_{\exp L^p}\le \|\vphi\|_{\exp L^p}, \qquad t>0.
\]
More delicate variations on this  theme lead to finer estimates between different scales of Lebesgue spaces $L^q$ and exponential-Orlicz spaces $\exp L^r$; see \cite{FKRT, I1, IRT1,IRT2,MT1,MT2} (Laplacian),  \cite{BMM,FK}  (fractional Laplacian) or \cite{MOT} (biharmonic). 

We speculate that  reliance on this method for calculating smoothing estimates of $S(t)$ is the main reason for limiting investigations to ones in exponential-Orlicz spaces and consequently to nonlinearities dominated by exponential growth.
By contrast,  in Theorem\,A and  Corollary\,A the smoothing estimate for the heat semigroup  is valid between two \emph{arbitrary} Orlicz spaces, without requiring any special exponential structure.

For  exponential nonlinearities, studies have  been concerned mainly with   global existence of solutions for small initial data. For  Schr\"odinger equations, the  problem  with $f(u)\!=\!\pm (\e^{\lam |u|^2}-1)$ was considered in \cite{NO}  in the Sobolev space $H^{n/2}(\R^n)$.  In the context of heat equations with initial data in  Orlicz spaces, problem \eqref{eq:she} seems to have been first studied  in   \cite{RT1} where the authors established local existence of weak solutions for small initial data and $f(u)\!=\!\Phi(u)\!=\!\e^{u^2}-1$. Subsequently it was shown in \cite{I1} (with $p\!=\!2$) and more recently in \cite{MT1,MT2} in the more general case $p>1$, that global-in-time weak solutions of \eqref{eq:she} exist  for small initial data   in  $\exp L^p$ for $f$ satisfying
\be
\label{eq:expLf}
|f(u)-f(v)|\le C|u-v| \left(|u|^{m-1} \e^{\lam |u|^p}+|v|^{m-1} \e^{\lam |v|^p}\right), 
\ee
provided $m\ge p>1$ and $m\ge 1+2p/n$. We consider such cases in Section\,\ref{sec:exp} as an application of  Corollary\,C and Theorem\,\ref{thm:J}. There we obtain stronger results due mainly to the flexibility we have in choosing the behaviours of $\Phi$ at zero and infinity independently, rather than being fixed by the single parameter $p$ as in previous studies. Again this advantage is due to the superior smoothing estimate afforded by Theorem\,A.

 Since \cite{IRT1,MT2},  there appear not to have been any  advances on the well-posedness theory of  \eqref{eq:she}  in Banach spaces for nonlinearities growing faster than $\exp(u^p)$. However,  some interesting results were obtained  in \cite{FI1}  for nonlinearities of  unrestricted growth  for \emph{sets} $X$ of initial data rather than initial data in a linear Banach space. As such a solution  
trajectory does not necessarily correspond to a continuous curve within a space of initial data (i.e.,  $\vphi\in X$ but $u(t)\not\in X$ in general) and does not  generate a semiflow.   
The authors developed a  Hopf-Cole transformation method and considered two large classes of nonlinearities, namely those of P-type and those of FP-type. Via an asymptotic quasi-scaling of the underlying equation  they established a nonlinear correspondence between solutions of   \eqref{eq:she} and solutions of two canonical forms of  \eqref{eq:she}, namely with $f(u)=u^p$ and $f(u)=\e^u$, corresponding to P-type and FP-type. In particular, they established (among other things)  threshold conditions for the existence or nonexistence of positive solutions for nonlinearities of FP-type.  Morally, under subsidiary assumptions on $f$ which we do not detail here,  if $f$ satisfies \eqref{eq:FIlim}
  and $[F(\vphi)]^{-\frac{n}{2}}\in \mathcal{L}^1_{\text{uloc}}\left(\R^n\right)$,
  then  \eqref{eq:she} possesses (in some sense) a local positive classical solution. Here  $\mathcal{L}^1_{\text{uloc}}$ is the closure of the space of bounded uniformly continuous functions  in    ${L}^1_{\text{uloc}}$, with ${L}^1_{\text{uloc}}$ being the space of uniformly locally integrable functions
  on $\R^n$. Conversely, for any $r\in (0,n/2)$, there exists positive $\vphi$ satisfying $ [F(\vphi)]^{-r}\in {L}^1_{\text{uloc}}$ for which  \eqref{eq:she}  has no positive solution.  
 The work  in \cite{FI1}  does not consider questions of uniqueness nor continuous dependence, sign-changing initial data or global solutions. Some of these topics, including global in time existence, were taken up subsequently in \cite{FI4,FI3,FI2,MS}, again using quasi-scaling techniques. However it does not
 appear  that these methods can be extended to obtain semiflows in linear spaces as here.

\section{Preliminaries}

Throughout we adopt the following notational conventions. $B_r$ will denote the open ball of radius $r>0$ in $\R^n$ centered at the origin;  the characteristic function on $\Omega\subset\R^n$ is denoted $\chi_\Omega$ or simply  $\chi_r$ when $\Omega\!=\!B_r$.  We write `a.e.' for {\em almost every(where)} with respect to Lebesgue measure; the set of all Lebesgue measurable functions on $\R^n$ is denoted by ${\scr{L}}(\R^n)$.

 For any nonnegative functions $g$ and $h$ we write $g\lesssim h$ whenever there exist constants $K,x_0>0$ such that $g(x)\le h(Kx)$ for all  $x\ge x_0$; we  define  $g\gtrsim h$ if and only if $h\lesssim g$ and define $g\approx h$ if and only if  $g\lesssim h$ and $h\lesssim g$.  Likewise we say `$g\lesssim h$ as $x\to 0$' if $g(x)\le h(Kx)$ for all  $|x|\le x_0$, with  
`$g\gtrsim h$ as $x\to 0$' and `$g\approx h$ as $x\to 0$' defined in the obvious way.  When comparing Young's functions we  will also have cause to write `$g\lesssim h$ on $[0,\infty )$' for the relation $g(x)\le h(Kx)$ holding for all $x\ge 0$. Thus `$g\lesssim h$' will by default refer to the relation `at infinity', all others referring explicitly to their domain of validity. We say $g$ is {\em asymptotic} to $h$ and write $g\sim h$,  
if and only if $g(x)/h(x)\to 1$  as $x\to\infty$ (or `$g\sim h$ as $x\to 0$'  if $g(x)/h(x)\to 1$  as $x\to 0$), again the default being at infinity. 

We use $C$ as a generic constant which may change value within and between  lines. By `increasing' we will mean non-decreasing, and analogously for `decreasing'.
By default,  the dependence on the domain of function spaces defined over $\R^n$ will in general be suppressed, i.e.,  $L^q\!:=\!L^q (\R^n)$, $L^{\Phi}\!:=\!\Lphi (\R^n)$, $\LO\!:=\!\LO(\R^n)$,    $\disp{M^{\Phi}\!:=\!\Mphi(\R^n)}$ (see Definition\,\ref{def:Orlicz}). Norms in Lebesgue space $L^q$ will be written as $\|\cdot\|_q$ for $q\in [1,\infty]$.

\subsection{Orlicz  and Morse-Transue Spaces}
We   recall  some definitions and concepts from the theory of Orlicz spaces,  referring the reader  to the 
comprehensive sources \cite{AF,BS,HH, KR,Lux,Mal,ON,RR1,RR2}. There exist several definitions of Orlicz spaces in the literature, chosen with particular contexts in mind, e.g., functional analysis, partial differential equations  or probability theory. It will prove fruitful when deriving our smoothing estimates to choose a weaker notion than those typically adopted in the  literature on nonlinear heat equations  \cite{FK,MT1,MT2,RT1}.
 
\begin{defn}[Young's function \cite{Mal,ON,RR1}]\label{def:Young}
We say that $\Phi \: [0,\infty ]\!\to\! [0,\infty ]$  is a (nontrivial) {\emph {Young's function}} if  $\Phi(0) \!=\! 0$ 
and either:
\bi[leftmargin=2.5em]
\item  [(I)] $\Phi$ is  convex, increasing, finite-valued, not identically zero on $[0, \infty )$, or
\item [(II)] there exists $x_{\infty}^{\Phi} > 0$ such that $\Phi$ is convex, increasing and finite-valued on $[0, x_{\infty}^{\Phi}]$, and for $x > x_{\infty}^{\Phi}$, $\Phi (x) \!=\! \infty$, or
\item [(III)] there exists  $x_{\infty}^{\Phi} > 0$ such that $\Phi$ is convex, increasing and finite-valued 
on $[0, x_{\infty}^{\Phi})$, $\displaystyle{\lim_{x\uparrow x_{\infty}^{\Phi}}\Phi (x) \!=\! \infty}$, and for $x\ge x_{\infty}^{\Phi}$, $\Phi(x) \!=\! \infty$. 
\ei
The set of all Young's functions is denoted by $\mathcal{Y}$. 
We say that $\Phi\in\mathcal{Y}$ is   {\emph {finite}} if and only if $\Phi$ satisfies (I) and {\em positive} if $\Phi >0$ on $(0,\infty)$.
\end{defn}

\begin{defn}[$N$-function \cite{AF,KR,RR1}]\label{def:Nfn}
Let $\Phi\in{\mathcal{Y}}$. We say that $\Phi$  is an  {\em {   $N$-function}} if $\Phi$ is  finite and positive    and satisfies
\[
\lim_{x\to 0}\frac{\Phi (x)}{x}=
\lim_{x\to \infty}\frac{x}{\Phi (x)}=0.
\]
The set of all $N$-functions is  denoted by $\mathcal{N}$.
\end{defn}

\begin{defn}[$\Delta_2$-condition \cite{AF,KR,Lux,Mal,RR1}]\label{def:delta2}
Let $\Phi\in{\mathcal{Y}}$. We say that $\Phi$  satisfies the   {\em {  $\Delta_2$-condition}} if  there exists  $C>0$ such that 
\[
\Phi (2x)\le C \Phi (x),
\qquad x\ge 0.
\]
 The set of all $\Phi\in{\mathcal{Y}}$  satisfying the      $\Delta_2$-condition is denoted by $\Delta_2$.
\end{defn}
It is easy to see that if $\Phi\in\Delta_2$ then $\Phi$   is  necessarily finite.   In the  theory of Orlicz spaces, the $\Delta_2$-condition is often required only to hold near  infinity (usually when the spatial domain has finite measure) while Definition\,\ref{def:delta2} is sometimes referred to as a global  $\Delta_2$-condition or $\Phi$ is called `$\Delta$-regular'  (see e.g., \cite[Sections 8.6-8.7]{AF}).

\begin{defn}[Orlicz space and subspaces \cite{AF,KR,Lux,Mal,RR1}]\label{def:Orlicz}
Let $\Phi\in{\mathcal{Y}}$.
\bi[leftmargin=2em]
\item[(a)] The {\emph {Orlicz space}} $\Lphi (\R^n) $ is defined as  
\[
\Lphi(\R^n)  = \left\{u\in {\scr{L}}(\R^n)\ :\  \Is \Phi\left(\frac{|u(x)|}{k} \right)\dee x<\infty
\ \text{{for some}}\  k>0\right\}
\]
endowed with the Luxemburg norm
\be\label{eq:Lux}
\| u\|_{\Phi}=\inf \left\{ k>0\ :\ \Is \Phi\left(\frac{|u(x)|}{k} \right)\dee x\le 1\right\}.
\ee
\item[(b)] The closure of $C_0^{\infty} (\R^n)$ in $\Lphi(\R^n) $ with respect to   $\| \cdot\|_{\Phi}$ is denoted by  $\LO(\R^n) $.
\item[(c)] The {\em Morse-Transue space} $\Mphi (\R^n)$  is defined by
\[
\Mphi (\R^n) = \left\{u\in {\scr{L}}(\R^n)\ :\  \Is \Phi\left(\frac{|u(x)|}{k} \right)\dee x<\infty
\ \text{{for all}}\  k>0\right\},
\]
with the induced norm $\| \cdot\|_{\Phi}$. 
\ei
\end{defn}

\begin{rem}\label{rem:subspace}
Definition\,\ref{def:Orlicz} can be found (in various guises) in \cite{KR,Lux,Mal,MT,ON,RR1,RR2}, with the  space $\Mphi$   first studied    in \cite{MT}. These works provide a comprehensive overview of the theory, application and history of Orlicz spaces. We mention a few facts of relevance  here.
 \bi[leftmargin=2em]
 \item[(a)] $\Lphi $ is a Banach space (\cite[Section\,3.3, Proposition\,11]{RR1}) and, as  closed subspaces  of $\Lphi$, 
 $\LO$ (by definition) and $\Mphi$  are too (see e.g. \cite[Lemma 1, p.55]{Lux}, \cite[Proposition\,1.18]{Leo} or \cite[p.97]{Mal} for the latter).
  By \cite[Theorem~12.4]{Mal}, $\Lphi\subseteq\Linf$ if and only if $\Phi$ is not finite (see Example\,\ref{eg:Young}(a-b) below).  It is easy to see (e.g. \cite[p.54-55]{Lux}) that $\Mphi\neq \{0\}$ if and only if $\Phi$ is finite. If $\Phi$ is finite then $\Mphi$ contains all bounded functions having bounded support and $\LO\subseteq\Mphi$.
  Of particular significance in this  paper, if $\Phi$ is an $N$-function ($\Phi\in{\mathcal{N}}$) then   $\Mphi =\LO$ and, less significantly for us, if $\Phi\in{\mathcal{N}}\cap \Delta_2$ then  $\Mphi =\LO=\Lphi$ (see  Lemma\,\ref{lem:L0M}). 
 \item[(b)]  If there exist constants $k,K>0$ such that
 \[
 \Psi(kx)\le \Phi(x)\le \Psi(Kx), \qquad x\ge 0
 \]
(i.e., $\Phi\approx\Psi$ on $[0,\infty)$)  then $\Lphi$ and $\Lpsi$ are equivalent as Banach spaces with equivalent norms. See e.g.,  \cite[Theorem 8.12]{AF}, \cite[Ch.II, Sec.2, Theorem 5]{Lux}, \cite[Theorem 3.4(a)]{Mal}.
 \item[(c)] Dilations of any Young's function $\Phi$  yield equivalent Orlicz spaces with equivalent norms. That is, if we define 
$\Phi_\lam\: [0,\infty ]\!\to\! [0,\infty ]$ ($\lam>0$) by $\Phi_\lam (x)=\Phi ( x/\lam)$, then $\Phi_\lam$ is a Young's function and $\Lphi =L^{\Phi_{\lam}}$ as  sets, with   norm satisfying $\| \cdot\|_{\Phi_\lam}={\lam}\| \cdot\|_{\Phi}$.
\item[(d)] $\Mphi$ is sometimes referred to as  the `heart' of $\Lphi$ and its elements  as  `finite'  (\cite[Definition 1, p.54]{Lux}). $\Lphi$ is also sometimes referred to as the `large' Orlicz space and 
$\Mphi $ as the small (or `mini') Orlicz space. We adopt the terminology of \cite[Definition 3]{RR2} in referring to $\Mphi$ as the Morse-Transue space, after the initiating work of \cite{MT}.
\ei
\end{rem}

The most familiar examples of  Orlicz spaces are of course Lebesgue spaces $L^q$, where $\Phi (x)=x^q/q$, $q\in [1,\infty)$. However,  note that the Young's functions permitted by Definition~\ref{def:Orlicz} include extended functions, i.e., functions which may not be finite-valued. Such functions are not included in the definition of Orlicz space used in \cite{FK,IRT1, MT1, MT2, RT1} for example, but are crucial here in deriving $L^{\infty}$-smoothing estimates for the heat semigroup acting on a general Orlicz spaces. See Example\,\ref{eg:Young}(a) for the important special case of $\Linf$.

\begin{defn}[Generalised inverse \cite{Mal,ON}]\label{def:geninv}
Let $\Phi\in{\mathcal{Y}}$. The {\emph {generalised inverse}} 
$\Phinv   \: [0,\infty ]\!\to\! [0,\infty ]$  is defined by
\[
\Phinv (y)=\inf \left\{ x> 0\ :\ \Phi (x)>y\right\},
\]
where $\inf\emptyset =\infty$.
\end{defn}

\begin{rem}\label{rem:Young} Let $\Phi\in{\mathcal{Y}}$.
\bi[leftmargin=2em]
\item[(a)] On the interval $ [0,x_{\infty}^{\Phi})$, $\Phi$ is locally Lipschitz continuous, differentiable a.e. (by Rademacher's Theorem or Lebesgue's Differentiability Theorem) and 
 $\Phi (x)\le x\Phi'(x)$ a.e. (see e.g., \cite[Theorem 25.1]{R70}).
\item[(b)]  $\Phi$ is bijective if and only if $\Phi$ is finite and positive on $(0,\infty)$. In such case its generalised inverse is equal to its inverse and $\Phinv$ is unbounded.
\item[(c)] $\Phinv$ is  unbounded if and only if $\Phi$ is finite.  If $\Phi$ is not finite (i.e., type (II) or  (III) in Definition\,\ref{def:Young}) then  $\Phinv (x)\le x_{\infty}^{\Phi}$  for all $x\ge 0$. 
\item[(d)] $\Phinv$ is finite-valued, increasing and right-continuous on $[0,\infty )$ with  $\Phinv (x)>0$ for all $x>0$ and  $\Phinv (\infty) = \infty$.
\item[(e)] $\Phi (x)=\sup \left\{ y>0\ :\ \Phinv (y)<x\right\}$,
where $\sup\emptyset =0$. In particular, for all $x\in [0,\infty )$
\be
\Phi \left(\Phinv (x)\right) \le x\qquad\text{and}\qquad
 \Phinv \left(\Phi (x)\right) \ge x.\label{eq:Ainv}
\ee
See e.g., \cite[Property\,1.3]{ON}.
\ei
\end{rem}

We provide a few illustrative examples for the reader unfamiliar with Orlicz spaces.  
The dilation in Example\,\ref{eg:Young}(a)  will make a subtle and important appearance in the proof of  Corollary\,A(i).

\begin{eg} \ 
 \label{eg:Young}
\bi[leftmargin=2em]
\item[(a)] (\cite[Ch.~12, Example 1]{Mal}). Condition (II) in Definition\,\ref{def:Young} permits the choice
\[
\Phi^{\infty} (x)=
\begin{cases}
0,& x\in [0,1],\\
\infty,& x>1.
\end{cases}
\]
Then 
$L^{\Phi^{\infty}}=L^{{\infty}}$,  $M^{\Phi^{\infty}}=\{0\}$, $x_{\infty}^{\Phi^{\infty}}=1$ and  $\left(\Phi^{\infty}\right)^{-1}(x)\equiv 1$. 
In particular, we note by Remark\,\ref{rem:subspace}(b) that the dilation $\Phi_2^{\infty} (x)=\Phi^{\infty} (x/2)$ satisfies
\[\left(\Phi_2^{\infty}\right)^{-1}(x)\equiv 2
\]
and 
\be \label{eq:phi2norm}
\| \cdot\|_{\infty}= \| \cdot\|_{\Phi^{\infty}} = \frac{1}{2}\| \cdot\|_{\Phi_2^{\infty}}.
\ee
\item[(b)]  (\cite[Theorem~12.4]{Mal}).  If $\Psi\in\CY$ is finite and positive and $\Phi$ is given by
\[
\Phi (x)=\Psi (x),\quad x\in [0,1]; \qquad \Phi (x)=\infty,\quad x>1,
\]
then $\Lphi =\Lpsi \cap\Linf$, $\Mphi=\{0\}$ and $x_{\infty}^{\Phi}=1$. In the special case where
$\Psi (x)=x$ we have $\Lphi =L^1 \cap\Linf$ and
\[
\Phinv (x)=x,\quad x\in [0,1]; \qquad \Phinv (x)=1,\quad x>1.
\]   
\item[(c)]  (\cite[Ch.\,12, Example 3]{Mal}).  Let $\Phi (x)=\max\{0,x-1\}$. Then $\Phi$ is finite, $\Lphi =L^1 +\Linf$,    $\Mphi\supseteq L^1 $, $x_{\infty}^{\Phi}=\infty$ and 
$\Phinv (x)=x+1$.
\item[(d)] (\cite[Theorem~12.1]{Mal}). For all $\Phi\in\CY$, $ L^1 \cap\Linf\subseteq\Lphi\subseteq L^1 +\Linf.$
\ei
\end{eg}

\begin{defn}[Young's Complement]\label{def:conj}
Let $\Phi\in\CY$. The {\emph {Young's complement}} 
$\Phi^*\: [0,\infty ]\!\to\! [0,\infty ]$  is defined by
\[
\Phi^*(x)=\sup_{ y\ge 0}  \left(xy-\Phi (y)\right).
\]
\end{defn}
It is well known (see e.g., \cite[Property\,1.6]{ON} or \cite[Section\,2.1, Proposition\,1(ii)]{RR1}) that $\Phi^*$ is a Young's function and the following inequality holds:
\be\label{eq:property1.6}
x\le \Phi^{-1}(x)\left(\Phi^*\right)^{-1}(x)\le 2x,\qquad x\ge 0.
\ee

\subsection{Regularly Varying Functions}\label{sec:regvar}

Recall (see e.g., \cite{BGT, Sen}) that a continuous function $g\: (a,\infty)\!\to\! (0,\infty)$ is {\em regularly varying of index $\rho$} if
there exists $\rho\in\R$ such that, for all $\lam >0$,
\be\lim_{x\to\infty}\frac{g(\lam x)}{g(x)}=\lro .\label{eq:kara1}\ee
If $\rho =0$ then $g$ is said to be {\it slowly varying}.  If $\sigma,\sigma_1,...,\sigma_j$ are  slowly varying functions then the following are  true \cite[Proposition\,1.3.6, Proposition\,1.5.10]{BGT}:
\begin{itemize}
[leftmargin=3em]
\item[(i)] for any $k\in\R$ and $c>0$,  $(\sigma(x))^k$  and $(\sigma(cx))$ are slowly varying;
\item[(ii)]  for any $k>0$,  $x^k\sigma(x)\to\infty$ and   $x^{-k}\sigma(x)\to 0$ as $x\to\infty;$
\item[(iii)] the product $\sigma_1(x)...\sigma_j(x)$ is slowly varying;
\item[(iv)] for any $k>1$, ${\int_{a}^{\infty}x^{-k}\sigma(x)\dee x<\infty.}$
\end{itemize}
For regularly varying $g$, the characterisation  theorem of Karamata (see e.g.,  \cite{BGT,K30,Sen})  ensures that
\[ {g(x)}=x^\rho\sigma (x),\qquad x>a 
\]
for some {slowly varying} function $\sigma$.  Moreover, if $\rho>0$ then $g(x)\to\infty$ as $x\to\infty$. 
If 
 \be \label{eq:rhoinf}
 \lim_{x\to\infty}\frac{g(\lam x)}{g(x)}=
\begin{cases}
0, & 0<\lam <1,\\
\infty, & \lam >1
\end{cases}
\ee
then $g$ is said to be {\em rapidly varying of index $\infty$}  (see e.g., \cite[Section\,2.4]{BGT} or \cite[Definition\,1.6]{Sen}). If the roles of the limits $0$ and $\infty$ in \eqref{eq:rhoinf}
are reversed then $g$ is said to be { rapidly varying of index $-\infty$}.  When convenient and where no confusion  should arise, we will {formally} include rapidly varying functions  within the class of regularly varying functions, allowing $\rho=\pm \infty$   in the sense of  \eqref{eq:rhoinf}.

One can also define  the above notions as $x\to 0$. We say that $g(x)$ is regularly varying of  index $\rho$ as $x\to 0$ if and only if $g(1/x)$ is regularly varying of  index $-\rho$ as $x\to\infty$. This will be relevant for the global well-posedness result in Corollary\,C.

\subsection{Solution Concepts}\label{sec:solnconcepts} 
For measurable $\vphi$, let us formally  define the heat semigroup $S(t)$  in $\R^n$ by 
\be\label{eq:semigp}
[S(t)\vphi](x)=(G*\vphi)(x,t):=\Is G(x-y,t)\vphi (y)\dee y,
\ee 
with  Gaussian heat kernel on $\R^n\times (0,\infty)$ 
\be\label{eq:eg}
 G(x,t):=(4\pi t)^{-\frac{n}{2}} \exp\left(-\frac{|x|^2}{4t}\right).
\ee 
It is standard to then study \eqref{eq:she} via  its  integral formulation
\be\label{eq:VoC}
u(t)=S(t)\vphi+\int_0^t S(t-s) f(u(s))\dee s=: {\F}(u;\vphi).
\ee
With  $Q_T:= \R^n\times (0,T)$ and $\F$ defined as in \eqref{eq:VoC} we  now make precise our solution concepts (see  e.g., \cite[Section 15]{QS19}).

\begin{defn}[Solution concepts]
Let $T>0$. 
\bi[leftmargin=3em]
\item[(i)] A  measurable, finite almost everywhere (a.e.) function $u\: Q_T\!\to\!\R$ is an 
\emph{ {integral solution}} of \eqref{eq:she} in $Q_T$ if $u$ satisfies ${\mathscr F}(u;\vphi)= u$ 
pointwise a.e. in $Q_T$.
\item[(ii)]  $u$ is an \emph{{  $M^{\Phi}$-mild solution}} of \eqref{eq:she} on $[0,T)$ if  $u$ is an integral solution of  \eqref{eq:she} in $Q_T$,  $u\in C\left([0,T),\Mphi\right)$  and $u(0)=\vphi$.
\item[(iii)]  $u$ is an \emph{{  $M^{\Phi}$-classical solution}} of \eqref{eq:she} on $[0,T)$ if  $u$ satisfies \eqref{eq:she} in the classical sense in $Q_T$, $u\in C\left([0,T),\Mphi\right)\cap L^{\infty}_{\loc}\left((0,T),L^{\infty}\right)$  and $u(0)=\vphi$.
\ei
\label{def:super}
\end{defn}

We note that (i) an $M^{\Phi}$-mild solution corresponds to a continuous trajectory $t\mapsto u(t)$ in $\Mphi$, for  $t\in [0,T)$; (ii)  if  $f$ is locally Lipschitz continuous and $u$ is an integral solution  of \eqref{eq:she}  such that  $u\in L^{\infty}_{\loc}\left((0,T),L^{\infty}\right)$, then by classical regularity results for parabolic equations $u\in C^{2,1}(Q_T)$ and $u$ is a classical solution of  the partial differential equation in \eqref{eq:she}; and (iii)  by definition, if there does not exist an integral solution of \eqref{eq:she}  then there does not exist an  $M^{\Phi}$-mild (or $M^{\Phi}$-classical) solution of \eqref{eq:she}. The latter will be useful in establishing nonexistence results for  \eqref{eq:she}.

\section{Main Results}\label{sec:results}
We now give precise statements of our main results, introducing our assumptions  as required.

\subsection{Theorem\,A and Corollary\,A: Heat Semigroup Properties}
\label{sec:thmA}

Recalling the definitions of subspaces of $\Lphi$ in Definition\,\ref{def:Orlicz} and  the heat semigroup $S(t)$ in \eqref{eq:semigp}, we have the following. 

\begin{theoremA}[Heat semigroup] \label{thm:A}
Let $\Phi\in{\mathcal{Y}}$.
\bi[leftmargin=2em]
\item[(a)]   (Semigroups).
\bi[leftmargin=*]
\item[(i)]  $S(t)$ is a semigroup on $\Lphi$ satisfying 
$\left\|S(t)\vphi \right\|_{\Phi}\le 2  
 \left\|\vphi \right\|_{\Phi}$  for all $\vphi\in \Lphi$ and $ t\ge 0$. 
\item[(ii)] For all $\vphi\in \LO$, $\|S(t)\vphi -\vphi\|_{\Phi}\to 0$ as $t\to 0$.
\item[(iii)] If  $\Phi\in\mathcal{N}$ 
then $S(t)$ is a $C_0$-semigroup on $\Mphi$ satisfying the estimate in (i)
 for all $\vphi\in \Mphi$  and $ t\ge 0$. 
\ei
 \item[(b)] (Smoothing). Let $\Psi\in{\mathcal{Y}}$
  and suppose there exists $\beta\in (0,1)$ and   
 $\T\in{\mathcal{Y}}$ such that
\be
\label{eq:cc}
\beta x\Psinv (x) \le { \Phinv (x)}\T^{-1}(x)\le {x\Psinv (x) },\qquad  x\ge 0.
\ee
\bi[leftmargin=*]
\item[(i)]  For all $t>0$, $S(t)\: \Lphi\!\to\!\Lpsi$ is a bounded linear operator and 
there exists a constant $\C =\C (n,\Phi,\Psi)>0$ such that  for all $\vphi\in \Lphi$
 and $ t>0$,
\be\label{eq:smoothing1}
  \left\|S(t)\vphi \right\|_{\Psi}\le  \C  \frac{{\Phinv(t^{-\frac{n}{2}})}}{\Psinv(t^{-\frac{n}{2}})}\left\|\vphi \right\|_{\Phi}. 
 \ee
 \item[(ii)]  If  $\displaystyle{\lim_{s\to \infty}\frac{\Psinv(s)}{\Phinv(s)}=0}$ then 
 \[
 \displaystyle{ \lim_{t\to 0}   \left[\frac{\Psinv(t^{-\frac{n}{2}})}{{\Phinv(t^{-\frac{n}{2}})}}\left\|S(t)\vphi \right\|_{\Psi}\right]= 0,
}
\] 
uniformly for $\vphi$ in compact subsets of  $\LO$.
\ei
\ei
\end{theoremA}

\begin{rem} \label{rem:thetachoice}\ 
\bi[leftmargin=2em]
\item[(i)] The bound in  (a)(i) of Theorem\,A is not new and can be obtained in a number of ways up to a constant,  for example via interpolation \cite[Theorem~3.2$\,^\prime$]{Mal}. 
\item[(ii)]  The choice  
\be\label{eq:thetachoice}
\Tinv (x)=
\frac{ x\Psinv(x)}
{\Phinv(x)},\qquad x>0
\ee 
clearly satisfies the inequality \eqref{eq:cc} and in many cases will be a strictly increasing bijection on $(0,\infty)$, so the generalised inverse  of  $\Tinv$ is the same as its standard function inverse; i.e.,  $(\Tinv)^{-1}=\T$. To  show that $\T$ is a Young's function, it is then sufficient to check that $\Tinv$ is concave. 
\ei
\end{rem}

Corollary\,A below is crucial for the development of our well-posedness theory for  \eqref{eq:she}.

\begin{corollaryA}[$L^{\infty}$-smoothing] \label{cor:Linf}
Let $\Phi\in{\mathcal{Y}}$.
\bi[leftmargin=2em]
\item[(i)] There exists a constant $\C=\C (n,\Phi) >0$ such that  for all $\vphi\in \Lphi$ and $ t>0$, 
\be\label{eq:smoothing3}
  \left\|S(t)\vphi \right\|_{\infty}\le \C  \Phinv\left( t^{-\frac{n}{2}}\right)\left\|\vphi \right\|_{\Phi}.
 \ee
\item[(ii)] If  $\Phi$ is  finite then
 \[
{\lim_{t\to 0}\frac{\left\|S(t)\vphi \right\|_{\infty}}{\Phinv\left( t^{-\frac{n}{2}}\right)}=0},
\]
uniformly for $\vphi$ in compact subsets of  $\LO$.
\ei
\end{corollaryA}

\begin{rem}\label{rem:Adelta2}
\bi[leftmargin=2em]
\item[\ ]
\item[(i)] To the best of our knowledge the results in Theorem\,A (except (a)(i)) and Corollary\,A in general Orlicz spaces are completely new. 
\item[(ii)]  Only Theorem\,A(a)(iii) requires $\Phi$ to be an $N$-function.  
\item[(iii)] By Remark\,\ref{rem:subspace}(a) (see  Lemma\,\ref{lem:L0M}), if    $\Phi\in{\mathcal{N}}$  then $\LO$ can be replaced by $\Mphi$ in Theorem\,A and Corollary\,A. Likewise, if 
   $\Phi\in{\mathcal{N}}\cap \Delta_2$ then both $\LO$ and $\Mphi$ can be replaced by $\Lphi$ in Theorem\,A and Corollary\,A.
   \ei
\end{rem}

We illustrate the estimates in  Theorem\,A  and Corollary\,A with a few simple examples, some  already known, some new. Those already in the literature were obtained by different methods using structural properties between exponential Orlicz spaces and Lebesgue spaces (see  previous comments in Section\,\ref{sec:background}).

\begin{eg} \label{eg:Lebesgue}
Consider the case  $\Phi (x)=x^q/q$ and  $\Psi (x)=x^r/r$ with $1\le q\le r<\infty$. 
Following Remark\,\ref{rem:thetachoice}(ii) we see that \eqref{eq:cc} is satisfied by the Young's function 
$\T (x)=  C(q,r)x^{\frac{qr}{qr+q-r}}$. 
It follows from  \eqref{eq:smoothing1} in Theorem\,A  that
\[
  \left\|S(t)\vphi \right\|_{r}\le \C    t^{-\frac{n}{2}\left(\frac{1}{q}-\frac{1}{r}\right)}\left\|\vphi \right\|_{q},\qquad t>0.
 \]
If $\Psi=\Phi^{\infty}$ then by \eqref{eq:smoothing3} in  Corollary\,A, 
\[
  \left\|S(t)\vphi \right\|_{\infty}\le \C   t^{-\frac{n}{2q}}\left\|\vphi \right\|_{q} ,\qquad t>0.
 \]
Up to a constant we therefore recover the familiar $L^q$-$L^r$ smoothing estimate between Lebesgue spaces \cite[Proposition 48.4]{QS19}.  
\end{eg}

\begin{eg} \label{eg:sumspace}
Consider $\Phi (x)=\max\{x-1,0\}$ so that $\Lphi =L^1 +\Linf$ (recall Example\,\ref{eg:Young}(c)) and  $\Psi =\Phi^{\infty}$, so that $\Lpsi=\Linf$. 
Then $\Phinv (x)=1+x$ and the smoothing estimate \eqref{eq:smoothing3} of Corollary\,A  yields
\[
  \left\|S(t)\vphi \right\|_{\infty}\le \C   \left(1+t^{-\frac{n}{2}}\right)\left\|\vphi \right\|_{L^1+L^{\infty}} ,\qquad t>0.
 \]
 This estimate is like  the one in \cite[Proposition~2.1]{ABCD}.
\end{eg}

\begin{eg} \label{eg:exp2}
Consider $\Phi (x)=\exp\exp(x^q)-\e$ for $q\ge 1$, writing $\Lphi=\exp\exp L^q$ and take $\Lpsi=\Linf$. 
The smoothing estimate \eqref{eq:smoothing3} of Corollary\,A gives
\[
  \left\|S(t)\vphi \right\|_{\infty}\le \C   \left(\log\log \left(\e+t^{-\frac{n}{2}}\right)\right)^{\frac{1}{q}}\left\|\vphi \right\|_{\exp\exp L^q},\qquad t>0.
 \]
This type of estimate with $\Phi$ growing faster than exponentially seems to be new.
\end{eg}

\begin{eg} \label{eg:expLqexpLr}
For $q\ge 1$, define $\exp L^q=\Lphi$, where $\Phi (x)=\exp (x^q)-1$. For $ r\in [q,\infty)$ we consider 
 $\Lpsi =\exp L^r$. Recalling  \eqref{eq:thetachoice}, we consider  the function
\[
\T^{-1}(x)= x \left(\log \left(1+ x\right)\right)^{\frac{1}{r}-\frac{1}{q}}.
\]
It can be shown that  $\T^{-1}$ is increasing, concave and invertible. Its inverse, $\T$,  is then a  Young's function satisfying  \eqref{eq:cc}. Estimate \eqref{eq:smoothing1} of  Theorem\,A then gives
 \[
  \left\|S(t)\vphi \right\|_{\exp L^r}\le \C  \left(\log \left(1+ t^{-\frac{n}{2}}\right)\right)^{\frac{1}{q}-\frac{1}{r}}\left\|\vphi \right\|_{\exp L^q} ,\qquad t>0.
 \]
This estimate between two exponential Orlicz spaces seems to be new.
\end{eg}

\begin{eg} \label{eg:LqexpLr}
For $ q,r\ge 1$, take $\Phi (x)=x^q/q$ and $\Psi (x)=\exp (x^r)-1$. In a similar manner to Example\,\ref{eg:expLqexpLr} and  recalling  \eqref{eq:thetachoice},
we consider the function
\[
\T^{-1}(x)= C(q)x^{1-\frac{1}{q}}\left(\log \left(1+ x\right)\right)^{\frac{1}{r}}.
\]
For $1\le q\le r$, $\T^{-1}$ is concave and its inverse $\T$  is  a suitable  Young's function satisfying \eqref{eq:cc}. 
The smoothing estimate \eqref{eq:smoothing1}  yields
\[
  \left\|S(t)\vphi \right\|_{\exp L^r}\le \C   t^{-\frac{n}{2q}}\left(\log \left(1+ t^{-\frac{n}{2}}\right)\right)^{-\frac{1}{r}}\left\|\vphi \right\|_{q} ,\qquad t>0.
 \]
This  estimate is obtained  in \cite[Proposition 3.2]{MT2} using the special exponential structure  discussed in Section\,\ref{sec:background}. See also \cite[Lemma~2.2]{I1} for a special case. 
\end{eg}

\subsection{Theorem\,B and  Corollary\,B: Local and Global   Well-posedness }\label{sec:thmB}

Recall  that the Fujita exponent is given by 
\[
\pF:=1+\frac{2}{n}.
\]
In all that follows we  will assume that $f$ satisfies:
\bi
\item[{\bf (L)}] $f\: \R\!\to\!\R$  is locally Lipschitz  continuous and $f(0)=0$.
\ei
For $f$ satisfying {  {\bf (L)}} we  define the increasing Lipschitz modulus function   $\ell  \: (0,\infty )\!\to\! (0,\infty )$  by 
\be
\ell (s)=\sup_{\substack{ |u|,|v|\le  s, \\ u\neq v}}\frac{|f(u)-f(v)|}{|u-v|}.
\label{eq:L}
\ee
Consequently, 
\be
|f(u)-f(v)|\le \ell (s)|u-v| \qquad\text{for all\ }\  |u|,|v|\le s.
\label{eq:fellplus}
\ee

\begin{defn}[Positone]\label{def:positone}
We  say that  $f\:\R\!\to\!\R$ is {\em positone} if $uf(u)\ge 0$ for all $u\in\R$.   
\end{defn}

\begin{theoremB}[Local Well-posedness]  \label{thm:B}
Let $\Phi\in\mathcal{N}$  and $f$ satisfy {\bf (L)}.   Suppose    there exists    $\lam >0$    such that
\be\tag{$I^{\infty}$}\label{eq:IC}
\int_1^{\infty} x^{-\pF}\ell\left(  \lam    \Phinv\left(x   \right)\right)\dee x<\infty .
\ee 
\bi[leftmargin=2em]
\item[(a)] (Uniform existence).   
For any   compact subset $\K$  of  $\Mphi$, there exists  $T_{\K}>0$ such that for all $\vphi\in\K$ there is   an   $M^{\Phi}$-classical solution $u(t;\vphi)$ of \eqref{eq:she} on $[0,T_{\K})$. Moreover, $u(t;\vphi)$ satisfies 
\be\label{eq:Ca}
\lim_{t\to 0}\frac{\|u(t;\vphi)\|_\infty }{\Phinv\left(t^{-\frac{n}{2}} \right)}=0
\ee
uniformly for $\vphi\in\K$. 
Furthermore, if $f$ is positone and 
$\vphi\ge 0$ (resp. $\vphi\le 0$), then $u(t;\vphi)\ge 0$ (resp. $u(t;\vphi)\le 0$) on $[0,T_{\vphi})$, where  $T_\vphi :=T_{\{\vphi\}}$.
 \item[(b)] (Uniqueness).   For all $T\in (0,T_{\vphi})$, $u(t;\vphi)$ is the unique      $M^{\Phi}$-classical solution  of \eqref{eq:she} on $[0,T)$.
\item[(c)] (Continuous dependence). For any compact subset $\K$  of  $\Mphi$, there exist $C_{\K}>0$ and $\tau_{\K}  >0$ such that for all   $\phi , \vphi\in\K$ and $ t\in (0,\tau_{\K} ]$,
\[
\left\|u (t;\phi)-u (t;\vphi)\right\|_{\Phi}+\frac{\left\|u (t;\phi)-u (t;\vphi)\right\|_{\infty}}{{\Phinv (t^{-\frac{n}{2}})}} \le C_{\K}\left\|\phi-\vphi\right\|_{\Phi}.
\]
\item[(d)] (Comparison principle).
Suppose  $f$ is  increasing.     
If  $\phi,\vphi\in\Mphi$ satisfy $ \phi\le \vphi$ then $u (t;\phi)\le u (t;\vphi)$ on $[0,T)$, where $T= \min\{T_{\phi},T_{\vphi}\}$.
\ei
\end{theoremB}

\begin{rem}\label{rem:I2I3}\ 
\bi[leftmargin=2em]
\item[(a)]   In the usual way (see e.g., \cite[Proposition 16.1(i)]{QS19}), for any $\vphi\in M^{\Phi}$ parts (a-c) of Theorem\,B allow one to define the maximal existence time of  an $M^{\Phi}$-classical solution of \eqref{eq:she}, 
which we  continue to denote by   $T_{\vphi}$. Since these solutions are  bounded and classical  for $t>0$, the usual `finite time blow-up versus global continuation' dichotomy holds in the sense of
$L^{\infty}$-classical solutions  \cite[Proposition 16.1(ii)]{QS19}. However, since we cannot guarantee a uniform existence time for initial data in bounded subsets of $M^{\Phi}$  (as per \cite[Proposition 16.1(iii)]{QS19}, say), but only in compact ones,   it is still conceivable that a solution might  be continued globally in time as an $M^{\Phi}$-mild solution even if $ T_{\vphi}<\infty$ and $\|u(t;\vphi)\|_{\infty}\to\infty$ as $t\to T_{\vphi}$  \cite{B77}.
\item[(b)] Likewise, parts (a-c) of Theorem\,B enable us to define a local semiflow $\U\:D\to M^{\Phi}$ on a  domain $D\subseteq M^{\Phi}\times [0,\infty)$ via $\U ( \vphi ,t) =u(t;\vphi)$, for $\vphi\in M^{\Phi}$ and $t\in [0,T_{\vphi})$. Moreover, if $f$ is positone (resp. increasing) then $\U $ is positone (resp. increasing) in $\vphi$.  
\item[(c)] If $f$ is   convex on $[0,\infty)$, positone and odd then $f$ satisfies {\bf (L)} and $\ell (s)=f'(s)$ for  a.e. $s\ge 0$. This type of nonlinearity is commonplace in applications  - see Section~7 for several model problems of this type, including power law and exponential nonlinearities and nonlinearities of arbitrarily rapid growth;  see also Remark\,\ref{rem:assC}.
\item[(d)] Any  $\Phi\in\mathcal{N}$  is bijective and locally Lipschitz so by the change of variables $y=\Phinv (x)$, \eqref{eq:IC} is seen to be equivalent to 
\be\label{eq:LEP0inf}
\int_1^\infty \ell  (\lam y)\left[\Phi ( y) \right]^{-\pF}  \Phi' ( y)\dee y<\infty .
 \ee
 \ei
\end{rem}

\begin{corollaryB}[Global Well-posedness]
\label{cor:B}
 Let $\Phi\in\mathcal{N}$  and $f$ satisfy {\bf (L)}.
Suppose there exists  $\lam > 0$ such that
 \be\tag{$I_0^{\infty}$}\label{eq:Iglobal}
 \int_0^{\infty} x^{-\pF}\ell\left(  \lam    \Phinv\left(x   \right)\right)\dee x<\frac{n}{4},
 \ee
 and let
\be\label{eq:Aglobal}
A :=\left(1-\frac{2}{n}\int_0^{\infty} x^{-\pF}\ell\left(  \lam    \Phinv\left(x   \right)\right)\dee x\right)^{-1}.
\ee
For every   $\vphi\in \Mphi$  with $ \left\|\vphi \right\|_{\Phi}\le \lam/(A\C)$ (where $\C$ is as in \eqref{eq:smoothing3}),  there exists  a   unique global-in-time     $M^{\Phi}$-classical solution $u\in L^{\infty}\left( (0,\infty), \Mphi\right)$ of \eqref{eq:she}. 
Moreover for all  such $\vphi$,  $\|u(t;\vphi)\|_{\infty}\le \lam\Phinv (t^{-n/2})$ for all $t>0$ 
and  $u(t;\vphi)\to 0$ in $\Linf$ as $t\to\infty$.
\end{corollaryB}

\begin{rem}\label{rem:BCdelta2}
\bi[leftmargin=2em]
\item[ ]
\item[(a)] By local uniqueness of $M^{\Phi}$-classical solutions (Theorem\,B(b)),  it is easy to see that the global solutions of Corollary\,B inherit all the same properties as the local solutions of Theorem\,B. Indeed, by uniqueness this is clearly true for small times, including continuous dependence, sign-invariance when $f$ is positone and monotonicity with respect to $\vphi$ when $f$ is increasing. Since 
these solutions are global classical solutions and $f$ is Lipschitz continuous with $f(0)=0$, classical $\Linf$-theory  imply that sign-invariance (resp. monotonicity)  are conserved for all time when $f$ is positone (resp. increasing).   
\item[(b)] In a similar manner to Remark\,\ref{rem:Adelta2}, if   
   $\Phi\in{\mathcal{N}}\cap \Delta_2$ then  $\Mphi$ can be replaced by $\Lphi$ in Theorem\,B and Corollary\,B.
  \item[(c)]   The integral condition \eqref{eq:Aglobal}   generalises the one obtained for P-type nonlinearities in \cite[Theorem\,5.1]{LS21}  in the special case  $\Lphi =L^1 \cap\Linf$. 
   \ei
\end{rem}

\subsection{Theorem\,C and Corollary\,C: Critical Young's Functions for  FP-type}\label{sec:thmC}

We now focus  attention on positone-convex nonlinearities of FP-type.
The integral condition \eqref{eq:IC} in Theorem\,B  provides a sufficient compatibility condition on $f$ and $\Phi$ for the local  well-posedness 
of classical solutions of \eqref{eq:she} in $\Mphi$,
while   \eqref{eq:Iglobal} of  Corollary\,B  provides a sufficient  condition  for  global  well-posedness  in $\Mphi$ for small initial data. 
In Theorem\,C(a) we  reduce the test  for  local  well-posedness  to a simple   comparison between $f$ and $\Phi$ with respect to the asymptotic order relation
 $\lesssim$ at infinity. 
We  show that this condition is also necessary in a certain sense, yielding an asymptotically critical  choice $\Phi_c$ of 
 $\Phi$ (modulo $\approx$)  for  local  well-posedness.  We also obtain an additional sufficient condition near zero for  local  well-posedness in the larger class $C([0,T],M^{\Phi})$ of mild-$M^{\Phi}$ solutions. In Theorem\,C(b) we are able to reduce  condition \eqref{eq:Iglobal} to a simpler one 
of integrability near zero. For nonlinearities varying regularly at zero of index $m>\pF$ (reminiscent of  \cite{Fuj1})  this condition simplifies  further to an order relation between $\Phi$ and $f$ on $[0,\infty)$ (Corollary\,C).

We introduce our  assumptions on  $f$.

\vspace{2mm}
\begin{itemize}[leftmargin=3em]
 \item[{\bf  (C)}] (Convexity).   $f\in C^1(\R)$ is positone,  positive  on $(0,\infty)$ and  convex on $[0,\infty)$.
  \item[$\boldsymbol{(\Lambda)}$] (Lipschitz majorant). There exist constants $\kappa > 0$ and  $ C >0$ such that      $\ell(u)\le Cf'(u)$ for all $u>\kappa$.  When the condition $\ell(u)\le Cf'(u)$ is required to hold for all $u>0$, we denote it by $\boldsymbol{(\Lambda_0)}$.
\end{itemize}

Essentially properties  $\boldsymbol{(\Lambda)}$ and  $\boldsymbol{(\Lambda_0)}$ respectively encapsulate the maxim that only the behaviour of  $f$ at infinity governs local existence   while control of  $f$ near zero is also important for global existence.

 \begin{rem}\label{rem:assC}  
\bi[leftmargin=2em]
 \item[\ ] 
 \item[(a)] Note that for any continuous, positone $f$ we have $f(0)=0$. Thus any $f$ satisfying  {\bf (C)} necessarily satisfies  {\bf (L)}.
  \item[(b)]    Suppose $f$  satisfies  {\bf (C)}.  If   we define the  one-sided Lipschitz modulus of $f$ by 
\[ \ell^+ (s)=\sup_{\substack{ 0\le u,v\le  s, \\ u\neq v}}\frac{|f(u)-f(v)|}{|u-v|},
\]
then  $\ell^+(u)=f'(u)$  and the condition $\ell(u)\le Cf'(u)$  in   $\boldsymbol{(\Lambda)}$ is equivalent to  $\ell(u)\le C\ell^+(u)$.   In particular,  if $f$ is odd then 
$\ell(u)=\ell^+(u)$ for all $u>0$ and $f$ satisfies $\boldsymbol{(\Lambda_0)}$. 
 \ei
 \end{rem}

Next we  introduce our FP-type  growth assumption.
 
\vspace{2mm}
\begin{itemize}[leftmargin=3em]
\item[{\bf (G)}]  (Growth).    $(\log f)'$ is regularly varying of  index $\rho-1$ with $\rho\in (0,\infty]$ and  
 $\displaystyle{\lim_{u\to\infty}{ f'(u)F(u)}}$ exists and is finite, where $F(u):= \int_{u}^\infty {1}/{f(s)}\dee s.$
\end{itemize}

 \begin{rem}\label{rem:assG}  
Suppose  $f$   satisfies  {\bf (C)} and {\bf (G)}.
\bi[leftmargin=2em]
 \item[(a)]   We will show  that  $\log f$  is regularly varying of index $\rho $ (Lemma\,\ref{lem:LHop}(i)) and that $f$ grows faster than any polynomial (Lemma\,\ref{lem:LHop}(ii)),  so $f$  is  of FP-type.
\item[(b)]  Lemma\,\ref{lem:LHop}(iv) will show that  the limiting value of $f'F$ in {\bf (G)} is   necessarily one. For sufficiently regular $f$ there is a sense in which the growth of $f'F$ and  $uf'(u)/f(u)$ at infinity  may be viewed as  H\"older conjugates;  see e.g., \cite[Remark\,1.2]{FI1} and \cite{DF1}. Thus the assumption   $f'(u)F(u)\to 1$ as $u\to\infty$ in {\bf (G)} (due to  Lemma\,\ref{lem:LHop}(iv))  can be thought of  equivalently as  $f$ having infinite growth rate as measured by $uf'(u)/f(u)$. 
\item[(c)]  For suitably regular $f$, verification of {\bf (G)} is often easier via  l'H\^opital's rule
\[
\lim_{u\to\infty} f'(u)F(u)=\lim_{u\to\infty}\frac{(f'(u))^2}{f''(u)f(u)}
=\lim_{u\to\infty}\frac{(\log f(u))'}{(\log f'(u))'},
\]
whenever the relevant limits exist.
\ei
\end{rem}

\begin{theoremC}[Critical Young's Functions] \label{thm:D}
Let    $f$  satisfy   {\bf (C)} and {\bf (G)}. 
\bi[leftmargin=2em]
\item[(a)]    
\bi[leftmargin=*]
\item[(i)] (Local well-posedness).  Suppose $f$ satisfies  $\boldsymbol{(\Lambda)}$. If      $\Phi\in\mathcal{N}$  and $\Phi \gtrsim f$ then  
the conclusions of Theorem\,B(a-c)  hold.  
\item[(ii)] (Uniqueness of mild solutions). Assume the hypotheses of (a)(i) hold.  If  there exists $C>0$ such that  $|f(u)|\le C f(|u|)$ for all $u\in\R$ and 
\be\label{eq:qzero} 
\Phi \gtrsim f^r \quad\text{ as} \quad u\to 0
\ee
for some  $r\ge 1$ and  $r>n/2$, then  
 local existence and  uniqueness holds in the  class  $C([0,T],\Mphi )$  of    $M^{\Phi}$-mild solutions.
\item[(iii)] (Nonexistence). If    $\Phi\in\mathcal{Y}$  and $\Phi \lesssim f$ then there exists nonnegative $\vphi\in\Lphi$ for which \eqref{eq:she} has no nonnegative integral solution.
\ei
\item[(b)]   (Local and global well-posedness). 
Suppose   $f$ satisfies  $\boldsymbol{(\Lambda_0)}$, $f'(0)=0$, $\Phi\in\mathcal{N}$ and $\Phi \gtrsim f$. If there exists $\lam_0>0$ such that 
\be\label{eq:intCorC}
\int_0^1 x^{-\pF}f'\left(  \lam_0    \Phinv\left(x   \right)\right)    \dee x <\infty,
\ee
then the conclusions of   Theorem\,B(a-c) and Corollary\,B (for some $\lam\in (0,\lam_0)$) both hold.
\ei
\end{theoremC}

\begin{corollaryC}[Local and Global Well-posedness]
Let    $f$  satisfy   {\bf   (C)},   $\boldsymbol{(\Lambda_0)}$ and {\bf (G)} and suppose  that  $f'$ is regularly varying at zero of index $m-1$  for some $m>\pF$.  If $\Phi\in\mathcal{N}$,  $\Phi (u)\gtrsim u^q$ as $u\to 0$ for some   $0<q<n(m-1)/2$ and $\Phi \gtrsim f$ then the conclusions of Theorem\,B(a-c) and Corollary\,B  (for some $\lam >0$) both hold. 
\end{corollaryC}

Setting $\Phi_c (u)= f(u)$ for $u\ge 0$ (noting that $\Phi_c\in\CY$ by  {\bf   (C)})  we see from Theorem\,C(a)(i) that local  well-posedness of $\Mphi$-classical solutions holds if  $\Phi\in\CN$ and $\Phi \gtrsim \Phi_c$, 
while by Theorem\,C(a)(iii)  if $\Phi \lesssim \Phi_c$ then nonexistence of a nonnegative  solution in  the superspace $\Lphi$ pertains for some initial datum.
In  this  sense   $\Phi_c$ is a critical  Young's function for local well-posedness. Modulo the  equivalence relation $\approx$ the choice $\Phi_c=f$
 is asymptotically unique  at infinity. However,  there are clearly many choices of  $\Phi\in\CN$ satisfying  $\Phi\approx f$ as $u\to\infty$
 but for which the  relation does not extend  globally to $[0,\infty)$, and thus do not define equivalent Orlicz spaces (recall Remark\,\ref{rem:subspace}(b)).  
  Thus while the choice of an asymptotically critical Young's function is essentially unique, the choice as defined on  $[0,\infty)$ is  non-unique.  
  That is, there exist $\Phi_1,\Phi_2\in\CN$    such that  $\Phi_1\approx f$ and $\Phi_2\approx f$  as $u\to\infty$, but $L^{\Phi_1}\neq L^{\Phi_2} $ and therefore $M^{\Phi_1}\neq M^{\Phi_2} $ in general. 
In practice this allows one to choose from a family of Young's functions satisfying $\Phi\approx f$ at infinity  in order to meet other desirable criteria conditioned by the behaviour of $\Phi$ near zero, such as the uniqueness of mild solutions  in Theorem\,C(a)(ii) or 
 global well-posedness as per Theorem\,C(b) or Corollary\,C.
 
  By Lemma\,\ref{lem:LHop}(iii), the requirement  in Corollary\,C   that $\Phi\gtrsim f^k$  at infinity  is in fact equivalent to $\Phi\gtrsim f$. Near zero,   $\Phi \gtrsim f^k$ for  $0<k<n(m-1)/(2m)$ can be rewritten as  $\Phi(u)\gtrsim u^{q}\sigma (u)$ for  $0<q<n(m-1)/2$   and a slowly varying function $\sigma$ (where $f(u)=u^m\mu (u)$ as $u\to 0$ and $\sigma =\mu^k$).  In practice one then chooses $\Phi\in\CN$ to interpolate between these two asymptotic order relations at zero and infinity.

\begin{rem}\label{rem:thmC}  
\bi[leftmargin=2em]
\item[\ ] 
\item[(a)]  In all parts of Theorem\,C except (a)(iii), if $f$ is  also increasing   then the comparison principle of Theorem\,B(d)   holds.  
\item[(b)] By Lipschitz continuity of $f$,  any $\Phi$  satisfying 
\be\label{eq:phiinf}
\liminf_{u\to 0}\frac{\Phi (u)}{u^d}>0,\qquad d>1
\ee
also satisfies \eqref{eq:qzero} for large enough $r$. Condition \eqref{eq:phiinf}  is not a stringent one since we may take $d$ arbitrarily large, permitting $\Phi$ to be non-degeneratively small to any finite order near zero. 
In the special case that $f$ is  odd,   $|f(u)|\le C f(|u|)$ clearly holds and  uniqueness in the class of   $M^{\Phi}$-mild solutions then follows from Theorem\,C(a)(ii). We  emphasise that the unique solution itself as guaranteed by Theorem\,B(a) is an  $M^{\Phi}$-classical one.
 We point out that uniqueness results in the class of mild solutions into the space of initial data are scarce in the literature on semilinear heat equations with nonlinearities growing faster than polynomially.
 Indeed the only result available until now is that of \cite[Theorem\,2.1]{IRT1} for the particular nonlinearity $f(u)=\exp(u^2)-1$. Theorem\,C(a)(ii) therefore represents a significant generalisation.  Non-uniqueness of solutions  were obtained in \cite{FIRT,HM,IRT1} under suitable conditions on $f$, but there the solution concepts were weaker than those considered here. 
\item[(c)]  Parts (i) and (iii) of Theorem\,C(a) provides a sharp characterisation for local well-posedness which is qualitatively unchanged for any  $\rho >0$; i.e. $\Phi_c\approx f$.  In effect, any subtle behaviour  in $f$ embodied within  the slowly varying function $\mu (u)$  (see \eqref{eq:expfintro}) is dominated by the $u^{\rho}$ term. If $\rho =0$ then the problem is more delicate and includes nonlinearities $f$ of both P-type (e.g., power law) and FP-type (e.g., quasi-polynomial). The interaction with the rate of limiting behaviour of $f'F$ is also more subtle. This case  will be addressed in a forthcoming work.
\ei
\end{rem}

\section{Proof of Theorem\,A and Corollary\,A}

 We establish first some basic properties of Young's functions and Orlicz spaces, including a proof in Lemma\,\ref{lem:L0M} that the spaces    $\LO$ and $ \Mphi $ (recall Definition\,\ref{def:Orlicz}) coincide whenever $\Phi$ is an $N$-function,  generalising  the special case $\Lphi =\exp L^2$ considered in \cite{IRT1}.  In Proposition\,\ref{prop:Gnorm} we  derive the key estimate for the norm of the Gaussian heat kernel in an arbitrary Orlicz space and apply a generalisation of Young's inequality for convolutions in Orlicz space to obtain the relevant smoothing estimates and strong continuity of the heat semigroup in  Theorem\,A.  The $L^{\Phi}$-$L^{\infty}$ smoothing estimate of Corollary\,A  will follow as a special case. 

\subsection{Lemmata}

Recall Definition\,\ref{def:geninv} for the generalised inverse of a (possibly infinite-valued)  Young's function.

\begin{lemma}
\label{lem:concave}
If $\Phi\in{\mathcal{Y}}$  then $\Phinv(x)$ is concave for  all $x\ge 0$ and $\displaystyle{\frac{x}{\Phinv(x)}}$ is increasing for all $x> 0$. 
\end{lemma}

\begin{proof}
 By convexity of $\Phi$ and \eqref{eq:Ainv}, for all $x,y\ge 0$ and $\lam\in [0,1]$ we have
\begin{align*}
\Phi\left(   \lam \Phinv (x)+ (1-\lam) \Phinv (y)\right)& \le  \lam\Phi\left(   \Phinv (x)\right)+ 
(1-\lam) \Phi\left(  \Phinv (y)\right)\\
& \le  \lam x+ (1-\lam) y.
\end{align*}
 Then, by \eqref{eq:Ainv}, convexity of $\Phi$  and since $ \Phinv $ is increasing, it follows that
\begin{align*}
  \lam \Phinv (x)+ (1-\lam) \Phinv (y)&\le \Phinv\left(  \Phi\left(   \lam \Phinv (x)+ (1-\lam) \Phinv (y)\right)\right)\\
  & \le  \Phinv\left(\lam \Phi (\Phinv (x))+ (1-\lam)\Phi (\Phinv (y)) \right)\\
& \le \Phinv\left(   \lam x+ (1-\lam) y  \right)
\end{align*}
and $\Phinv$ is concave. Monotonicity of ${{x}/{\Phinv(x)}}$ follows immediately from this.
\end{proof}

\begin{lemma} \label{lem:L0M}
Let $\Phi\in\mathcal{Y}$. 
\bi[leftmargin=3em]
\item[(i)]  If $\Phi$ is  finite then $\LO\subseteq\Mphi $.
\item[(ii)] If $\Phi\in\mathcal{N}$  then $\LO =\Mphi $.
 \item[(iii)] If $\Phi\in{\mathcal{N}}\cap \Delta_2$  then $\LO =\Mphi=\Lphi $.
\ei
\end{lemma}

\begin{proof}
(i) Let $u\in\LO$ and $k>0$ be arbitrary. Choose a sequence $ u_j$ in $C_0^{\infty}  $ and $m\in\N$ such that $\|u_j-u\|_{\Phi}<1/(2k)$ for all $j\ge m$. By convexity of $\Phi$, 
\begin{align*}
\Is \Phi (k|u(x)|)\dee x &\le  \Is \Phi \left(\frac{1}{2}{(2k|u(x)-u_m(x)|)}+\frac{1}{2}2k|u_m(x)|\right)\dee x \\
&\le  \frac{1}{2}\Is \Phi \left(\frac{|u(x)-u_m(x)|}{\|u_m-u\|_{\Phi}}\right)\dee x+\frac{1}{2}\Is \Phi \left(2k|u_m(x)|\right)\dee x\\
&\le  \frac{1}{2}+\frac{1}{2}\Is \Phi \left(2k|u_m(x)|\right)\dee x< \infty 
\end{align*}
since $\Phi$ is  finite and $u_m\in C_0^{\infty}$.  Hence $u\in\Mphi $. 

(ii)  Let $\Phi\in\mathcal{N}$ and  $u\in\Mphi $, so that
for all $k>0$
\be\label{eq:L1allk}
\Is \Phi \left(k|u(x)|\right)\dee x<\infty .
\ee
Define the sequence $u_j$ by
\[
u_j(x)=
\begin{cases}
u(x)\chi_j (x),& {\text {if}}\quad |u(x)|<j,\\
j\chi_j (x) \text{ sgn}\, u(x), & {\text {otherwise}}.
\end{cases}
\]
Note that for all $j$, $u_j$ is bounded with bounded support.
For arbitrary $\eps >0$, 
\begin{align*}
&\Is \Phi \left(\frac{|u(x)-u_j(x)|}{\eps}\right)\dee x \\
= &\ \int_{B_j\cap \{x:|u(x)|\ge j\}} \Phi \left(\frac{|u(x)-j|}{\eps}\right)\dee x+\int_{B_j^c} \Phi \left(\frac{|u(x)|}{\eps}\right)\dee x \\
= & \ \Is \Phi \left(\frac{|u(x)-j|}{\eps}\right)\chi_{A_j} (x)+ \Phi \left(\frac{|u(x)|}{\eps}\right)\chi_{B_j^c}(x)\dee x
\end{align*}
where $B_j$ is the open Euclidean ball of radius $j$ and 
\[
A_j= \{x\in B_j :|u(x)|\ge j\}.
\]
Since $u$ is measurable, for a.e. $x\in\R^n$ and all $j\in\N$
\[
\Phi \left(\frac{|u(x)-j|}{\eps}\right)\chi_{A_j} (x)\le \Phi \left(\frac{2|u(x)|}{\eps}\right).
\]
Also, for a.e. $x\in\R^n$ we have $\chi_{A_j} (x)=0$ for all $j$ large enough and so
\[
\Phi \left(\frac{|u(x)-j|}{\eps}\right)\chi_{A_j} (x)\to 0 \qquad\text{as}\qquad j\to\infty.
\]
Likewise, for a.e. $x\in\R^n$ and all $j\in\N$, 
\[
\Phi \left(\frac{|u(x)|}{\eps}\right)\chi_{B_j^c}(x)\le  \Phi \left(\frac{|u(x)|}{\eps}\right) 
\]
and
\[
\Phi \left(\frac{|u(x)|}{\eps}\right)\chi_{B_j^c}(x)\to 0 \qquad\text{as}\qquad j\to\infty.
\]
Hence by \eqref{eq:L1allk} and  Lebesgue's dominated convergence theorem, 
\[
\Is \Phi \left(\frac{|u(x)-u_j(x)|}{\eps}\right)\dee x\to 0\qquad\text{as}\qquad j\to\infty.
\]
Thus, for $j$ sufficiently large $\|u-u_j\|_{\Phi}\le \eps$. With  $E^{\Phi}$ denoting the closure in $\Lphi$ of the set of  bounded functions in $\R^n$ having bounded support (as in \cite[8.14, p.270]{AF}), we have shown that $u\in E^{\Phi}$.  Now since $\Phi$ is an $N$-function,  by   \cite[Theorem\,8.21(d)]{AF}   $E^{\Phi}=\LO$. Thus   $u\in\LO$ and $\Mphi \subseteq\LO$.   The reverse inclusion from part (i)  yields the required result.

 (iii) For $\Phi\in{\mathcal{N}}\cap \Delta_2$  it is well-known that  $\Mphi=\Lphi $ (see e.g., \cite[Corollary, p.21]{Mal}) and the result follows by (ii).
\end{proof}

\begin{rem}\label{rem:AF} Let $\Phi\in\mathcal{Y}$ be  finite but not necessarily an $N$-function.
\bi[leftmargin=2em]
\item[(i)]  
If $u_j\in E^{\Phi} $ instead of $u_j\in C_0^{\infty}$ and $u_j\to u$  in $\Lphi$ as $j\to\infty$, then the proof in part (i) of Lemma\,\ref{lem:L0M} goes through unchanged and we deduce that $ E^{\Phi} \subseteq\Mphi $.
\item[(ii)]  The equality in Lemma\,\ref{lem:L0M}(ii) has been claimed in several  works  (e.g., \cite{BMM,FK,MOT,MT2}) without requiring that $\Phi$ be an $N$-function,  citing \cite{IRT1} (where $\Phi$ is a particular exponential $N$-function) as the source of proof. It would be interesting to see an explicit proof of this claim which does not rely on the $N$-function property, as we do here via \cite[Theorem\,8.21(d)]{AF}.
\ei
\end{rem}

The following lemma is the natural generalisation of Young's inequality for convolution integrals from  Lebesgue spaces to Orlicz spaces. The reader is reminded of the definition of $ G*\vphi$ in \eqref{eq:semigp}, where $G$ is the Gaussian heat kernel  in \eqref{eq:eg}.

\begin{lemma}\rm{(\cite[Theorem~2.5]{ON})}.
\label{lem:oneil}
Let $\Phi,\Psi\in{\mathcal{Y}}$
and suppose there exists $\T\in{\mathcal{Y}}$ 
 such that
\be\label{eq:Tinv}
\T^{-1}(x){ \Phinv (x)}\le {x\Psinv (x) },\qquad  x\ge 0.
\ee
If $G\in \LT$ and  $\vphi\in \Lphi$ then $G*\vphi\in\Lpsi$ and
\[
  \left\|G*\vphi \right\|_{\Psi}\le 2  \left\|G \right\|_{\T}   \left\|\vphi \right\|_{\Phi}.
 \]
\end{lemma}

\begin{rem}\label{rem:Jod} In \cite[Theorem~4.1]{Jod} it is shown that \cite[Theorem~2.5]{ON} holds under weaker  assumptions such that the class of Young's functions ${\mathcal{Y}}$ can be replaced by `$Y$-functions'. A   $Y$-function $\Phi$ is one having the same properties as a Young's function, except the requirement that $\Phi$ be convex and instead requiring only that  $\Phi(x)/x$ be increasing. In fact the work in \cite{Jod} shows that many results in  Orlicz spaces persist on replacing  Young's functions with  $Y$-functions. The solvability of \eqref{eq:Tinv} for $\T\in{ {Y}}$  and its necessity for the convolution estimate in \cite[Theorem~2.5]{ON} were also considered in \cite{Jod} in the larger class of $Y$-functions.
\end{rem}

\subsection{Orlicz Norm Estimate of the Gaussian Heat Kernel}

In order to apply Lemma\,\ref{lem:oneil}  we require an estimate for $\left\|G \right\|_{\T}$.

\begin{prop} \label{prop:Gnorm}
Let $\T\in\CY$ 
and $G$ be the Gaussian heat kernel in \eqref{eq:eg}. 
There exists a constant $C =C (n,\T)>0$  such that for all $ t>0$, 
\be
 \left\|G(\cdot ,t) \right\|_{\T}   \le  \frac{  Ct^{-\frac{n}{2}} }{\T^{-1}(t^{-\frac{n}{2}})}.
    \label{eq:gaussnorm}
 \ee
\end{prop}

\begin{proof}
First recall from Definition~\ref{def:Young} that $x_{\infty}^{\T}\in (0,\infty]$ is the largest value of $x$ such that 
$\T$ is finite-valued on $[0,x)$. We set $\hat{x}^{\T}=1/x_{\infty}^{\T}$. 

Since $G(\cdot,t)$ is radially symmetric and decreasing in $|x|$,   from \eqref{eq:Lux} we have  
\begin{align}
  \left\|G(\cdot ,t) \right\|_{\T}&=\inf\left\{  k>0\ :\   \Is \T (G(x,t)/k)\dee x \le 1 \right\}\nonumber\\
  &=\inf\left\{  k> (4\pi t)^{-\frac{n}{2}}{\hat{x}^{\T}}\ :\   \Is \T (G(x,t)/k)\dee x \le 1 \right\}\label{eq:GG}
 \end{align}
since we necessarily restrict to those $k$ for which $\T (G(x,t)/k)$ is finite a.e. 
Using spherically radial coordinates and setting  $\kappa =(4\pi)^{-n/2}$  we have  
\begin{align}
 \Is \T (G(x,t)/k)\dee x &= C\int_0^{\infty}\T ((4\pi t)^{-\frac{n}{2}} \e^{-r^2/4t}/k)r^{n-1}\dee r\nonumber\\
&= C\tn \int_0^{\kappa   \tmn /k}\frac{\T (y)}{y}\left(  \log \left( \frac{\kappa \tmn}{ky}\right)\right)^{\frac{n}{2}-1}\dee y,
\label{eq:cov}
\end{align}
after  the second  change of variable $y=\kappa t^{-\frac{n}{2}} \e^{-r^2/4t}/k$ to obtain \eqref{eq:cov}. 
We now wish to estimate the integral on the right hand side of \eqref{eq:cov}. To do this we will consider the cases of $n$  even or $n$ odd  separately and proceed by induction in each case. 

Let $a\in (0,x_{\infty}^{\T})$ be fixed. Consider first the case of $n$ even, setting  $n=2j$. We claim that for all $j\in\N$,
\be\label{eq:Ij}
I_j(a):=\int_0^{a}\frac{\T (y)}{y}\left(  \log \left( \frac{a}{y}\right)   \right)^{j-1}\dee y\le (j-1)!\, \T (a).
\ee
By Remark\,\ref{rem:Young}(a),
\[\frac{\T (y)}{y}\le \T '(y),\qquad a.e.\  y\in (0,x_{\infty}^{\T}).
\]
Hence,
\[
I_1(a)=\int_0^{a}\frac{\T (y)}{y}\dee y \le      \int_0^{a}{\T' (y)}\dee y= \T (a).
\]
Now suppose that \eqref{eq:Ij} holds for some $j\in\N$. Then,
\begin{align*}
 I_{j+1}(a) &=  \int_0^{a}\frac{\T (y)}{y}\left(  \log \left( \frac{a}{y}\right)\right)^{j}\dee y\\
&\le    \int_0^{a}{\T' (y)}\left(  \log \left( \frac{a}{y}\right)   \right)^{j}\dee y\\
&=  \left[\T (y)\left(  \log \left( \frac{a}{y}\right)   \right)^{j} \right]_0^{a}+j \int_0^{a}\frac{\T (y)}{y}\left( \log \left( \frac{a}{y}\right)   \right)^{j-1}\dee y\\
&=   j\, I_j(a)\le j(j-1)!\, \T(a)= j!\, \T (a),
\end{align*}
using the fact that $\T (y)\le C y$ as $y\to 0$ on the penultimate line since $\T\in\CY$. By induction 
\eqref{eq:Ij} holds for all $j\in\N$.

For $k>{\hat{x}^{\T}} \kappa t^{-\frac{n}{2}}$ set $a=\kappa \tmn /k\in (0,x_{\infty}^{\T})$.  For $n$ even  it follows from \eqref{eq:cov} and \eqref{eq:Ij}  that  
 \[\Is \T (G(x,t)/k)\dee x \le C \left(\frac{n}{2}-1\right)!\, \tn \T (\kappa \tmn /k)=C\tn \T (\kappa \tmn /k).
\]

Next we consider the case of $n$ odd and let  $n=2j-1$. We claim that for all $j\in\N$, there exists a constant $C_j>0$ (independent of $a$) such that
\be\label{eq:Jj}
J_j(a):=\int_0^{a}\frac{\T (y)}{y}\left(  \log \left( \frac{a}{y}\right)   \right)^{j-3/2}\dee y\le C_j\T (a).
\ee 
First, we have
\begin{align*}
J_1(a) &= \int_0^{\a}\frac{\T (y)}{y}\left(  \log \left( \frac{\a}{y}\right)   \right)^{{-1/2}}\dee y \\
&\le
   \int_0^{\a/2}{\T' (y)}\left(  \log \left( \frac{\a}{y}\right)   \right)^{{-1/2}}\dee y
-2\int_{\a/2}^{\a}\T (y)\left(\left(  \log \left( \frac{\a}{y}\right)   \right)^{{1/2}}\right)'\dee y\\
&\le \begin{multlined}[t] \left(\log 2   \right)^{{-1/2}}{\T (\a/2)}-2\left[\T (y)\left(  \log \left( \frac{\a}{y}\right)\right)^{{1/2}} \right]_{\a/2}^{\a}\\
+2 \int_{\a/2}^{\a}{\T' (y)}\left(  \log \left( \frac{\a}{y}\right)   \right)^{{1/2}}\dee y\end{multlined}\\
&\le  \left(\log 2   \right)^{{-1/2}}{\T (\a/2)}+2\left(\log 2   \right)^{{1/2}}\T (a/2) +2 \left(\log 2   \right)^{{1/2}}(\T (\a)-\T (\a/2))\\
& \le  C_1{\T (\a)}.
\end{align*}
Now suppose that \eqref{eq:Jj} holds for some $j\in\N$. Using { exactly} the same reasoning as for $I_{j+1}$, we have
\begin{align*}
 J_{j+1}(a) &=  \int_0^{a}\frac{\T (y)}{y}\left(  \log \left( \frac{a}{y}\right)\right)^{j-1/2}\dee y\\
&\le    \int_0^{a}{\T' (y)}\left(  \log \left( \frac{a}{y}\right)   \right)^{j-1/2}\dee y\\
&=  \left[\T (y)\left(  \log \left( \frac{a}{y}\right)   \right)^{j-1/2} \right]_0^{a}+\left(j-\frac{1}{2}\right) \int_0^{a}\frac{\T (y)}{y}\left( \log \left( \frac{a}{y}\right)   \right)^{j-3/2}\dee y\\
&=   \left(j-{1}/{2}\right) J_j(a)\le \left(j-{1}/{2}\right) C_j \T (a)\\
&=:C_{j+1}\T(a),
\end{align*}
again using the convexity  of $\T$ near zero. By induction, 
\eqref{eq:Jj} holds for all $j\in\N$.

As with the even case above,  for $k>{\hat{x}^{\T}} \kappa t^{-\frac{n}{2}}$ and  $a=\kappa \tmn /k$ it follows from \eqref{eq:cov} and \eqref{eq:Jj}
 that for $n$ odd 
 \[\Is \T (G(x,t)/k)\dee x \le C\tn \T (\kappa \tmn /k).
\]
Hence, recalling \eqref{eq:GG}, we have for any $n\in\N$
\begin{align*}
  \left\|G(\cdot ,t) \right\|_{\T}&=\inf\left\{  k>{\hat{x}^{\T}}\kappa t^{-\frac{n}{2}}\ :\   \Is \T (G(x,t)/k)\dee x \le 1 \right\}\\
  &\le  \inf\left\{ k>{\hat{x}^{\T}}\kappa t^{-\frac{n}{2}}\ :\  C\tn \T (\kappa \tmn /k)  \le 1 \right\}.
 \end{align*}

Now let $C_0 \ge\max\{ 1,C\}$ and set
\be\label{eq:k01}
k_0(t):=\frac{ \kappa C_0 t^{-\frac{n}{2}}  }{\T^{-1}(t^{-\frac{n}{2}})}.
\ee
By  Lemma\,\ref{lem:concave}  $x/\T^{-1}(x)$ is increasing, and so
\begin{align*}
 & k_0(t)\ge \frac{ \kappa C_0  \left(C_0 ^{-1}t^{-\frac{n}{2}}\right)  }{\T^{-1}(C_0 ^{-1}t^{-\frac{n}{2}})}
=\frac{\kappa t^{-\frac{n}{2}}}{\T^{-1}(C_0 ^{-1}t^{-\frac{n}{2}})}\\
 \Rightarrow\   &  \kappa t^{-\frac{n}{2}}/k_0(t)\le \T^{-1}(C_0 ^{-1}t^{-\frac{n}{2}})\\
 \Rightarrow\    & \T\left(\kappa t^{-\frac{n}{2}}/k_0(t)\right)\le \T\left(\T^{-1}(C_0 ^{-1}t^{-\frac{n}{2}})\right)\le  C_0 ^{-1}t^{-\frac{n}{2}}\\
 \Rightarrow\   &  Ct^{\frac{n}{2}}\T\left(\kappa t^{-\frac{n}{2}}/k_0(t)\right)\le C_0 t^{\frac{n}{2}}\T\left(\kappa t^{-\frac{n}{2}}/k_0(t)\right)\le 1,
\end{align*}
where we have used   
  \eqref{eq:Ainv} in the penultimate implication. If we can also show that 
  \be
  k_0(t)>{\hat{x}^{\T}}\kappa t^{-\frac{n}{2}},\label{eq:k02}
  \ee
then it will follow that $ \left\|G(\cdot ,t) \right\|_{\T}\le k_0(t)$. 

If $\T$ is finite then ${\hat{x}^{\T}}=0$ and \eqref{eq:k02} clearly holds. If $\T$ is not finite then ${\hat{x}^{\T}}\in (0,\infty)$ 
and so by Remark\,\ref{rem:Young}(c),  $\Tinv$ is uniformly bounded above. Choosing $C_0$ large enough in \eqref{eq:k01} then ensures that \eqref{eq:k02}  holds. We deduce that
\[
 \left\|G(\cdot ,t) \right\|_{\T}\le k_0(t)=\frac{ \kappa C_0 t^{-\frac{n}{2}}  }{\T^{-1}(t^{-\frac{n}{2}})}, 
\]
yielding \eqref{eq:gaussnorm}.
\end{proof}

\subsection{Proof of Theorem\,A: $C_0$-semigroup and Smoothing}

\begin{proof} 
 (a)(i)  We need only show the boundedness estimate. Taking $\Psi=\Phi$ and $\T (x)=x=\Tinv (x)$ in \eqref{eq:Tinv}   of Lemma\,\ref{lem:oneil},  
we have $G\in L^1 =\LT$ and
\[
\left\|S(t)\vphi \right\|_{\Phi}= \left\|G*\vphi \right\|_{\Phi}\le 2\|G\|_1 \left\|\vphi \right\|_{\Phi}=2 \left\|\vphi \right\|_{\Phi}.
\]

(a)(ii) First note that  the embedding $L^1 \cap\Linf\hookrightarrow\Lphi$ (recall Example\,\ref{eg:Young}(d)) is continuous; see e.g., eqn. (2) in the proof of \cite[Theorem 12.1(c)]{Mal}. This is easy to see when $\Phi(1)$ is finite; if $\Phi(1)=\infty$, so that $x_{\infty}^{\Phi}\in (0,1]$, then one may dilate $\Phi$ to $\Phi_{\lam}$ such that  $x_{\infty}^{\Phi_{\lam}}>1$ and $\Phi_{\lam}(1)$ is finite. Since the norms  $\|\cdot\|_{\Phi}$ and $\|\cdot\|_{\Phi_{\lam}}$ are equivalent, the result follows from the previous case.

For any $\vphi\in\LO$, choose a sequence 
$\vphi_j$ in $C_0^{\infty} $ such that $\vphi_j\to\vphi$ in $\Lphi$ as  $j\to\infty$. 
Using the embedding above we have, for all $j\in\N$,
 \begin{align*}
 \|S(t)\vphi -\vphi\|_{\Phi}&\le \|S(t)(\vphi -\vphi_j)\|_{\Phi}+\|S(t)\vphi_j -\vphi_j\|_{\Phi}+\|\vphi_j -\vphi\|_{\Phi}\\
 &\le \begin{multlined}[t] 2\|\vphi -\vphi_j\|_{\Phi}+C\left(\|S(t)\vphi_j -\vphi_j\|_{1}+\|S(t)\vphi_j -\vphi_j\|_{\infty}\right)\\
+
 \|\vphi_j -\vphi\|_{\Phi}.\end{multlined}
 \end{align*} 
Since $\vphi_j\in C_0^{\infty} $ and $S(t)$ is strongly continuous in $L^1 $, we have 
\[
\limsup_{t\to 0}\|S(t)\vphi -\vphi\|_{\Phi}\le 3\|\vphi -\vphi_j\|_{\Phi}.
\]
Letting $j\to\infty$  yields the result.

(a)(iii) Let  $\Phi\in\mathcal{N}$ and $\vphi\in\Mphi $. For all $k>0$,
 $\Phi (k|\vphi|)\in L^1 $  so by Jensen's inequality, Fubini's theorem and standard properties of $G$ we have
\begin{align*}
\Is \Phi (k|S(t)\vphi|)\dee x & \le \Is \Phi (S(t)(k|\vphi|))\dee x\le \Is S(t)(\Phi (k|\vphi|))\dee x\\
& = \Is \Phi (k|\vphi|)\dee x<\infty .
\end{align*}
Hence $S(t)\vphi\in\Mphi$ for all $t\ge 0$. By (a)(i) $S(t)$ is a  continuous semigroup on $\Mphi$. 
By Lemma\,\ref{lem:L0M}, $\LO=\Mphi$  so by (a)(ii) $S(t)$ is strongly continuous on $\Mphi$. Thus $S(t)$ is a $C_0$-semigroup on $\Mphi$.

(b)(i)  Take $\Theta$ as in \eqref{eq:cc} so that for some $\beta\in (0,1]$,
\be\label{eq:tpp}
\frac{x}{\Tinv (x)}\le \frac{\Phinv(x)}{\beta \Psinv(x)},\qquad x>0.
\ee
The smoothing estimate \eqref{eq:smoothing1} now follows immediately from Lemma\,\ref{lem:oneil}, Proposition\,\ref{prop:Gnorm} and \eqref{eq:tpp}.

(b)(ii) 
Set 
\[
\gamma (t)=
\frac{\Psinv(t^{-\frac{n}{2}})}{{\Phinv(t^{-\frac{n}{2}})}},
\qquad t>0.
\]
By assumption $\gamma (t)\to 0$ as $t\to 0$.  
Let $\K$ be any compact subset of $\LO$. We wish to show that  for all $\eps >0$  there exists $T_{\K}>0$
 such that  for all $\vphi\in\K$ and  $t\in (0,T_{\K})$, 
\[
\gamma (t)\left\|S(t)\vphi \right\|_{\Psi}<\eps .
\]
We argue by contradiction and suppose  there exist  $\eps>0$ and  sequences $\vphi_j\in\K$
and $t_j\to 0$ as  $j\to\infty$ such that 
\be\label{eq:psij} 
\gamma (t_j)\left\|S(t_j)\vphi_j \right\|_{\Psi}\ge \eps
\ee
for all $j\in\N$. Since $\K$ is compact there exists   $\vphi\in\K$ and a subsequence (which we continue to label  $\vphi_j$) such that $\vphi_j\to\vphi$  in $\Lphi$ as $j\to\infty$. For any $s>0$ we then have, by (a)(i) and (b)(i), 
\begin{align*}
\gamma (t_j)\left\|S(t_j)\vphi_j \right\|_{\Psi}& \le  \gamma (t_j)\left\|S(t_j)(\vphi_j - S(s)\vphi  ) \right\|_{\Psi}+\gamma (t_j)\left\|S(s)S(t_j) \vphi  \right\|_{\Psi}\\
& \le  \C\left\| S(s)\vphi -\vphi_j   \right\|_{\Phi}+\C\gamma (t_j)\gamma(s)^{-1}\left\|S(t_j)\vphi  \right\|_{\Phi}\\
& \le  \C\left\| S(s)\vphi -\vphi \right\|_{\Phi}+\C\left\| \vphi -\vphi_j  \right\|_{\Phi}+2\C\gamma (t_j)\gamma(s)^{-1}\left\|\vphi  \right\|_{\Phi}.
\end{align*}
Hence
\[
\limsup_{j\to \infty} \gamma (t_j)\left\|S(t_j) \vphi_j \right\|_{\Psi} \le  \C\left\| S(s)\vphi-\vphi  \right\|_{\Phi}.
\]
Letting $s\to 0$ and again using the strong continuity of $S(t)$ from (a)(ii) (since  $\vphi\in\LO$) yields the required contradiction to \eqref{eq:psij}.
\end{proof}

\subsection{Proof of Corollary\,A: $L^{\infty}$-smoothing}

\begin{proof} 
Let $\Phi_2^{\infty}$ be as in Example\,\ref{eg:Young}(a) and let  $\Psi =\Phi_2^{\infty}$, so that  $\Psinv (x)\equiv  2$. Choosing  $\T=\Phi^*$, we see from \eqref{eq:property1.6} that \eqref{eq:cc} is satisfied with $\beta =1/2$. Recalling  \eqref{eq:phi2norm},  Theorem\,A(b)(i)  then yields part (i) of the  corollary.  Remark\,\ref{rem:Young}(c) and Theorem\,A(b)(ii)   yield part (ii). 
\end{proof}

\section{Proof of Theorem\,B and Corollary\,B}
 We prove the local and global existence of solutions via a contraction mapping argument in the spirit  of  \cite[Theorem\,4]{W80} but here within the context of Orlicz spaces and without \emph{a priori} growth restrictions on the  nonlinearity.   Theorem\,A(a)(iii) and Corollary\,A(ii) play an important role.
We establish  uniqueness and continuous dependence of $M^{\Phi}$-classical solutions of \eqref{eq:she} by following a similar strategy to that in \cite{LS20} for polynomially bounded nonlinearities and initial data in $L^1 $. To do this we first show that the solutions constructed via the contraction mapping argument in Theorem\,B(a)   depend  continuously on their initial data in $M^{\Phi}$. By uniformity of their existence time for initial data in compact subsets $\K$ of $M^{\Phi}$ we obtain  the restriction $\U_{\K}(t)\:\K\!\to\!\Mphi$ whose continuity properties  allow us to deduce the uniqueness of \emph{any} $M^{\Phi}$-classical solution. Since any solution from Theorem\,B(a) is itself an $M^{\Phi}$-classical solution, it is {\em the} $M^{\Phi}$-classical solution, yielding Theorem\,B(b)  and  simultaneously inheriting the continuous dependence property of Theorem\,B(c).  
For increasing $f$, we construct $M^{\Phi}$-classical solutions via a monotone iteration method which preserves the ordering of solutions for ordered initial data. The limiting solutions therefore inherit this ordering and a  comparison principle holds for the solutions constructed in this way.    By uniqueness, the comparison principle follows for  {\em any} $M^{\Phi}$-classical solutions, giving Theorem\,B(d).  Finally, we show that the contraction mapping argument of Theorem\,B(a) can be adapted to obtain   global-in-time solutions in Corollary\,B.

\subsection{Proof of Theorem\,B(a): Uniform Local Existence  }
\begin{proof} 
Suppose $\Phi\in\mathcal{N}$ 
  and $f$ satisfies {\bf (L)} and \eqref{eq:IC}. Let $\K$ be   any compact set  of initial data in $\Mphi$.  
 By Corollary\,A(ii)  there exists $\tau_{\K}>0$ such that for all $\psi\in\K$, 
 \be\label{eq:Bbound}
 \left\| 2S(s)|\psi |\right\|_{\infty}\le \lam \Phinv\left(s^{-\frac{n}{2}}\right),\qquad s\in (0,\tau_{\K} ).
 \ee
 Now let $\vphi\in\K$ and set 
 \be\label{eq:2S}
 w(t):=2S(t)|\vphi|, \qquad t\ge 0.
 \ee
 For $T\in (0,\tau_{\K})$ define 
\be\label{eq:defXphi}
X_{\vphi}=\left\{ u\in L^{\infty}\left( (0,T ), \Mphi\right)\ :\ |u(t)|\le w(t),\  t\in (0,T)\right\}
\ee
and endow $X_{\vphi}$ with the induced metric 
 \[
 d(u,v)=\sup_{t\in (0,T)}\|u(t)-v(t)\|_{\Phi},\qquad u,v\in X_{\vphi}.
 \]
Clearly the zero function is in $X_{\vphi}$, so $X_{\vphi}$ is  non-empty.  As in the more familiar $L^q$ case, $X_{\vphi}$ is a closed subset of the Banach space $L^{\infty}\left( (0,T ), \Mphi\right)$ and as such $(X_{\vphi},d)$  is a non-empty  complete metric space. To see this, 
 let $u_j\in X_{\vphi}$  be any  convergent sequence  in  $L^{\infty}\left( (0,T ), \Mphi\right)$ with limit $u$. Then  for a.e. $t\in (0,T)$, $u_j(t)$ converges in $\Mphi$ to $u(t)$. As in the textbook proof that $L^p $ is complete where one shows the a.e. convergence of a subsequence, the same method is used when proving $\Lphi$ or $\Mphi$ is complete (see e.g., \cite[Proposition\, 1.18]{Leo} for a simple proof). Hence,   passing to a subsequence if necessary, $u_j(x, t )$ converges pointwise a.e. to the same limit; i.e., $u_j(x,t)\to  u(x,t)$ a.e. as $j\to\infty$.  Since $|u_j|\le w$ for all $j$, it follows that $|u|\le w$. Thus $X_{\vphi}$ is closed in $L^{\infty}\left( (0,T ), \Mphi\right)$.

Recalling  \eqref{eq:VoC} we  show first that $\F$ maps  $X_{\vphi}$ into itself for small enough $T$, so let $u\in X_{\vphi}$. By definition of $\ell$ (recall \eqref{eq:fellplus}), \eqref{eq:Bbound} and \eqref{eq:2S}, we have for $s\in (0,T)$
\begin{equation}\label{eq:fLw}
|f(u(s))|\le |u(s)|\ell (w(s)) \le  \ell\left(\left\|w(s)\right\|_{\infty}\right)w(s) \le \ell\left(\lam \Phinv\left(s^{-\frac{n}{2}}\right)\right)w(s),
\end{equation}
so that
\begin{align}
|\F (u)| & \le   S(t)|\vphi|+ \It S(t-s)|f(u(s))|\,\dee s\nonumber\\
& \le   S(t)|\vphi|+ \It \ell\left(\lam \Phinv\left(s^{-\frac{n}{2}}\right)\right)S(t-s)w(s)\,\dee s\nonumber\\
& = S(t)|\vphi|+ 2\It \ell\left(\lam \Phinv\left(s^{-\frac{n}{2}}\right)\right)S(t-s)S(s)|\vphi|\,\dee s\nonumber\\
& =   S(t)|\vphi|+ 2S(t)|\vphi|\It \ell\left(  \lam  \Phinv\left(s^{-\frac{n}{2}}\right)\right)\,\dee s\nonumber\\
& =   \frac{1}{2}w(t)+  \frac{2}{n}w(t)\int^{\infty}_{t^{-\frac{n}{2}}} x^{-\pF}\ell\left(  \lam  \Phinv\left(x\right)\right)\,\dee x\nonumber\\
& \le w(t)\label{eq:CMTw}
\end{align}
for all $t\in (0,T)$ and $T$ sufficiently small (and independent of $\vphi\in\K$), using \eqref{eq:IC}.  By Theorem\,A(a)(iii), $w\in X_{\vphi}$ so it also follows from \eqref{eq:CMTw} that  \(\F (u)\in  L^{\infty}\left( (0,T ), \Mphi\right).\)
Hence for such $T$,  $\F (u)\in X_{\vphi}$ so that $\F \:X_{\vphi}\!\to\! X_{\vphi}$.

We now show that $\F$ is a contraction.  
Let $u,v\in X_{\vphi}$ 
and $t\in (0,T )$. Again by \eqref{eq:fellplus},  \eqref{eq:Bbound}, \eqref{eq:2S} and Theorem\,A(a)(i), 
\begin{align}
\|\F(u(t))-\F(v(t))\|_{\Phi} & \le    \It \left\|S(t-s)\left(f(u(s))-f(v(s))\right)\right\|_{\Phi}\,\dee s\nonumber\\
&\le   \It \ell\left(\lam \Phinv\left(s^{-\frac{n}{2}}\right)\right)\left\|S(t-s)\left(u(s)-v(s)\right)\right\|_{\Phi}\,\dee s\nonumber\\    
& \le   2\It \ell\left( \lam\Phinv\left(s^{-\frac{n}{2}}\right)\right)\left\|u(s)-v(s)\right\|_{\Phi}\,\dee s\nonumber\\
& \le   2d(u,v)\int_0^T \ell\left( \lam\Phinv\left(s^{-\frac{n}{2}}\right)\right)\,\dee s\nonumber\\
& =   \frac{4}{n}d(u,v)\int^{\infty}_{T^{-\frac{n}{2}}}  x^{-\pF}\ell\left(  \lam  \Phinv\left(x\right)\right)\,\dee x.\label{eq:d1}
\end{align}
It follows from \eqref{eq:d1} that $\F$ is a contraction on $X_{\vphi}$ for small enough $T$ (independent of $\vphi\in\K$) and thus possesses a unique  fixed point $u=u(t;\vphi)$ in $X_{\vphi}$ satisfying $\F (u)=u$. Hence  there exists a  $ T_{\K} >0$ and a local integral solution 
$u\in L^{\infty}\left( (0,T_{\K} ), \Mphi\right)\cap L^{\infty}_{\rm loc}\left( (0,T_{\K} ), L^{\infty}\right)$
satisfying   $\|u(t)\|_\infty\le \|w(t)\|_\infty\le \lam \Phinv\left(t^{-\frac{n}{2}} \right)$, for $t\in (0,T_{\K} )$.

From the integral formulation \eqref{eq:VoC} and 
classical parabolic regularity results for  Lipschitz continuous $f$, 
$u$ is a classical  solution  of \eqref{eq:she} in $Q_{T_{\K} }$. By \eqref{eq:fLw}, $f(u(t))\in \Mphi$ for all $t\in (0,T_{\K} )$
and since $u\in L^{\infty}\left( (0,T_{\K} ), \Mphi\right)$ we also have 
\[
\|f(u)\|_{\Phi} \le \ell (\|w(t)\|_\infty)\|u(t)\|_{\Phi}\le C \ell\left( \lam  \Phinv \left(t^{-\frac{n}{2}} \right) \right).
\]
By  \eqref{eq:IC}, $f(u)\in L^1((0,T_{\K} ), \Mphi)$ and so the function
\[
\left\{t\mapsto \It S(t-s)f(u(s))\,\dee s\right\}\in  C([0,T_{\K} ), \Mphi).
\]
Since by Theorem\,A(a)(iii) $S(t) $ is a $C_0$-semigroup on $\Mphi$,   
the function $t\mapsto S(t)\vphi$ is in $C([0,T_{\K} ), \Mphi)$ (see e.g., \cite[Corollary~2.3]{Pazy}). Hence $u=\F (u)\in C([0,T_{\K}), \Mphi)$ and  $u$ is an   $M^{\Phi}$-classical solution of  \eqref{eq:she}.
 
 The convergence estimate as $t\to 0$ in  \eqref{eq:Ca} follows from $ |u(t)|\le  2S(t)|\vphi|$, compactness of $\K$ and Corollary\,A(ii).

Finally,  the existence of positive (resp. negative) solutions for $\vphi\ge 0$ (resp. $\vphi\le 0$) when $f$ is positone, follows by the same contraction argument as above, but now in the metric space $(X_{\vphi}^+,d)$ (resp. $(X_{\vphi}^-,d)$), where
\[
X_{\vphi}^+=\left\{ u\in L^{\infty}\left( (0,T ), \Mphi\right)\ :\ 0\le u(t)\le w(t),\  t\in (0,T)\right\}.
\]
(resp. $X_{\vphi}^-=\left\{ u\in L^{\infty}\left( (0,T ), \Mphi\right)\ :\ -w(t)\le u(t)\le  0\right\}$).    It is clear from the positonicity of $f$ and \eqref{eq:CMTw} that   $\F $ maps $ X_{\vphi}^+$ (resp. $X_{\vphi}^-$) into itself for small enough $T>0$ and contractivity is unchanged. Regularity also follows as before.
\end{proof}

\subsection{Conditional Continuous Dependence}

\begin{prop}
Suppose $\Phi\in\mathcal{N}$
 and $f$ satisfies {\bf (L)} and \eqref{eq:IC}. 
  Let $\K$ be any compact subset  
  of $\Mphi$ and for any $\vphi\in\K$ let $u(t;\vphi)$ denote the solution of \eqref{eq:she} guaranteed by Theorem\,B(a), 
  for $t\in [0,T_{\K})$. There exist $C_{\K}>0$ and $\tau_{\K}\in (0,T_{\K})$ 
 such that for all  $\phi , \vphi\in\K$ and $t\in (0,\tau_{\K} ]$, 
\be
\label{eq:cd}
\left\| u(t;\phi)- u(t;\vphi)\right\|_{\Phi}+\frac{\left\|u (t;\phi)-u (t;\vphi)\right\|_{\infty}}{{\Phinv (t^{-\frac{n}{2}})}} \le C_{\K}\left\|\phi-\vphi\right\|_{\Phi}.
\ee
\label{prop:cd}
\end{prop}

\begin{proof}   
Let $\phi , \vphi\in\K$. To simplify notation  let   $u(t)=u (t;\phi)$ and $v(t)=u (t;\vphi)$ and set 
\[
\A (t)={{\Phinv (t^{-\frac{n}{2}})}},\qquad \gamma (t)=\frac{1}{\A(t)}, \qquad t>0.
\]
By Lemma\,\ref{lem:concave}, it is easy to see that 
\[
\A (t/2)\le 2^{n/2}\A (t),\qquad t>0.
\]
{From the proof of Theorem\,B(a) we have that $u,v\in X_{\vphi}$ so that $|u(t)|\le 2 S(t)|\phi|$ and $|v(t)|\le 2 S(t)|\vphi|$, recalling  \eqref{eq:2S} and \eqref{eq:defXphi}}. By Corollary\,A(ii),  there exists $\tau :=\tau_{\K}\in (0,T_{\K})$ such  that
 \be
 \left\|u(t)\right\|_{\infty}\le {\lam} \A (t) \quad\text{and}\quad  \left\|v(t)\right\|_{\infty}\le {\lam}  \A (t),\qquad t\in (0,\tau ].\label{eq:uv}
 \ee
{\allowdisplaybreaks
For such $t$, by \eqref{eq:uv}, Theorem\,A(a)(i) and recalling \eqref{eq:L},
\begin{align}
\|u(t)-v(t)\|_{\Phi} & \le   \|S(t)(\phi-\vphi)\|_{\Phi}+ \It \left\|S(t-s)\left(f(u(s))-f(v(s))\right)\right\|_{\Phi}\,\dee s\nonumber\\
&\le   2\|\phi-\vphi\|_{\Phi} +\It  \ell\left( \lam\A (s)\right)\left\|S(t-s)\left(u(s)-v(s)\right)\right\|_{\Phi}\,\dee s\nonumber\\
& \le  2\|\phi-\vphi\|_{\Phi} + 2\It \ell\left( \lam\A (s)\right)\left\|u(s)-v(s)\right\|_{\Phi}\,\dee s.\label{eq:L1}
\end{align}
}
Next, using Theorem\,A(a)(i) and Corollary\,A(i), we have
{\allowdisplaybreaks
\begin{align}
& \|u(t)-v(t)\|_{\infty}
   \begin{multlined}[t] \le   \|S(t)(\phi-\vphi)\|_{\infty}\\
+\int_0^{t} \left\|S(t-s)\left(f(u(s))-f(v(s))\right)\right\|_{\infty}\,\dee s\nonumber \end{multlined}\\
 \le &\  \begin{multlined}[t]\C \A (t)\|\phi-\vphi\|_{\Phi}+2\int_{t/2}^t \left\|{f(u(s))-f(v(s))}\right\|_{\infty}\,\dee s\nonumber\\
  +\int_0^{t/2} \left\|\frac{f(u(s))-f(v(s))}{u(s)-v(s)}\right\|_{\infty}\left\|S(t-s)(u(s)-v(s))\right\|_{\infty}\,\dee s\nonumber
\end{multlined}\\
 \le &\  \begin{multlined}[t] \C\A (t)\|\phi-\vphi\|_{\Phi}+ 2\int_{t/2}^t \ell\left(\lam\A (s) \right) \| {u(s)-v(s)}\|_{\infty}\,\dee s \nonumber\\
   + \C\int_0^{t/2} \ell\left(\lam\A (s) \right) \A (t-s)\| {u(s)-v(s)}\|_{\Phi}\,\dee s\nonumber
   \end{multlined}\\
\le &\   \begin{multlined}[t]\C\A (t)\|\phi-\vphi\|_{\Phi}+2\A (t/2)\int_{t/2}^t \ell\left( \lam\A (s)\right) \gamma(s)\| {u(s)-v(s)}\|_{\infty}\,\dee s\nonumber\\
  +\C\A (t/2)\int_0^{t/2} \ell\left(  \lam \A (s)\right)\| {u(s)-v(s)}\|_{\Phi}\,\dee s\nonumber
  \end{multlined}\\
\le &\   \begin{multlined}[t] \C\A (t)\|\phi-\vphi\|_{\Phi}+2^{1+n/2}\A (t)\int_{0}^t \ell\left( \lam \A (s)\right) \gamma(s)\| {u(s)-v(s)}\|_{\infty}\,\dee s\\
  +2^{n/2}\C\A (t)\int_0^{t} \ell\left(  \lam \A (s)\right)\| {u(s)-v(s)}\|_{\Phi}\,\dee s.\label{eq:Linf}
  \end{multlined}
\end{align}
}
Combining (\ref{eq:L1})-(\ref{eq:Linf}) we obtain, for some $C=C(\C,n)>0$ and all $t\in (0,\tau ]$,
\begin{align}
& \|u(t)- v(t)\|_{\Phi} +\gamma(t)\|u(t)-v(t)\|_{\infty}   \le C\|\phi-\psi\|_{\Phi} \nonumber\\
&+C\int_0^{t} \ell\left(  \lam\A (s)\right)\left( \| {u(s)-v(s)}\|_{\Phi}+ \gamma(s) \| {u(s)-v(s)}\|_{\infty}\right)\,\dee s.\label{eq:gron}
\end{align} 
Now define $y(t)$ on $[0,\tau ]$ by $y(0)=\|\phi-\vphi\|_{\Phi}$ and
\[
y (t)=\|u(t)-v(t)\|_{\Phi} +\gamma(t)\|u(t)-v(t)\|_{\infty}, \qquad t\in (0,\tau ].
 \]
Then by \eqref{eq:gron}, 
\[
y(t)\le C y(0)+C\It \ell\left(  \lam\Phinv \left(s^{-n/2}\right)\right) y(s)\,\dee s, \qquad t\in (0,\tau ].
\]
By assumption $u,v\in C\left([0,\tau ],\Lphi\right)$ and  both are classical solutions for $t>0$. By standard parabolic regularity results we have that
 $u,v\in  C\left((0,\tau ),L^{\infty} \right)$,
  so that $y$ is continuous on $(0,\tau ]$. By \eqref{eq:Ca} of Theorem\,B(a), $y(t)\to \|\phi-\vphi\|_{\Phi}=y(0)$ as $t\to 0$ and so  $y$ is continuous on $[0,\tau ]$.
 Hence by  \eqref{eq:IC} and  the singular Gronwall inequality (see  e.g.,  \cite[Ch.XII, Theorem 4]{MPF}), it follows  that
\[
\|u(t)-v(t)\|_{\Phi} +\gamma(t)\|u(t)-v(t)\|_{\infty}   \le  C\|\phi-\vphi\|_{\Phi} \e^{q(t)}
\]
for all $t\in (0,\tau ]$, where
\[
q(t)=C\int_0^t \ell\left( \lam \Phinv \left(s^{-n/2}\right)\right)\,\dee s.
\]
Clearly by \eqref{eq:IC}  $q$ is continuous and $q(t)\to 0$ as $t\to 0$   so there exists $C_{\K}>0$ such that 
$q(t)\le C_{\K}$  for all $t\in (0,\tau ]$ and \eqref{eq:cd} follows.
\end{proof}

\subsection{Proof of Theorem\,B(b-c): Uniqueness and Continuous Dependence}

\begin{proof}

(b) (Uniqueness). We follow the methods in \cite[Lemma\,9]{BC} and \cite[Theorem\,2.4]{LS20} in Lebesgue spaces, with minor adjustments to incorporate Proposition\,\ref{prop:cd}. Set $u (t):=u (t;\vphi)$ for $t\in [0,T_{\vphi})$ where $u (t;\vphi)$ is the $M^{\Phi}$-classical solution guaranteed  by Theorem\,B(a). 
Let $T\in (0, T_{\vphi})$ be arbitrary and  suppose  there  exists another  
    $M^{\Phi}$-classical solution $v$ on $[0,T]$  with $v(0)=\vphi$.
 By classical $L^{\infty}$-theory   it is clear that  if there exists a $T_0 \in (0,T)$ such that $u (t)=v(t)$ on $(0,T_0)$, then
  $u (t)=v(t)$ on $[0,T]$, i.e., uniqueness for sufficiently small times implies uniqueness on $[0,T]$.

Let $\K =v\left([0,T]\right)$. Since  by assumption $v\in C\left([0,T],\Mphi\right)$,  $\K$
 is a  compact  subset (metric space) of $\Mphi$. Now let  $\tau_{\K}$ be as in Proposition\,\ref{prop:cd} and set $\tau =\min\{\tau_{\K},T\}$.      
For all $t\in [0,\tau)$ we may then define $\U_{\K}(t)\:\K\!\to\! \Mphi$  by $\U_{\K}(t)\psi =u (t;\psi )$  and deduce from
 Proposition\,\ref{prop:cd}  that $\U_{\K}(t)$ is continuous (with respect to the induced metric from $\|\cdot\|_{\Phi}$). Again
 by classical $L^{\infty}$-uniqueness theory, for any $s\in (0,T)$ and $0<t< \min\{\tau,T-s\}$, we have  $v(t+s)=\U_{\K}(t)v(s)$.   
 Letting $s\to 0$ and using the continuity of $v\: [0,\tau ]\!\to\!\K$ and $\U_{\K}(t)\:\K\!\to\! \Mphi$ we  obtain $v(t)=\U_{\K}(t)v(0)=\U_{\K}(t)\vphi=u (t;\vphi)=u (t)$ for all  $t> 0$ small enough, as required.

(c) (Continuous dependence). By uniqueness of       $M^{\Phi}$-classical solutions from part (a),  any such solution is {\emph{necessarily}} the solution obtained in Theorem\,B(a). Hence it satisfies the inequality \eqref {eq:cd}
 of Proposition\,\ref{prop:cd} and continuous dependence of  $M^{\Phi}$-classical solutions follows. 
\end{proof}


\subsection{Proof of Theorem\,B(d): Comparison Principle}

\begin{proof} Assume $f$ is increasing and thus positone, since $f(0)=0$. For any $\phi\in\Mphi$ set 
\[
\phi^-=\min\{\phi,0\}\le 0,\qquad \phi^+=\max\{\phi,0\}\ge 0
\]
and  
\[
v(t;\phi )=2S(t)\phi^-\le 0,\qquad w(t;\phi )=2S(t)\phi^+\ge 0.
\]
By Theorem\,A(a)(iii) and Corollary\,A(i), 
\[ v,w\in L^{\infty}\left((0,T),\Mphi \right)\cap L^{\infty}_{\loc}\left((0,T),L^{\infty}\right)\]
 and by  Corollary\,A(ii)  there exists $\tau_\phi>0$ such that  
 \[
\left\| 2S(s)\phi^{\pm}\right\|_{\infty}\le \lam \Phinv\left(s^{-\frac{n}{2}}\right),\qquad s\in (0,\tau_\phi ).
 \]
 Arguing  as in  \eqref{eq:CMTw}, we have
 \begin{align}
\F (w;\phi) & =   S(t)\phi+ \It S(t-s)f(w(s))\,\dee s\nonumber\\
& \le   S(t)\phi^+ + 2\It \ell\left(\left\| 2S(s)\phi^+\right\|_{\infty}\right)S(t-s)S(s)\phi^+\,\dee s\nonumber\\
& \le   S(t)\phi^+ + 2S(t)\phi^+\It \ell\left(  \lam  \Phinv\left(s^{-\frac{n}{2}}\right)\right)\,\dee s\nonumber\\
& =   \frac{1}{2}w(t)+  \frac{2}{n}w(t)\int^{\infty}_{t^{-\frac{n}{2}}} x^{-\pF}\ell\left(  \lam  \Phinv\left(x\right)\right)\,\dee x\nonumber\\
& \le w(t)\label{eq:superw}
\end{align}
for all $t\in (0,T^+_\phi)$ and $T^+_\phi>0$ sufficiently small. Hence   $w$ is an integral supersolution (see e.g., \cite{LS20,RS}). 
We may then define the sequence $w_k$ via the iterative procedure 
\[
w_{k+1}(t;\phi )={\F}(w_k;\phi),\qquad w_0(t;\phi )=2S(t)\phi^+ ,\qquad t\in [0,T^+_\phi ).
\]
By \eqref{eq:superw}, $w_1(t;\phi )\le w_0(t;\phi )$. Since $f$ is monotone increasing we see that $\F (w ;\phi)$ is increasing in $w$ and it follows by induction that   $w_k(t;\phi )$ is a decreasing sequence.

In an almost identical manner  it can be shown that $\F (v;\phi)\ge v$  for all $t\in (0,T^-_\phi)$ and $T^-_\phi>0$ sufficiently small and  one can construct
 an increasing  sequence $v_k(t;\phi )$ via  
\[
v_{k+1}(t;\phi )={\F}(v_k;\phi),\qquad v_0(t;\phi )=2S(t)\phi^- ,\qquad t\in [0,T^-_\phi ).
\]
Again since $\F (\cdot ;\phi)$ is increasing and $v_0(t;\phi )\le w_0(t;\phi )$  it is easy to show by induction that $v_k(t;\phi )\le w_k(t;\phi )$ for all $k$. 

In particular, the sequence $w_k(t;\phi )$ is decreasing and bounded below by $v_0(t;\phi )$. By Levi's monotone convergence theorem  we may  pass to the pointwise limit in $w_k(t;\phi )$ to obtain a measurable function $U(t;\phi)$ satisfying $U={\F}(U;\phi)$. Since $v_0\le U\le w_0$, $U$ is a.e. finite in $Q_{T_\phi}$ and $U$ is a local integral solution of \eqref{eq:she}. Also since $v_0,w_0\in L^{\infty}\left((0,T),\Mphi \right)\cap L^{\infty}_{\loc}\left((0,T),L^{\infty}\right)$, one may argue as in the proof of regularity in Theorem\,B(a) to show that $U$ is an   $M^{\Phi}$-classical solution. Then by uniqueness   $U(t;\phi)=u(t;\phi)$, with $u(t;\phi)$  as in Theorem\,B(a).

For initial data $\vphi\in\Mphi$ with $\vphi\ge\phi$, we may repeat the  iterative procedure above to obtain an   $M^{\Phi}$-classical solution $U(t;\vphi)$ of \eqref{eq:she} 
as the pointwise limit of the decreasing sequence defined by
\[
w_{k+1}(t;\vphi )={\F}(w_k;\vphi),\qquad w_0(t;\vphi )=2S(t)\vphi^+ ,\qquad t\in [0,T^+_\vphi ).
\]
Again by uniqueness   $U(t;\vphi)=u(t;\vphi)$, with $u(t;\vphi)$  as in Theorem\,B(a).

Finally, by the monotonicity of $\F (\cdot ;\cdot)$  in  its second argument it is easy to show  inductively that 
$w_{k}(t;\phi )\le w_{k}(t;\vphi )$ for all $k$. Passing to the pointwise limit as $k\to\infty$ we obtain 
$u(t;\phi)= U(t;\phi)\le U(t;\vphi)=u(t;\vphi)$, for all small $t$. Since these solutions are classical for $t>0$, comparison then holds on their common interval of existence by standard $L^{\infty}$-theory.
\end{proof}

\subsection{Proof of Corollary\,B: Global Well-posedness }

\begin{proof} 
Assume \eqref{eq:Iglobal} holds and $A$ is given by  \eqref{eq:Aglobal}. Observe that $1<A<2$. Suppose   $\vphi\in \Mphi$ with $ \left\|\vphi \right\|_{\Phi}\le \lam/(A\C)$, where $\C$ is as in \eqref{eq:smoothing3}. We proceed in an almost identical manner to the proof of local existence from part (a) of Theorem\,B, via a contraction mapping argument. As the details are so similar we summarise the calculations more succinctly.

Set $w(t)=AS(t)|\vphi|$ and define  
\[
X^{\infty}_{\vphi}=\left\{ u\in L^{\infty}\left( (0,\infty), \Mphi\right)\ :\ |u(t)|\le w(t),\  t>0\right\}
\]
 with the induced metric 
 \[
 d_{\infty}(u,v)=\sup_{t>0}\|u(t)-v(t)\|_{\Phi},\qquad u,v\in X^{\infty}_{\vphi}.
 \]
By Theorem\,A(a)(i) and (iii), $w\in X^{\infty}_{\vphi}$  and $(X^{\infty}_{\vphi},d_{\infty})$ is a non-empty  complete metric space. 
For any  $u\in X^{\infty}_{\vphi}$,
 \begin{align*} 
\left|{f(u(s))}\right|& \le  
\ell\left(\left\|w(s)\right\|_{\infty}\right){w(s)}
\le \ell\left(  A\C \|\vphi\|_{\Phi} \Phinv\left(s^{-\frac{n}{2}}\right)\right){w(s)}\\
& \le \ell\left(  \lam \Phinv\left(s^{-\frac{n}{2}}\right)\right){w(s)}.
\end{align*}
Hence for all $t>0$,
 \begin{align*}
|\F (u)| & \le   
    S(t)|\vphi|+ \It \ell\left(  \lam \Phinv\left(s^{-\frac{n}{2}}\right)\right)S(t-s)w(s)\,\dee s\nonumber\\
& =   S(t)|\vphi|+ AS(t)|\vphi|\It \ell\left(  \lam \Phinv\left(s^{-\frac{n}{2}}\right)\right)\,\dee s\nonumber\\
& =   \frac{1}{A}w(t)+  \frac{2}{n}w(t)\int^{\infty}_{0} x^{-\pF}\ell\left(  \lam  \Phinv\left(x\right)\right)\,\dee x\nonumber\\
& =w(t).
\end{align*}
 Since $w\in X^{\infty}_{\vphi}$ it  follows  that $\F (u)\in X^{\infty}_{\vphi}$ and   $\F \:X^{\infty}_{\vphi}\!\to\! X^{\infty}_{\vphi}$.

For any $u,v\in X^{\infty}_{\vphi}$ we have
\begin{align}
\|\F(u(t))-\F(v(t))\|_{\Phi} & \le    \It \left\|S(t-s)\left(f(u(s))-f(v(s))\right)\right\|_{\Phi}\,\dee s\nonumber\\
&\le   \It \ell\left(\left\|w(s)\right\|_{\infty}\right)\left\|S(t-s)\left(u(s)-v(s)\right)\right\|_{\Phi}\,\dee s\nonumber\\
& \le   2\It \ell\left(  \lam \Phinv\left(s^{-\frac{n}{2}}\right)\right)\left\|u(s)-v(s)\right\|_{\Phi}\,\dee s\nonumber\\
& \le    \frac{4}{n}d_{\infty}(u,v)\int^{\infty}_{0} x^{-\pF}\ell\left(  \lam  \Phinv\left(x\right)\right)\,\dee x,\nonumber
\end{align}
yielding a contraction  by \eqref{eq:Iglobal}.

In an identical way to the proof of Theorem\,B(a), we obtain an    
$M^{\Phi}$-classical solution $u\in L^{\infty}\left( (0,\infty), \Mphi\right)$ of \eqref{eq:she}. 
In particular $|u(t)|\le AS(t)|\vphi|$ so that by  Corollary\,A(i) $\|u(t)\|_{\infty}\le \lam\Phinv (t^{-n/2})$. Since $\Phi$ is positive $\Phinv (0)=0$, so by continuity of  $\Phinv$
\[\limsup_{t\to\infty}\|u(t)\|_{\infty}\le \lam\Phinv (0)=0.\]
\end{proof}

\section{Proof of Theorem\,C and  Corollary\,C}

We begin by establishing some   regular variation  properties of relevant  functions. 
We then show that the asymptotic order relation satisfied by $\Phi$  in  Theorem\,C(a)(i)  is sufficient to ensure that the integral condition \eqref{eq:IC} of   Theorem\,B holds, to obtain local well-posedness of \eqref{eq:she}. For uniqueness of mild solutions in Theorem\,C(a)(ii) we show that the additional condition \eqref{eq:qzero} is sufficient to guarantee that any mild solution is necessary classical and thus by Theorem\,B(b) unique.  The nonexistence result in part (a)(iii) of Theorem\,C is obtained via an application of some recent results  in \cite{FHIL2} on dilation-critical singularities.  Theorem\,C(b) and Corollary\,C   are obtained by verifying \eqref{eq:Iglobal}  subject to  
suitable conditions on $f$ and $\Phi$ near  zero.

Where necessary the reader is advised to consult Section\,\ref{sec:regvar} on regularly varying functions, especially regarding our terminology in the case $\rho =\infty$.
We will write $fF$ for the product function $fF(u):=f(u)F(u)$.

\begin{lemma}\label{lem:LHop}
Suppose  $f:(0,\infty)\to (0,\infty)$ is  $C^1$ and  $(\log f)'$ is regularly varying of  index $\rho-1$ with $\rho\in (0,\infty]$.  Then the following hold:
\bi[leftmargin=2em]
\item[(i)] $\log f$ is regularly  varying of index $\rho\in (0,\infty]$;
\item[(ii)] for  any $k>0$,  if $\rho\in (0,\infty)$ then $u^{-k}f(u)\to\infty$ as $u\to\infty$ while  if  $\rho=\infty$ then $u^{-k}\log f(u)\to\infty$ as $u\to\infty$;
\item[(iii)] for any $k>0$, $f^k\approx f$.
\ei
Furthermore,  if $f$  is eventually nondecreasing and  satisfies  {\bf (G)} then
\bi[leftmargin=2em]
\item[(iv)]  $\displaystyle{\lim_{u\to\infty}{ f'(u)F(u)}}=1$;
\item[(v)]   $fF$ is regularly   varying of index ${1-\rho}\in [-\infty, 1)$.
\ei
\end{lemma}

\begin{proof}

(i)      
If $\rho\in (0,\infty )$ then  by l'H\^opital's rule  we obtain
	\begin{align}
	 \lim_{u\to\infty}\frac{\log f(\lam u)}{\log f(u)}& = \lim_{u\to\infty}\frac{(\log f(\lam u))'}{(\log f(u))'}=
	\lam \lim_{u\to\infty}\frac{ f'(\lam u)/f(\lam u)}{ f'(u)/f(u)}\nonumber\\
& =\lam \lim_{u\to\infty}\frac{ (\log f)'(\lam u)}{(\log f)' (u)}=\lam^{\rho},\label{eq:logrv}
	\end{align}
so that  $\log f$ is regularly varying of index $\rho $. If $\rho =\infty$ and $\lam\in (0,1)$ then the limit in \eqref{eq:logrv} is zero and  l'H\^opital's rule still applies; for $\lam >1$, setting $v=\lam u$ in $\log f(\lam u)/\log f(u)$ then yields an infinite limit, as the reciprocal of \eqref{eq:logrv}. Thus  $\log f$ is  rapidly varying of index $\rho =\infty$.

(ii) If $\rho\in (0,\infty)$  then by (i), $f(u)=\exp({u^\rho \mu (u)})$ 
for some slowly varying function $\mu$ and the first result follows.  
If $\rho=\infty$ then let $g=\log f$ and 
  $k>0$ be arbitrary. By  definition of rapid variation, if $\lambda < 1$ then ${g(\lambda u)}/{g(u)} \to 0$ as $u\to\infty$. 
Now by considering the function $h(u) = u^{-k}g(u)$, we have
\[ 
\lim_{u\to\infty}\frac{h(\lambda u)}{h(u)} = \lim_{u\to\infty}\frac{g(\lambda u)}{(\lambda u)^k} \frac{u^k}{g(u)} = {\lambda^{-k}}\lim_{u\to\infty}\frac{g(\lambda u)}{g(u)}=0.  \]
Hence $h$  is rapidly varying so that $h(u)\to\infty$ as $u \to \infty$, as required.

(iii)  Suppose first that  $\rho\in (0,\infty)$.  By part (i),    $f(u)=\exp (u^\rho \mu(u))$  for some slowly varying function $\mu$. For any $\lam >k^{-1/\rho}$,
\[
\frac{f^k(\lam u)}{f(u)}=\exp \left(u^\rho \mu(\lam u)\left(k\lam^\rho -  \frac{\mu(u)}{\mu(\lam u)}\right)\right)
\sim \exp \left(\left(k\lam^\rho -  1\right)u^\rho \mu(\lam u)\right)\to\infty
\]
as $u\to\infty$. Thus  $f^k\gtrsim f$. Likewise choosing $\lam<k^{-1/\rho}$ shows that $f^k\lesssim f$. Hence $f^k\approx f$. If $\rho =\infty$ then again by part (i) (recalling \eqref{eq:rhoinf}) we have, for any $\lam >1$, 
\[
\lim_{u\to\infty}\frac{\log f^k(\lam u)}{\log f(u)}=k\lim_{u\to\infty}\frac{\log f(\lam u)}{\log f(u)}=\infty.
\]
Hence for $u$ large enough, $\log f^k(\lam u)\ge \log f(u)$ so that $ f^k(\lam u)\ge  f(u)$ and $f^k\gtrsim f$. A similar argument with $\lam <1$ shows that 
$f^k\lesssim f$.

 (iv)   Let $A\in (0,\infty)$ denote the limit of $f'(u)F(u)$ as $u\to\infty$ in {\bf (G)}. By   \cite[Remark\,1.1]{FI1},  necessarily $A\ge 1$. To show that $A=1$, suppose for contradiction  that  $A>1$ and let $\eps \in (0,A-1)$.  Then there exists $u_1>0$ such that  $A-\eps\le f'(u)F(u)\le A+\eps$ for all $u\ge u_1$. Since $F'=-1/f<0$, we have
 \[
 \frac{ F''(u)}{ F'(u)}=f'(u)F'(u)\le (A-\eps)\frac{ F'(u)}{ F(u)},\qquad u\ge u_1. 
 \] 
Integrating twice leads to a bound of the form $F(u)\ge {C}{u^{-\frac{1}{A-1-\eps}}}$ for $u\ge u_1$. 
Then integrating the inequality 
\[
f'(u)\le \frac{A+\eps}{F(u)}\le {C^{-1}}{(A+\eps)}{u^{\frac{1}{A-1-\eps}}}
\]
yields a bound of the form $f(u)\le Cu^{\frac{A-\eps}{A-1-\eps}}$ for $u\ge u_2$, 
contradicting the growth estimates of part (ii).
 
(v) By {\bf (G)}, (i) and  (iv) we have,
\begin{align*}
\lam^{\rho-1}&=\lim_{u\to\infty}\frac{ f'(\lam u)/f(\lam u)}{ f'(u)/f(u)}=
\lim_{u\to\infty}\left\{\left[\frac{ f(  u)F (u)}{f(\lam  u)F(\lam u)}\right]\frac{  \left[f'(\lam u)F(\lam u)\right]}{ \left[f'( u)F( u)\right]} \right\}   \nonumber \\
	&=  \lim_{u\to\infty}\frac{f(  u)F (u)}{f(\lam  u)F(\lam u)}
	\end{align*}
so that (upon taking the reciprocal) $fF$ is regularly varying of index ${1-\rho}\in [-\infty,1)$   (rapidly varying of index $-\infty$ when $\rho =\infty$). 
\end{proof}

\begin{rem}\label{rem:after71}
Suppose $f$ satisfies the hypotheses of Lemma\,\ref{lem:LHop}(iv)-(v) and let $\lam, \eps\in (0,1)$ and $M>0$ be arbitrary.   (Note in particular that any $f$ satisfying {\bf (C)} and {\bf (G)} has these properties.) Let us  record some immediate consequences of  Lemma\,\ref{lem:LHop}(iv)-(v),    used repeatedly in  subsequent proofs.  
\bi[leftmargin=2em]
\item[(i)] There exists    $y_1=y_1(\eps)>0$  such that 
\be
1-\eps\le f'(y)F(y)\le 1+\eps, \qquad y\ge y_1.\label{eq:epsfF}
\ee  
\item[(ii)] For any $\rho\in (0,\infty)$,  there exists    $y_2=y_2(\lam, \eps ,\rho)>0$  such that 
\be
\frac{f(\lam y)F(\lam y)}{\lam f(y)F ( y)}\ge \lam^{-\rho}-\eps ,\qquad y\ge y_2.\label{eq:dellam}
\ee
\item[(iii)] For $\rho=\infty$  (calling the definition of rapidly varying index $-\infty$, following \eqref{eq:rhoinf}), 
there exists   $y_3=y_3(\lam,M)>0$ such that
\be
\frac{f(\lam y)F(\lam y)}{\lam f(y)F ( y)}\ge M ,
 \qquad y\ge y_3.\label{eq:liminfy1}
\ee
\ei
\end{rem}

\subsection{Proof of Theorem\,C(a)(i): Local well-posedness}

\begin{proof}
Let   $\Phi\in\mathcal{N}$ and    $f$ satisfy    {\bf   (C)},  $\boldsymbol{(\Lambda)}$ and {\bf (G)}.
  We wish to  apply Theorem\,B(a)  to deduce well-posedness of \eqref{eq:she} in $\Mphi$.    
  We  need only verify the integral condition  \eqref{eq:IC}. 

Let $\rho\in (0,\infty]$ and $\Phi \gtrsim f$. By  {\bf (C)},  $f$ is (eventually) invertible and by $\boldsymbol{(\Lambda)}$,   $\ell (u) \le Cf'(u)$ for all $u>0$ large enough.
  By Remark\,\ref{rem:Young}(b), 
$\Phi$ is also invertible.    
Hence there exist $k >0$ such that $ \Phinv ( u)\le  kf^{-1}(u)$  for $u>0$ large enough. 

Regarding \eqref{eq:IC}, by Remark\,\ref{rem:I2I3}(d) and \eqref{eq:epsfF}  we have for $a>0$ large enough 
\begin{align}
\int_a^{\infty} x^{-\pF}\ell\left(  \lam_0    \Phinv\left(x   \right)\right)\dee x &\le 
C\int_a^{\infty} x^{-\pF}f' \left(  \lam   f^{-1}\left(x   \right)\right)\dee x \nonumber\\
&=C\int_b^{\infty} f(y)^{-\pF}f' (\lam y)f'(y)\dee y \label{eq:intlater}\\
&\le 
C\int_b^{\infty} \left[f(y)^{\pF}F (\lam y)F(y)\right]^{-1}\dee y,\label{eq:intF}
\end{align}
where  $b=f^{-1}\left(a \right)$ and  $\lam:=k\lam_0 $ with $\lam\in (0,1)$  to be chosen later. 
Now  consider the integrand in \eqref{eq:intF}, setting
\[h(y):=\left[f(y)^{\pF}F ( y)F(\lam y)\right]^{-1}.
\] 
We wish to show that $h$ is integrable for large $y$.

Fix $M >n/2$. For $0<\eps <2/(n \pF)<1$  let
\be\label{eq:A1}
\A_1(\rho) :=
\begin{cases}
{(2/n-\eps \pF)(\lam^{-\rho}-\eps)-1}, & \rho\in (0,\infty),\\
{(2/n-\eps \pF)M-1}, & \rho = \infty.
\end{cases}
\ee 
For fixed $\rho\in (0,\infty)$  we see from the first of \eqref{eq:A1}  that $\A_1(\rho)>0$ for  $\lam >0$ sufficiently small. In the case $\rho =\infty$,  we see from the second of \eqref{eq:A1} that $\A_1(\rho)>0$ for   $\eps >0$ sufficiently small. 
We now fix  $\lam$ and $\eps$ in this way so   that $\A_1(\rho) >0$.

By  \eqref{eq:epsfF}, \eqref{eq:dellam} and \eqref{eq:liminfy1},  for $y$ large enough ($y\ge y_0$ say) we have
\begin{align*}
h'(y)&=
-h(y)^2\left[\pF f(y)^{2/n}f'(y)F(y)F(\lam y)+f(y)^{\pF}F'(y)F(\lam y)
+\lam f(y)^{\pF}F(y)F'(\lam y)\right]
\nonumber\\
&=-h(y)\left(\pF\frac{f'(y)}{f(y)}+\frac{F'(y)}{F ( y)}+\frac{\lam F'(\lam y)}{F(\lam y)}\right)\nonumber\\
&=-h(y)\left(\pF\frac{f'(y)}{f(y)}-\frac{1}{f(y)F ( y)}-\frac{\lam }{f(\lam y)F(\lam y)}\right)\nonumber\\
&\le -h(y)\left(\frac{\pF(1-\eps)-1}{f(y)F ( y)}-\frac{\lam}{f(\lam y)F(\lam y)}\right)\nonumber\\
&= -\frac{{\lam} h(y)}{f(\lam y)F(\lam y)}\left(\frac{(2/n-\eps \pF)f(\lam y)F(\lam y)}{{\lam}f(y)F ( y)}-1\right)\nonumber\\
 &\le -\frac{\A_1(\rho){\lam } h(y)}{f(\lam y)F(\lam y)}\le  -  \frac{\A_1(\rho)}{1+\eps}\frac{\lam f'(\lam y)}{f(\lam y)}h(y)\nonumber\\
&=-A_1 (\log (f(\lam y)))'h(y),\nonumber
\end{align*}
where
\[
A_1=\A_1(\rho)/(1+\eps)>0.
\]
Integration of this inequality then yields
\be\label{eq:hint}
h(y)
\le  C[f(\lam y)]^{-A_1},\qquad y\ge y_0.
\ee
But by Lemma\,\ref{lem:LHop}(ii),  $f$ grows faster than any polynomial at infinity and so $h$  is integrable on  $[y_0,\infty )$.
Hence \eqref{eq:she} is well-posed in $\Mphi$ in the sense of  Theorem\,B(a)-(c).
\end{proof}

\subsection{Proof of Theorem\,C(a)(ii): Uniqueness of Mild Solutions}

\begin{proof}

  We  show  uniqueness in the larger class $C([0,T],\Mphi)$ of $\Mphi$-mild solutions   under the additional assumptions  that $|f(u)|\le C f(|u|)$ for all $u\in\R$   and 
  $\Phi (u)\gtrsim f(u)^r$ near zero, for some $r\ge 1$ and $r>n/2$ (recall \eqref{eq:qzero}).

To this end, we claim first  that   there  exist $C,\mu >0$  such that  $f (u )^r\le C\Phi (\mu u )$ for all $u\ge 0$. 
Since  $\Phi \gtrsim f^r$ near zero there exist $ \mu_1,u_1>0$ such that  
 \be\label{eq:C1}
 f(u)^r\le  \Phi ( \mu_1 u ),\qquad u\in [0,u_1].
 \ee 
 By assumption $\Phi\gtrsim f$ and by Lemma\,\ref{lem:LHop}(iii), $f^r\approx f$. Hence $\Phi\gtrsim f^r$ and  so there exist and $\mu_2,u_2>0$ such that 
\be\label{eq:C2}
f(u)^r\le \Phi(\mu_2 u) ,\qquad u\ge u_2.
\ee
By continuity,  \eqref{eq:C1}, \eqref{eq:C2} and the positivity of  $\Phi $  on the compact interval $[u_1,u_2]$, it  follows  that there exists $C>0$ such
\be\label{eq:frbound}
\sup_{u>0 }\frac{ f( u)^r}{\Phi(\mu u)}\le C,
\ee
where $\mu=\max\{\mu_1,\mu_2\}$ and the claim follows.

We now use the bound \eqref{eq:frbound} to show  that any $M^{\Phi}$-mild solution is necessarily an $M^{\Phi}$-classical one.  Uniqueness of $M^{\Phi}$-mild solutions will then follow from Theorem\,B(b).
So let $v\in C([0,\tau],\Mphi )$ be any $M^{\Phi}$-mild solution  of \eqref{eq:she}.  
We show that $v\in L^{\infty}_{\loc}\left((0,T),L^{\infty} \right)$ for  $T\in (0,\tau)$ sufficiently small, so that $v$ is necessarily an $M^{\Phi}$-classical solution
for small positive times. Uniqueness on  $(0,T_{\vphi})$ then follows by classical $L^{\infty}$-theory.

For any $T\in (0,\tau)$, $v$ satisfies the integral equation
\[
v(t)=S(t)\vphi+\int_0^t S(t-s) f(v(s))\dee s,\qquad t\in (0,T),
\]
with $v(0)=\vphi$. By Corollary\,A(i),   $S(\cdot)\vphi\in L^{\infty}_{\loc}\left((0,T),L^{\infty} \right)$ so it  remains only to show that the function
\be\label{eq:vlocbdd}
\left\{
t\mapsto \int_0^t S(t-s) f(v(s))\dee s
\right\}
\in L^{\infty}_{\loc}\left((0,T),L^{\infty} \right),
\ee
for  $T>0$ sufficiently small.      
Since $v\in C([0,T],\Mphi )$  we can choose  $T$ sufficiently small such that
\[
\sup_{s\in (0,T)} \|v(s)-\vphi\|_{\Phi}\le \frac{1}{2\mu}.
\]
By  definition of $\|\cdot\|_{\Phi}$ we then have,  for all $s\in (0,T)$,
\[
\Is \Phi\left(   2\mu |v(x,s)-\vphi(x)|\right)\dee x\le 1.
\]
By assumption $|f(u)|\le C f(|u|)$, \eqref{eq:frbound}  and  convexity of $\Phi$,
\begin{align*}
\|f(v(s))\|_r^r &
\le C \Is \left[f(|v(x,s)|)\right]^r\dee x \le  C\Is \Phi(\mu |v(x,s)|)\dee x\\
&\le  \frac{C}{2}\Is \Phi\left(   2\mu |v(x,s)-\vphi(x)|\right)+\Phi\left(   2\mu |\vphi(x)|\right)\dee x\\
&\le  {C}\left(1+\Is\Phi\left(   2\mu |\vphi(x)|\right)\dee x\right)<\infty,
\end{align*}
recalling that  $\vphi\in\Mphi$. Hence,
\[\sup_{s\in (0,T)}\|f(v(s))\|_r<\infty .
\]
Recalling that  $r\ge 1$ and $r>n/2$,  by standard $L^r$-$L^{\infty}$ smoothing of the heat semigroup we have
\[ \sup_{t\in (0,T)}\left\| \int_0^t S(t-s) f(v(s))\dee s\right\|_{\infty}
 \le  C\sup_{t\in (0,T)}\int_0^t (t-s)^{-n/(2r)}\dee s\le \frac{2Cr}{2r-n}T^{1-\frac{n}{2r}}<\infty ,
\]
and \eqref{eq:vlocbdd} follows.
\end{proof}

\subsection{Proof of Theorem\,C (a)(iii): Nonexistence}\

\begin{proof}
Let   $\Phi\in\mathcal{Y}$ and  
 $f$ satisfy     {\bf   (C)}  and {\bf (G)}. 
 Assumptions {\bf   (C)} and  {\bf (G)} allow us to utilise  results of \cite{FHIL2} which, among other things,  provide  sufficient conditions on $f$ and the initial data $\varphi$ for nonexistence of nonnegative integral solutions of \eqref{eq:she}. Specifically, \cite[Corollary\,1.1]{FHIL2}  guarantees that there exists $\G >0$ such that
 \eqref{eq:she} possesses no nonnegative integral solution for the initial datum
\be\label{eq:vc}
\vphi_c (x) =F^{-1}(\G |x|^{2}),\qquad x\in\R^n\backslash\{0\}.
\ee
Then \cite[Theorem\,1.1(a)]{FHIL2}  ensures the same is true for the initial datum
\be\label{eq:nonexist}
\vphi (x) =F^{-1}(\G |x|^{2})\chi_R(x),\qquad x\in\R^n\backslash\{0\},
\ee
for any $R>0$.   (Note: our assumptions {\bf   (C)} and {\bf (G)} here ensure that assumptions `M', `S', `C' and `L' in \cite{FHIL2} hold.  In the  terminology of \cite{FHIL2} we have $q_f=1$, $p_f=\infty$ and  $f$ is `supercritical'. By utilising  \cite[Corollary\,1.1]{FHIL2} we ensure  that $ \vphi_c$ in \eqref{eq:vc} is  locally integrable,  so that the same is true of $\vphi$
in \eqref{eq:nonexist}. This in turn allows us to  apply \cite[Theorem\,1.1(a)]{FHIL2} with initial data $\vphi$.)
It therefore suffices to show  that $\vphi$ as given by \eqref{eq:nonexist} satisfies $\vphi\in\Lphi$.

Since $\Phi \lesssim f$ there exist $ K, u_0 > 0$ such that $\Phi (u)\le f(K u )$  for
all $u \ge u_0$. Set $\lam=K\lam_0 $ (with $\lam>0$  to be chosen later) and choose $R=R(\lam)>0$ small enough  such that $\lam_0F^{-1}(\G R^2)\ge u_0$. Since $\vphi$ is radially symmetric and decreasing,  for all $x\in B_R\backslash\{0\}$ we have  
$\lam_0\vphi (x)\ge \lam_0F^{-1}(\G R^2)\ge u_0$.  Hence,
\begin{align}
\Is \Phi \left(\lam_0 \vphi (x) \right) \dee x &\le  \int_{B_R} f \left(\lam_0 K\vphi (x) \right) \dee x
 =  C\int_0^R f \left(\lam F^{-1}(\G r^{2}) \right)r^{n-1} \dee r\nonumber\\
&= 
C\int_{b}^{\infty} \frac{f (\lam s)}{f(s )}[F(s)]^{(n-2)/2} \dee s\label{eq:intf1},
\end{align}
where  $b=F^{-1}(\G R^{2})$, recalling that $F(u)\to 0$ as $u\to\infty$. 

Now set
\[ g(s):=\frac{f(\lam s)}{f(s )}[F(s)]^{(n-2)/2}.
\]
By \eqref{eq:intf1} we wish to show that $g$ is integrable for some $\lam>0$.

Fix $M >2/n$. For  $0<\eps <n/2$ define  the constant
\be\label{eq:A2}
\A_2(\rho) :=
\begin{cases}
{(n/2-\eps)(\lam^{-\rho}-\eps )-(1+\eps)}, & \rho\in (0,\infty),\\
{(n/2-\eps )M-1}, & \rho = \infty.
\end{cases}
\ee 
For fixed $\rho\in (0,\infty)$  we see from the first of \eqref{eq:A2}  that $\A_2(\rho)>0$ for  $\lam >0$ sufficiently small. In the case $\rho =\infty$,  we see from the second of \eqref{eq:A2} that $\A_2(\rho)>0$ for   $\eps >0$ sufficiently small. 

We now fix  $\lam$ and $\eps$ in this way so   that $\A_2(\rho) >0$.
By  \eqref{eq:epsfF}, \eqref{eq:dellam} and \eqref{eq:liminfy1},  for $s$ large enough we have
\begin{align*}
g'(s)&=
-g(s)\left( \frac{ f'(s)}{f(s)}-\lam \frac{ f'(\lam s)}{f(\lam s)}+\frac{(n-2)}{2f(s)F(s)}\right)\nonumber\\
&\le  -g(s)\left( \frac{1-\eps}{f(s)F(s)}-\frac{ \lam (1+\eps)}{f(\lam s)F(\lam s)}+\frac{(n-2)}{2f(s)F(s)}\right)\nonumber\\
&=  -\frac{\lam g(s)}{f(\lam s)F(\lam s)}\left( (n/2-\eps)\frac{f(\lam s)F(\lam s)}{\lam f(s)F(s)}-  (1+\eps)\right)\nonumber\\
&= -\frac{\A_2(\rho)\lam g(s)}{f(\lam s)F(\lam s)}\le -\frac{\A_2(\rho)}{(1+\eps)} \frac{\lam f'(\lam s)}{f(\lam s)}g(s)\nonumber\\
&=-A_2 (\log (f(\lam s)))'g(s),
\end{align*}
where
\[
A_2={\A_2(\rho)}/{(1+\eps)}>0.
\]
The integrability of $g$ for large $s$ then follows in the same way as for  $h$ in \eqref{eq:hint}.
\end{proof}

\subsection{Proof of Theorem\,C(b): Global Well-posedness.}

\begin{proof} 
Clearly, since  $\boldsymbol{(\Lambda_0)}\Rightarrow\boldsymbol{(\Lambda)}$ the assumptions of Theorem\,C(a)(i) hold and so the local well-posedness statement is trivial.

Regarding global well-posedness,   by $\boldsymbol{(\Lambda_0)}$   and recalling \eqref{eq:Iglobal} we have
\[ \int_0^{\infty} x^{-\pF}\ell\left(  \lam    \Phinv\left(x   \right)\right)\dee x\le C \int_0^{\infty}H(\lambda, x) \dee x,
\]
where
\[
H(\lambda, x):=x^{-\pF}f'\left(  \lam    \Phinv\left(x   \right)\right).
\]
Since  $f'(0)=0$,  $H(\lambda, x)\to 0$ as $\lambda \to 0$ pointwise in $x>0$  and   $f$ is convex, we have $ 0\le H(\lambda, x) \le H(\lambda_0, x)$
for all $\lam\le\lam_0$ and  $x>0$. If we can show that 
$H(\lambda', x)$ is integrable on $(0,\infty)$ for some $\lam'>0$, then global well-posedness   will follow by the dominated convergence theorem and Corollary\,B. 
Clearly it is sufficient to verify only the integrability of $H$ for $x$ near zero and $x$ near infinity. By assumption,  $H(\lambda_0, x)$ is integrable  near zero.  Since $\Phi\gtrsim f$, the integrability of $H$ at infinity for some $\lam_1\in (0,1)$ has already been shown in the proof of Theorem\,C(a)(i)  (recall the calculation around \eqref{eq:intlater}). Hence $H(\lambda', x)$ is integrable on $(0,\infty)$ with  $\lam'=\min\{\lam_0,\lam_1\}$, so  by the dominated convergence theorem
\[
\lim_{\lam\to 0}\int_0^\infty H(\lambda, x)\dee x=0.
\]
Thus, \eqref{eq:Iglobal} holds for all $\lam >0$ sufficiently small and the  result follows by Corollary\,B.
\end{proof}

\subsection{Proof of Corollary\,C: Local and Global Well-posedness.}

\begin{proof} Again  the assumptions of Theorem\,C(a) hold so the local well-posedness statement is trivial. 

For global well-posedness we verify  the hypotheses of Theorem\,C(b).  
By  assumption $\Phi \gtrsim f$.  
Since $f'$ is regularly varying of index $m-1$ at zero, 
$f'(x)=x^{m-1}\mu (1/x)$   
for  $x\in (0,1)$, where $\mu$ is slowly varying at infinity. Clearly  $f'(0)=0$, recalling the properties of regularly varying functions following \eqref{eq:kara1}. 
Now we check   \eqref{eq:intCorC}.  Since $\Phi(u)\gtrsim  u^q$ near zero, there exist $K,a>0$ such that $\Phinv(u)\le K u^{1/q}$ for  $u\in (0,a)$.  Hence, for $0<q<n(m-1)/2$ and with $\lam =\lam_0K$ and  $b=a^{1/q}$,
\begin{align*}
\int_0^a x^{-\pF}f'\left(  \lam_0    \Phinv\left(x   \right)\right)    \dee x  & \le  \int_0^a x^{-\pF}f'\left(  \lam   x^{1/q}\right)    \dee x=C\int_0^b  y^{-1-2q/n}f'(\lambda y) \dee y\\
&=C\int_0^b  y^{m-2-2q/n}\mu(1/(\lam y))\dee y\\
&=  C \int_{1/b}^\infty z^{-m+2q/n}\mu(z)  \dee z <\infty
\end{align*}
since  $-m+2q/n<-1$ and $\mu$ is slowly varying at infinity,  again recalling the properties of regularly varying functions following \eqref{eq:kara1}.  Hence  \eqref{eq:intCorC} holds.
\end{proof}

\section{Applications}\label{sec:Apps}
We illustrate our results with  some applications and examples of both FP-type and P-type.  In order to more easily separate  the behaviour of $f$ (and therefore $\Phi$) near infinity and  zero, in Theorem\,\ref{thm:J} we introduce a family of  odd nonlinearities of the form $f(u)=u^mJ(u)$ (for $u\ge 0$) to which Theorem\,C or Corollary\,C apply. 
Since it is  not immediately obvious that there do indeed exist functions of arbitrarily large growth rate satisfying the hypotheses of  Theorem\,C or Corollary\,C, we show in Lemma\,\ref{lem:rapid} that the family in Theorem\,\ref{thm:J}  contains a very large sub-family  of such functions (at least as many as there are positive increasing convex functions). We also present two examples of exponential type from this family; one of a type studied previously in the literature, where $\rho$ is finite \cite{FKRT,I1,IRT1,IRT2, MT1,RT1}, the other a composition of exponentials with $\rho=\infty$ which, as far as we know, has not been 
considered in the literature in the context of heat equations in Orlicz spaces, or indeed any other Banach space (excepting the trivial case of $\Linf$).   We stress that our choice of exponential functions is purely for expositional convenience, with no reliance on any particular structural properties such as those required for smoothing estimates  in previous studies. Finally we consider  two examples of P-type to which  Theorem\,B and Corollary\,B apply. 
The first  is the classical Fujita equation set in Lebesgue space for which the standard theory in \cite{BC,W80, W81} already applies and thus acts as a comparator.   The second example is a log-corrected Fujita equation for which our theory provides sharper results via Orlicz spaces than those that can be obtained from  \cite{BC,W80} in Lebesgue spaces.

\subsection{A Canonical Family of Nonlinearities}\label{sec:J}

 We consider  odd nonlinearities of the form $f(u)=u^mJ(u)$ for $u\ge 0$, where $J(0)>0$ and $J$ is of FP-type. Thus $f$ behaves like $u^m$ as $u\to 0$
and like $J(u)$ as $u\to\infty$, with $J$ encapsulating the FP-type growth of $f$ at infinity.

\begin{theorem} \label{thm:J}
Suppose $J\: [0,\infty)\to [0,\infty)$ satisfies the following conditions: 
\bi
\item[{\bf   (C$^\prime$)}] $J(0)>0$,  $J\in C^1$ is increasing and  convex on $[0,\infty)$ and eventually $C^2$.
\item[{\bf (G$^\prime$)}] $ (\log J)'$ is { regularly varying} of index $\rho-1$ with  $\rho\in (0,\infty]$ and 
\newline $\displaystyle{ \lim_{u\to\infty} \frac{(J'(u))^2}{J''(u)J(u)}  =1}$. 
\ei
For any $m \ge 1$, let  $f:\R\to\R $ be odd with $f(u):=u^mJ(u)$ for $u\ge 0$.
\bi[leftmargin=2em]
\item[(a)]    
\bi[leftmargin=*]
\item[(i)] (Local well-posedness). If      $\Phi\in\mathcal{N}$ and  $\Phi \gtrsim J$ then 
the conclusions of Theorem\,B(a-c) all hold.   If   in addition $\Phi(u) \gtrsim u^{d}$  as $u\to 0$
for some  $d>1$ then uniqueness holds in the  class  $C([0,T],\Mphi )$  of    $M^{\Phi}$-mild solutions.
\item[(ii)] (Nonexistence). If    $\Phi\in\mathcal{Y}$  and $\Phi \lesssim J$ then there exists nonnegative $\vphi\in\Lphi$ for which \eqref{eq:she} has no nonnegative integral solution.
\ei
\item[(b)]   (Local and  Global well-posedness).
 Suppose    $m>\pF$ and  $1<q<n(m-1)/2$.  If $\Phi\in\mathcal{N}$, $\Phi (u)\gtrsim u^q$ as $u\to 0$ and  $\Phi \gtrsim J$, then
the conclusions of  Theorem\,B(a-c) and Corollary\,B (for some $\lam>0$) both  hold.
\ei
\end{theorem}


\begin{proof}
We verify the relevant hypotheses  of Theorem\,C and Corollary\,C. 
Firstly, note that since $J$ satisfies {\bf   (C$^\prime$)}, it is obvious that   the odd function $f$ satisfies  {\bf   (C)}  and  $\boldsymbol{(\Lambda_0)}$ (recall Remark\,\ref{rem:assC}(b)).  We now check that $f$ satisfies  {\bf (G)}.

Let $h :=(\log J)'= J'/J$. By {\bf (G$^\prime$)}, if $\rho\in (0,\infty)$  then  $h(u)=u^{\rho-1}\sigma(u)$ for some slowly varying function $\sigma$, while if  $\rho =\infty$  then $u^{-k}h(u) \to \infty$ as $u \to \infty$ for any $k>0$ (Lemma\,\ref{lem:LHop}(ii)). By {\bf (G$^\prime$)},  $h(u)^2J(u)/J''(u)=(J'(u))^2/(J''(u)J(u))\to 1$ as $u\to\infty$, so for either $\rho$ finite or $\rho =\infty$,
\begin{align*}
   \lim_{u\to\infty} \frac{(f'(u))^2}{f''(u)f(u)} &= \lim_{u\to\infty}\frac{(\frac{m}{u} + h(u))^2}{\frac{m(m-1)}{u^2} + \frac{2m}{u} h(u) + \frac{J''(u)}{J(u)}}\\
    &=
    \lim_{u\to\infty} \frac{(\frac{m}{uh(u)} + 1)^2}{\frac{m(m-1)}{u^2 h(u)^2} + \frac{2m}{uh(u)} + \frac{J''(u)}{J(u)h(u)^2}}=1,
\end{align*}  
so $f$ satisfies the limit condition in {\bf(G)}. 
Again by {\bf (G$^\prime$)},
\[ \lim_{u\to\infty} \frac{(\log f)'(\lam u)}{(\log f)'(u)} =
\lim_{u\to\infty} \frac{\frac{m}{\lam u} + h(\lam u)}{\frac{m}{u} + h(u)}=
\lim_{u\to\infty} \frac{ h(\lam u)}{ h(u)}
=\lim_{u\to\infty} \frac{(\log J)'(\lam u)}{(\log J)'(u)},
\]
so $f$ also satisfies the regular variation condition in {\bf (G)} with $\rho\in (0,\infty]$.

(a)(i): Suppose $\Phi\in\mathcal{N}$ and  $\Phi \gtrsim J$.  To apply     Theorem\,C(a)(i) we need to verify that $\Phi \gtrsim f$. It is therefore sufficient to show that
$J \gtrsim  f$.  If  $\rho\in (0,\infty)$ then  by Lemma\,\ref{lem:LHop}(i)    $J(u)=\exp (u^\rho \mu(u))$  for some slowly varying function $\mu$. For any $\lam >1$ we then have
\[
\frac{J(\lam u)}{u^mJ(u)}=u^{-m}\exp \left(u^\rho \mu(\lam u)\left(\lam^\rho -  \frac{\mu(u)}{\mu(\lam u)}\right)\right)
\sim  u^{-m}\exp \left(\left(\lam^\rho -  1\right)u^\rho \mu(\lam u)\right)\to\infty
\]
as $u\to\infty$. If $\rho =\infty$ then by Lemma\,\ref{lem:LHop}(i)-(ii)  (recalling \eqref{eq:rhoinf}) we have, for any $\lam >1$, 
\[
\lim_{u\to\infty}\frac{\log J(\lam u)}{\log (u^mJ(u))}=\lim_{u\to\infty}\frac{\log J(\lam u)}{m\log u +\log J(u)}=\lim_{u\to\infty}\frac{\log J(\lam u)}{\log J(u)}=\infty.
\]
Hence for $\rho\in (0,\infty]$, $J(u) \gtrsim u^mJ(u)=f(u)$.

The condition for uniqueness in the class of mild solutions is obvious, via   Theorem\,C(a)(ii) and \eqref{eq:phiinf} in Remark\,\ref{rem:thmC}.

(a)(ii): Suppose $\Phi\in\mathcal{\CY}$ and  $\Phi \lesssim J$.  To apply    Theorem\,C(a)(iii) we need to verify that $\Phi \lesssim f$. It is therefore sufficient to show that $J \lesssim f$.  The argument is identical to the one above for $J \gtrsim f$, simply interchanging $\lam>1$ for $\lam <1$.

(b): Since $f$ is $C^1$ and $J(0)>0$ it is clear that $f'$ is regularly varying at zero of index $m-1$.  By assumption $\Phi \gtrsim J$ and by the proof of part (a) above, $J\gtrsim f$. Hence $\Phi \gtrsim f$ and the result follows by Corollary\,C.
\end{proof}


\subsection{Nonlinearities of Arbitrarily Rapid Growth}\label{sec:rapidJ}
Starting from essentially any positive, increasing, convex function   $r$, we construct a function $J$ satisfying the hypotheses of  Theorem\,\ref{thm:J} and growing at least as fast as 
$\exp\left( r(u)\right) $.  The crucial point being that since $r$ can grow arbitrarily fast, so can $J$, and therefore  the nonlinearity $f(u)=u^mJ(u)$ of Theorem\,\ref{thm:J}.

\begin{lemma}\label{lem:regvarminor}
Let $q:(0,\infty)\to(0,\infty)$ be a continuous function.
There exists a function $p:(0,\infty)\to(0,\infty)$ such that
\bi[leftmargin=3em]
\item[(i)] $0 < p(x) \le q(x)$ for all $x>0$;
\item[(ii)] $p(x) \to 0$ as $x\to\infty$;
\item[(iii)] $p$ is rapidly varying of index $-\infty$, i.e.,  for all $\lam  >1$,
$\displaystyle{\lim_{x\to\infty} \frac{p(\lambda x)}{p(x)} = 0.}$
\ei
\end{lemma}

\begin{proof}
Define $\psi\:(0,\infty)\to\R$ by
\[
\psi(x) = -\log q(x), \qquad x>0.
\]
Since  $\psi$ is  continuous we may define  
\[
\phi(x) = \max_{1 \le y \le x} \psi(y), \qquad  x\ge 1.
\]
Clearly $\phi$ is  increasing on $[1,\infty)$ and {for all } $x\ge 1$,
\[
\phi(x) \ge \psi(x),\qquad \phi(x) \ge \psi (1)=:c.
\]
Define $p:(0,\infty)\to(0,\infty)$ by  
\[
p(x) =
\left\{
\begin{array}{ll}
\min\bigl\{ q(x), \e^{-1-\phi(1)} \bigr\}, & 0<x<1,\\
\e^{-\phi(x)-x}, &  x\ge 1.
\end{array}
\right.
\]
We now verify that the three conditions are met by this $p$.

(i) Clearly $p(x)  \le q(x)$ for $0<x<1$. For $x\ge 1$, we have $\phi(x) \ge \psi(x)$, so
\[
p(x) = \e^{-\phi(x)-x} \le \e^{-\psi(x)} = q(x).
\]

(ii) For $x\ge 1$ we have
\[
0 < p(x) = \e^{-\phi(x)-x} \le \e^{-c-x} \to  0\quad\text{as}\quad x\to\infty.
\]

(iii)  Let $\lambda >1$. Since $\phi$ is increasing, for $x\ge 1$ we have
\[
\frac{p(\lambda x)}{p(x)}
= \exp\bigl(-(\phi(\lambda x)-\phi(x) + (\lambda-1)x)\bigr)
\le \exp\bigl((1-\lambda)x\bigr)\to 0\quad\text{as}\quad x\to\infty.
\]
Hence $p$ is rapidly varying of  index $-\infty$.
\end{proof}

\begin{lemma}\label{lem:rapid}
Let $r\:[0,\infty)\to [0,\infty)$ be  any $C^2$  function satisfying  $r''>0$ on $(0,\infty)$,  $r(0)=0$, $r'(0)=1$ and $r'(\infty)=\infty$.  There exists a function  $J\: [0,\infty)\to [0,\infty)$ satisfying the hypotheses of Theorem\,\ref{thm:J}  such that $J(x)\ge \exp(r(x))$ for all $x\ge 0$.
\end{lemma}

\begin{proof}
For $x>0$ let   $q(x):=-(1/r'(x))'>0$ and  $p:(0,\infty)\to(0,\infty)$ be as in Lemma\,\ref{lem:regvarminor}. By the assumptions on $r$ we have
$\int_0^\infty q(x)\dee x=1$.  Hence by parts (i) and (ii)  of  Lemma\,\ref{lem:regvarminor} respectively we have $p\in L^1(0,\infty)$ and  $p(x)\to 0$ as $x\to\infty$.

Now set 
\be\label{eq:gJ}
g(x) = \left({\int_x^\infty p(s) \,\dee s}\right)^{-1},\qquad 
J(x) = \exp \left( \int_0^x g(s) \,\dee s \right),\qquad x\ge 0,
\ee
so that $g$ and $J$ satisfy the ODEs
\[ g'=p(x)g^2,\qquad g(0)= \left({\int_0^\infty p(s) \,\dee s}\right)^{-1},\]
and
\[J'=g(x)J, \qquad J(0)=1.
\]
Clearly  $g$ and $J$ are both positive and increasing and thus by its ODE, $J$ is convex. It is then easy to see that $J$  satisfies  {\bf  (C$^\prime$)}  of Theorem\,\ref{thm:J}. 

Now we check {\bf (G$^\prime$)}. Firstly, 
\[
\frac{(J')^2}{J''J} = \frac{(gJ)^2}{(g'J + gJ')J} =  \frac{1}{p(x) + 1}\to 1\quad\text{as}\quad x\to\infty. 
\]
Next,    $(\log J)'= J'/J=g$ and for any $\lam >1$ we have, by   l'H\^opital's rule and  Lemma\,\ref{lem:regvarminor},
\[ \lim_{x\to\infty}\frac{g(\lambda x)}{g(x)} 
=\lim_{x\to\infty} \frac{\int_x^\infty p(s) \,\dee s}{\int_{\lam x}^\infty p(s) \,\dee s}
=\lam^{-1}\lim_{x\to\infty}\frac{p(x)}{p(\lambda x)}=\infty.
\]
Thus  $(\log J)' $ is rapidly varying of index $\rho=\infty$ and $J$  satisfies  {\bf (G$^\prime$)}. 

Finally,  for all $x> 0$,
\[
\int_x^\infty p(s) \,\dee s\le\int_x^\infty q(s) \,\dee s
=\int_x^\infty -(1/r'(s))' \,\dee s=\frac{1}{r'(x)}
\]
and so
\[ J(x) = \exp\left(\int_0^{x} g(s) \dee s\right) \ge \exp\left(\int_{ 0}^{x} r'(s) \dee s\right)
=\exp\left( r(x) \right).\]
\end{proof}

\begin{eg}\label{eg:rpJ}
We illustrate how to construct an exemplar function $J$ satisfying the hypotheses of Theorem~\ref{thm:J}, via Lemma~\ref{lem:regvarminor}, Lemma~\ref{lem:rapid} and a suitable `seed' function $r$.  We take $r(x)= 
\e^x-1$, which satisfies the hypotheses of Lemma~\ref{lem:rapid} and use the procedure in the proof of Lemma~\ref{lem:rapid} to  construct $J$ using \eqref{eq:gJ}. Thus  we define  $q(x)=-(1/r'(x))'=\e^{-x}$ and seek a function $p$ satisfying the  
hypotheses of Lemma~\ref{lem:regvarminor}. Any such $p$ will suffice for constructing $J$ via \eqref{eq:gJ}, and we observe that    $p(x)=q(x)=\e^{-x}$ is just such a choice. (The algorithm in the proof of Lemma~\ref{lem:regvarminor} provides a mechanism for determining $p$ more generally without relying on observation; in this case this procedure yields $\psi(x)=x$, $\phi(x)=x$ with $p(x)=\e^{-2x}$ for $x\ge 1$ and $p(x)=\e^{-2}$ for $0<x<1$. One could then use this choice of $p$ to construct $g$ and $J$ 
according to \eqref{eq:gJ}.)  With $p(x)=\e^{-x}$, \eqref{eq:gJ} yields $g(x)=\e^{x}$ and $J(x)=\e^{\e^{x}-1}$. In this case we actually have  $J(x)=\e^{r(x)}$ since we were able to choose $p(x)=q(x)$. 

With  nonlinearity  $f(u)=|u|^{m-1}uJ(|u|)=|u|^{m-1}u\e^{\e^{|u|}-1}$  ($m\ge 1$), Theorem~\ref{thm:J} shows that the critical growth of $\Phi$ at infinity is given by $\Phi_c(u)=J(u)=\e^{\e^{u}-1}$ for $u>0$ large.
\end{eg}

\subsection{Exponential Nonlinearity (FP-type)}
 \label{sec:exp}
For $ m,p\ge 1$  consider 
\begin{equation}\label{eq:fexpp} 
u_t=\Lap  u+|u|^{m-1}u\e^{|u|^p}.
\end{equation}
With $J (u)= \exp (u^p)$ for $ u\ge 0$ 
and $f$ being the odd extension of $f(u)=u^mJ(u)$, we may apply Theorem\,\ref{thm:J}.  Conditions {\bf   (C$^\prime$)} and {\bf (G$^\prime$)} are easily verified, with $\rho =p$.
We therefore obtain  local well-posedness of $M^{\Phi}$-classical solutions for any       $\Phi\in\mathcal{N}$   satisfying  $\Phi \gtrsim \exp (u^p)$, with solutions being unique in $C([0,T],\Mphi )$ if  $\Phi(u) \gtrsim u^{d}$  as $u\to 0$ for some  $d> 1$. Nonexistence of a local nonnegative integral solution  pertains  for some $\vphi\in\Lphi$ if   $\Phi\in\mathcal{Y}$  and $\Phi \lesssim f$.
 By Theorem\,\ref{thm:J}(b), if $m>\pF$ and  $1<q<n(m-1)/2$  then we obtain both local well-posedness in $\Mphi$ and global well-posedness for small initial data 
 in $\Mphi$ with  $\Phi (u)=u^q\exp (u^p)$ with solutions satisfying the decay estimate
 \[ \|u(t;\vphi)\|_{\infty}\le \lam\Phinv (t^{-n/2})\sim  \lam t^{-\frac{n}{2q}}\quad \text{as}\ t\to\infty,\]
for small $\vphi\in\Mphi $.

So far as we can tell, problem \eqref{eq:fexpp}  represents the  only class of FP-type problems studied in Orlicz  spaces within the literature on nonlinear heat  equations \cite{FK, FKRT,I1, IRT1,IRT2, MT2, MT1,RT1} - see Section\,\ref{sec:background} around \eqref{eq:expLf}. Furthermore, previous studies consider only  the choice  $\Phi(u)=\e^{u^p}-1$ and $p>1$.

It has been  shown in \cite[Theorem\,1.2]{MT1}  that global weak solutions exist for small initial data in  $\exp L^p$ whenever $ m\ge p>1$ and $m\ge 1+2p/n$. This permits the critical case $p=n(m-1)/2$ provided $m\ge p>1$.  In all cases the solutions take on the initial data in a much weaker sense  than  here (i.e., weak${}^*$).
Thus when $ m\ge p>1$ and $m\ge 1+2p/n$ both hold, \cite{MT1} is stronger permitting the critical case. If either of the conditions $ m\ge p>1$ or $m\ge 1+2p/n$ fail then \cite{MT1} is not applicable but our  results still apply for  $m>\pF$, independently of $p$. 
   Our solutions also take on the initial data in a  strongly continuous sense as $t\to 0$.    Moreover the case  $p=1$ is permitted here, in contrast to other works.  This is because our  results allow us greater freedom in choosing  $\Phi$. If one works only in $\exp L^p$   then one must have $p>1$ in order that the defining Young's function $\Phi(u)=\exp(u^p)-1$ satisfy the decay condition  for an $N$-function near zero.

\subsection{Doubly Exponential  Nonlinearity (FP-type)}
\label{sec:doubleexp}

For $ m,p\ge 1$  consider the problem 
\[ u_t=\Lap  u+|u|^{m-1}u\e^{\e^{|u|^p}}.
\]
With $J (u)= \exp\exp (u^p)$ for $ u\ge 0$ and $f$ being the odd extension of $f(u)=u^mJ(u)$,  one may carry out the same checks as in Section\,\ref{sec:exp} and  draw the
similar  conclusions via Theorem\,\ref{thm:J}. Conditions {\bf   (C$^\prime$)} and {\bf (G$^\prime$)} are easily verified, with $\rho =\infty$.  Local well-posedness follows for 
$\Phi (u)\gtrsim \exp\exp (u^p)$ and global well-posedness with $\Phi (u)=u^q\exp\exp (u^p)$ if $m>\pF$ and  $1<q<n(m-1)/2$.

There do not appear to be any works in the literature for this kind of problem set in Orlicz {spaces}. Indeed, the fact that $\rho =\infty$ means that the special procedure described in Section\,\ref{sec:background}  for obtaining smoothing estimates for the heat semigroup between spaces like $L^q$ and $\exp L^q$, cannot be carried out.

Analogous  results can be obtained for nonlinearities of this type with $J$ being any finite number of compositions of the exponential function.

 \subsection{Power Law Nonlinearity (P-type)}

For  $p> 1$ consider the Fujita equation
\begin{equation}\label{eq:fup} 
u_t=\Lap  u+|u|^{p-1}u.
\end{equation}
In the context of Lebesgue spaces, where
 \[ \Phi (u)={u^q}/q,\qquad q\ge 1,
\]
the well-known results of \cite{BC,W80} ensure the well-posedness of  $L^q$-classical solutions as follows:  if $1<p<\pF$ then \eqref{eq:fup} is well-posed in $L^1 $; if $p=\pF$ then  \eqref{eq:fup} is well-posed in $L^q $ for any $q>1$; if $p>\pF$ then \eqref{eq:fup} is well-posed in $L^q $ for any $q\ge q_c:=n(p-1)/2>1$. 

Let us now compare the classical results of  \cite{BC,W80} with those from  Theorem\,B. For $p>\pF$, we  see  (recalling Remark~\ref{rem:I2I3}(c) so that $\ell(u)=u^{p-1}$) that \eqref{eq:IC} is satisfied   if and only if $q> q_c$, with  $q>1$ ensuring that $\Phi$ is an $N$-function. Thus Theorem\,B covers the subcritical range $q>q_c$  but not the critical case $q= q_c$ of \cite{BC,W80}.  In fact this is  unsurprising since the classical results for the critical case are obtained via bespoke methods using  special properties (e.g. homogeneity) of both $f$ and  $\Phi$ for Lebesgue spaces (see also \cite{RS}), whereas our result derives from very large classes of  $f$ and $\Phi$. 
For $p=\pF$ we obtain the same results as in \cite{BC,W80}.  For $1<p<\pF$, we would like to choose $q=1$, but $\Phi$ would  not then be an $N$-function. In fact (though we do not detail it here since our primary interest is in FP-type nonlinearities), we do not actually require $\Phi$  to be an $N$-function in this particular case;  in our general setting it  is indeed   a sufficient (but not necessary) condition to ensure that $S(t)$ is a $C_0$-semigroup on $M^{\Phi}$, as per Theorem\,A(a)(iii), and thus obtain the well-posedness results of Theorem\,B.  But in the case of Lebesgue spaces,  $M^{\Phi}=L^{\Phi}$  so we can  take $q=1$ and obtain the same results as   \cite{BC,W80}. We mention that for $q=1$ one may  consider more general nonlinearities than those of power law type, as in \cite{LRSV,LS20}.

In the critical Fujita case $p=\pF$, one may go beyond the Lebesgue spaces  of \cite{BC,W80}  and consider the Orlicz space $L^1\log^r L$ with $N$-function
\be\label{eq:appuplog}
 \Phi (u)=u\log^r (1+u),\qquad r>0.
\ee
Recalling  \eqref{eq:LEP0inf}, it is easy to check that \eqref{eq:IC} is satisfied   if and only if $r>n/2$, whence Theorem\,B is again applicable. This yields a sharper result than \cite{BC,W80}, where well-posedness in $L^q$ is  guaranteed only  for $q>1$.  In \cite[Theorem\,1.3]{M21},  $r=n/2$ is permitted for existence but the setting  is  an Orlicz \emph{class} (\cite[Section 8.7]{AF}),  defined for a different  $\Phi$ to that in  \eqref{eq:appuplog} and one which is not an $N$-function, and therefore not necessarily a  linear {space}.  For $r>n/2$ our uniqueness result is unconditional, whereas  \cite[Theorem\,1.5]{M21} imposes growth bounds on the solution as $t\to 0$, akin to that in \cite[Theorem\,4]{W80}. 

Finally, let us consider the question of global solutions  when $p>\pF$, with 
 \[ \Phi (u)=\max\{u^q,u^r\},\qquad 1<q<r.
\]
Since $\Phi\in{\mathcal{N}}\cap \Delta_2$, by Lemma\,\ref{lem:L0M}(iii) and \cite[Theorem\,12.1(a)]{Mal} it follows that $\Mphi=\Lphi=L^q \cap L^r  $.
Condition \eqref{eq:Iglobal} of Corollary\,B is then seen to hold provided that $q<q_c<r$. We note that the limiting case $q=r=q_c$ represents the  result of  \cite[Theorem\,3(b)]{W81} in  $L^{q_c}$, although it should be noted that only positive solutions were considered there.

\subsection{Log-Corrected Power Law Nonlinearity (P-type)}

For  $p> 1$ and $m>0$, consider the following logarithmically-corrected Fujita equation
\begin{equation}\label{eq:fuplog} 
u_t=\Lap  u+|u|^{p-1}u\log^m (1+|u|).
\end{equation}
The solvability  of this problem was considered in \cite{FHIL1} for positive Radon measure initial data. Again one may choose to consider this problem in   $L^q $ and utilise the results in \cite{BC,W80} since the nonlinearity in \eqref{eq:fuplog} is majorised at infinity by $|u|^{p'-1}u$, for any $p'>p$.  If $1<p<\pF$ then  we may choose $p'\in (p,\pF)$ and \eqref{eq:fuplog} is well-posed in $L^1$. (Alternatively, one may obtain well-posedness  in $L^1$  directly from the results in \cite{LS20}, without first having to majorise the nonlinearity.)  If $p\ge \pF$ then \eqref{eq:fuplog} is well-posed in $L^q$ for any $q\ge n(p'-1)/2>n(p-1)/2=:q_c$, again by  \cite{BC,W80}.

We  can obtain a sharper result by using  Theorem\,B with  
 \[\Phi (u)=u^{q_c}\log^r (1+u),\qquad r\ge 1.
\]
Recalling Remark\,\ref{rem:I2I3}(c-d) so that $\ell(y)=f'(y)$ for large enough $y$, we have for such $y$ that
\[
 \ell  (\lam y)\left[\Phi ( y) \right]^{-\pF} \Phi' ( y)  \le  C\lam^{p-1}  y^{-1}\log^m (1+\lam y)\log^{-2r/n} (1+y).  
 \]
Taking   $r>n(m+1)/2$ and $r\ge 1$  ensures that \eqref{eq:LEP0inf} holds and well-posedness follows in the space $M^{\Phi}$ by Theorem\,B. Moreover, since 
$\Phi\in\mathcal{N}$ also satisfies the $\Delta_2$-condition, well-posedness holds in the Orlicz space $L^{\Phi}$ by Remark\,\ref{rem:subspace}(a) (see also Lemma\,\ref{lem:L0M}(iii)).

\vspace{5mm}
\noindent{\bf Acknowledgements.}  RL and KH were supported by a  Daiwa Anglo-Japanese Foundation Award [grant number 14353/15194].  Part of this work was conducted while RL was visiting The Graduate School  of Mathematical Sciences at The University of Tokyo, where KH was then a Research Fellow. RL would like to thank
Prof. Kazuhiro Ishige  and the school for  valuable discussions and their kind hospitality. YF was supported in part by JSPS KAKENHI [grant number 23K03179].
\addcontentsline{toc}{section}{Acknowledgements}

 The authors would like to thank the referees for their careful reading of the paper and suggestions for enhancing its readability.

\end{document}